\documentclass{amsart}
\usepackage{amssymb}
\usepackage{amsmath}
\usepackage{amsthm}
\usepackage[alphabetic]{amsrefs}

\usepackage{color}
\usepackage{euscript}
\usepackage[linktoc=page,colorlinks,citecolor = green!50!black,linkcolor= blue,urlcolor = blue]{hyperref}
\usepackage{cleveref}

\usepackage{enumitem}
\usepackage{tikz-cd} 
\usepackage{tikz}
\usepackage{mathrsfs}
\usepackage[all]{xy}
\usepackage{verbatim}
\usepackage[margin=1in]{geometry}

\theoremstyle{plain}
\newtheorem{thm}{Theorem}[section]
\newtheorem{lem}[thm]{Lemma}

\newtheorem{conj}{Conjecture}
\newtheorem{cor}[thm]{Corollary}
\newtheorem{prop}[thm]{Proposition}

\newtheorem{assump}[thm]{Assumption}
\newtheorem*{thm*}{Theorem}
\newtheorem*{problem}{Problem}
\newtheorem*{question}{Question}
\newtheorem*{metatheorem}{Meta-theorem}

\theoremstyle{definition}
\newtheorem{defn}[thm]{Definition}
\newtheorem{exmp}[thm]{Example}
\newtheorem{construction}[thm]{Construction}

\theoremstyle{definition}
\newtheorem{note}[thm]{Note}
\newtheorem{rk}[thm]{Remark}

\newtheorem{fact}[thm]{Fact}

\newcommand{\bC}{{\mathbb C}}

\newcommand{\bG}{{\mathbb G}}

\newcommand{\bK}{{\mathbb K}}

\newcommand{\bN}{{\mathbb N}}

\newcommand{\bP}{{\mathbb P}}
\newcommand{\bQ}{{\mathbb Q}}
\newcommand{\bR}{{\mathbb R}}

\newcommand{\bZ}{{\mathbb Z}}

\newcommand{\cA}{{\mathcal A}}
\newcommand{\cB}{{\mathcal B}}
\newcommand{\cC}{{\mathcal C}}
\newcommand{\cD}{{\mathcal D}}
\newcommand{\cE}{{\mathcal E}}

\newcommand{\cH}{{\mathcal H}}
\newcommand{\cI}{{\mathcal I}}

\newcommand{\cM}{{\mathcal M}}
\newcommand{\cN}{{\mathcal N}}
\newcommand{\cO}{{\mathcal O}}

\newcommand{\cQ}{{\mathcal Q}}
\newcommand{\cR}{{\mathcal R}}
\newcommand{\cS}{{\mathcal S}}
\newcommand{\cT}{{\mathcal T}}

\newcommand{\cW}{{\mathcal W}}

\newcommand{\scrE}{\EuScript E}
\newcommand{\scrF}{\EuScript F}
\newcommand{\scrG}{\EuScript G}

\newcommand{\hocolim}{\operatornamewithlimits{hocolim}\limits}
\newcommand{\colim}{\operatornamewithlimits{colim}\limits}

\newcommand{\Z}{\mathbb{Z}}

\newcommand{\R}{\mathbb{R}}
\newcommand{\B}{\mathcal{B}}

\newcommand{\fc}{\mathfrak{c}}
\newcommand{\fT}{\mathfrak{T}}
\newcommand{\fD}{\mathfrak{D}}

\newcommand{\defi}{\begin{defn}}
\newcommand{\edefi}{\end{defn}}
\newcommand{\mc}{\mathcal}
\newcommand{\eq}{\begin{equation}}
\newcommand{\eeq}{\end{equation}}
\newcommand{\bu}{\bullet}
\newcommand{\pf}{\begin{proof}}
\newcommand{\epf}{\end{proof}}
\newcommand{\op}[1]{\operatorname{#1}}
\newcommand{\sub}{\subseteq}
\newcommand{\ov}{\overline}
\newcommand{\gr}{\operatorname{\gamma}}
\newcommand{\red}[1]{\textcolor{red}{#1}}
\newcommand{\tw}{\operatorname{Tw}}

\DeclareTextFontCommand{\textemph}{\em}
\DeclareTextFontCommand{\emph}{\bfseries}

\renewcommand{\l}{\lambda}
\renewcommand{\a}{\alpha}

\renewcommand{\d}{\partial}
\newcommand{\e}{\epsilon}

\newcommand{\g}{\gamma}

\newcommand{\s}{\sigma}

\renewcommand{\l}{\lambda}

\renewcommand{\O}{\Omega}
\newcommand{\Si}{\Sigma}

\newcommand{\fk}{\mathfrak}
\newcommand{\fstop}{\mathfrak{f}}

\newcommand{\lt}{\left}
\newcommand{\rt}{\right}
\newcommand{\tms}{\times}

\newcommand{\sheafhom}{\mathcal{H} \kern -.5pt \mathit{om}}

\numberwithin{equation}{section}
\numberwithin{figure}{section}
\title[Categorical filtrations via localization]{Categorical action filtrations via localization and the growth as a
	symplectic invariant}

\author[Laurent C\^ot\'e]{Laurent C\^ot\'e}
\address{Department of Mathematics, Harvard University, Cambridge, MA 02138, USA}
\email{lcote@math.harvard.edu}

\author[Yusuf Bar\i\c{s} Kartal]{Yusuf Bar\i\c{s} Kartal}
\address {School of Mathematics, University of Edinburgh, Edinburgh, UK}
\email {ykartal@ed.ac.uk}


\date{}

\begin{document}

\maketitle
\begin{abstract}
We develop a purely categorical theory of action filtrations and their associated growth invariants. When specialized to categories of geometric interest, such as the wrapped Fukaya category of a Weinstein manifold, and the bounded derived category of coherent sheaves on a smooth algebraic variety, our categorical action filtrations essentially recover previously studied filtrations of geometric origin.

Our approach is built around the notion of a smooth categorical compactification. We prove that a smooth categorical compactification induces well-defined growth invariants, which are moreover preserved under zig-zags of such compactifications. The technical heart of the paper is a method for computing these growth invariants in terms of the growth of certain colimits of (bi)modules. In practice, such colimits arise in both geometric settings of interest. 

The main applications are: (1) A ``quantitative" refinement of homological mirror symmetry, which relates the growth of the Reeb-length filtration on the symplectic geometry side with the growth of filtrations on the algebraic geometry side defined by the order of pole at infinity (often these can be expressed in terms of the dimension of the support of sheaves).
(2) A proof that the Reeb-length growth of symplectic cohomology and wrapped Floer cohomology on a Weinstein manifold are at most exponential. (3) Lower bounds for the entropy and polynomial entropy of certain natural endofunctors acting on Fukaya categories.

\end{abstract}

\tableofcontents

\section{Introduction}

\subsection{Filtrations arising from algebraic and symplectic geometry}
Let $U$ be a smooth algebraic variety over $\bC$. The ring of regular functions $\mc{O}(U)$ can be endowed with a filtration by the order of pole at infinity. More precisely, consider a smooth compactification $U\subset X$ such that $D:=X\setminus U$ is a divisor. Define a filtration by setting $F^p\cO(U)$ to be the set of regular functions with a pole along $D$ of order at most $p$. Any two such compactifications are related by a sequence of blow-ups and blow-downs with center in the complement of $U$, and using this, one can show that this filtration is well-defined.

With a little more work, one can extend this filtration to other vector bundles, and in fact, to the entire bounded derived category of coherent sheaves, i.e.\ one can construct a filtration on $RHom_U(\scrF,\scrF')$, for $\scrF,\scrF'\in D^bCoh(U)$.


Symplectic topology provides another source of filtered vector spaces. Here is an example that will be important to us. Given a Liouville manifold $(M, \l)$, one can define invariants involving the Reeb dynamics in the contact boundary at infinity. For instance, if $K,L\subset M$ are exact Lagrangians (cylindrical at infinity) then one can define the wrapped Floer cohomology $HW(K,L)$. This is the cohomology of a chain complex which is essentially generated by the (finitely many) intersection points of $K$ and $L$, and the Reeb chords at the ideal contact boundary of $M$ from $\partial_\infty K$ to $\partial_\infty L$. The length of the Reeb chords defines a filtration on $HW(K,L)$. Of course, there is an ambiguity, because this definition involves a choice of contact form. Nevertheless, the resulting filtration is also known to be well-defined ``up to scaling". 

Mirror symmetry links algebraic and symplectic geometry. It has long been speculated that the above filtrations are related under mirror symmetry. The following result puts this expectation on rigorous footing. 

\begingroup
\def\thethm{\ref*{corollary:hms-main-equiv}}
\begin{cor}
Suppose that $(M, \fk{f})$ and $(X, D)$ are homologically mirror pairs and set $U=X\setminus D$. 
Let $K, L \in \mc{W}(M)$ be objects and let $\scrF_K, \scrF_L$ be their image in $D^bCoh(U)$. Then the growth of the Reeb-length filtration on $HW(K,L)$ and the growth of the pole-order filtration on $RHom_U(\scrF_K, \scrF_L)$ asymptotically coincide.\footnote{We define the growth function of a filtered vector space below, and state a more precise version of \Cref{corollary:hms-main-equiv} in \Cref{sec:applications}.
}
\end{cor}
\addtocounter{thm}{-1}
\endgroup

The purpose of this paper is to develop a categorical theory of action filtrations which recovers the Reeb-length filtration and pole-order filtration when specialized to wrapped Fukaya categories and derived categories of coherent sheaves, respectively. \Cref{corollary:hms-main-equiv} is an immediate consequence of the existence of such a theory. Our approach is based on the notion a \emph{smooth categorical compactification}, which is due to Efimov \cite{efimov2013homotopy} and is recalled below. 


While our aims are mostly foundational, we also mention the following purely symplectic application. 
\begingroup
\def\thethm{\ref*{corollary:growth-most-exponential}}
\begin{cor}
If $M$ is a Liouville manifold (equipped with orientation/grading data as in \Cref{assumption:orientation-grading}), the Reeb-length growth of wrapped Floer cohomology \cite{mclean} is at most exponential. Similarly, the growth of symplectic cohomology \cites{seidel-biased-view, mclean-gafa} is at most exponential.
\end{cor}
\addtocounter{thm}{-1}
\endgroup


\Cref{corollary:growth-most-exponential} is an immediate consequence of the general theory developed in this paper, combined with the (easy) fact that the categorical entropy in the sense of Dimitrov--Haiden--Kontsevich--Katzarkov \cite{dhkk} is always finite. It can be shown that any contact manifold of dimension at least $3$ admits a non-degenerate contact form with the property that the number of Reeb orbits grows super-exponentially with respect to length. In such situations, one learns from \Cref{corollary:growth-most-exponential} that most orbits cancel cohomologically. Note that this also implies that there are plenty of holomorphic curves (e.g. the number of Floer trajectories grows super-exponentially in total boundary length).

In the remainder of the introduction, we explain the key features of our categorical approach to constructing action filtrations, as well as various difficulties which need to be overcome in order to implement it.


\subsection{Filtrations via categorical compactifications}
%
Given varieties $U\subset X$ as above, one has an induced categorical localization map $D^bCoh(X)\to D^bCoh(U)$ whose kernel is given by the sheaves whose hypercohomology groups are set theoretically supported on $D=X\setminus U$. This is an example of a smooth categorical compactification. More precisely, given a (homologically) smooth category $\cC$, a smooth categorical compactification is a smooth, proper category $\cB$ and a localization functor $\cB\to \cC$ whose kernel is split-generated by a finite set of objects; see e.g.\ \cite[Def.\ 1.7]{efimov2013homotopy}.

A natural way to obtain smooth categorical compactifications in symplectic geometry is the following: given a Weinstein manifold $M$, one can endow it with a Lefschetz fibration $M\to \bC$ and define the associated Fukaya-Seidel category $\cW(M,\fstop)$. There is a natural ``stop removal" functor $\cW(M,\fstop)\to\cW(M)$, and this data is an example of a smooth categorical compactification. In this case, the kernel of the functor is generated by finitely many Lagrangian discs. The stop $\fstop$ is an example of what we will call a \emph{full stop}. 

Given a smooth categorical compactification $\cB\to \cC$, the basic idea for defining a filtration on $\cC(K,L)$ (or rather on $H(\cC)(K,L)$) is as follows: fix lifts $\ov K,\ov L$ of $K,L$ to $\cB$. As $\cC$ is equivalent to the quotient of $\cB$ by the category generated by a finite set $\cD$, giving a (closed) morphism $f\in\cC(K,L)$ is essentially the same as giving a filtration 
\begin{equation}\label{eq:introreso}
	\xymatrix{ K'=K_p \ar[r]& K_{p-1}\ar[r]& K_{p-2}\ar[r]&\dots \ar[r]&  K_0=\ov K } 	
\end{equation}
such that each $K_i\to K_{i-1}$ has cone in $\cD$ (or is at least quasi-isomorphic to a shifted direct sum of objects of $\cD$), and a closed morphism $f'\in \cB(K',\ov L)$. Then define $F^pH(\cC(K,L))$ to be the set of morphisms for which such a resolution \eqref{eq:introreso} of length at most $p$ exists. 

For instance, if $U$ is a smooth affine curve, we can construct a smooth (geometric!) compactification $U \subset X$ by adding finitely many points. Let $\cD$ denote the set of skyscraper sheaves of these points and let $K=L=\cO_U$. Let $\ov K=\ov L=\cO_X$. A function $f\in\cO(U)=RHom_U(\cO_U,\cO_U)$ has a pole of order at most $p$ if and only if it extends to a section of $\cO_X(x_1+\dots x_p)$, for $x_i\in D=X\setminus U$, i.e. to a map $K'= \cO_X(-x_1-\dots -x_p)\to \cO_X$. Then a resolution is given by 
\begin{equation}
		\xymatrix{ K'=\cO_X(-x_1-\dots -x_p) \ar[r]& \cO_X(-x_2-\dots -x_p)\ar[r]& \dots \ar[r]&  \cO_X=\ov K } 	
\end{equation}
with cones given by the skyscraper sheaves $\cO_{x_i}$, and $f$ is represented by the holomorphic extension $f':K'\to \ov L$. In higher dimensions, one needs to use a finite collection of vector bundles supported on $D=X\setminus U$.

\begin{figure}
	\centering
	\includegraphics[height=2.5 cm]{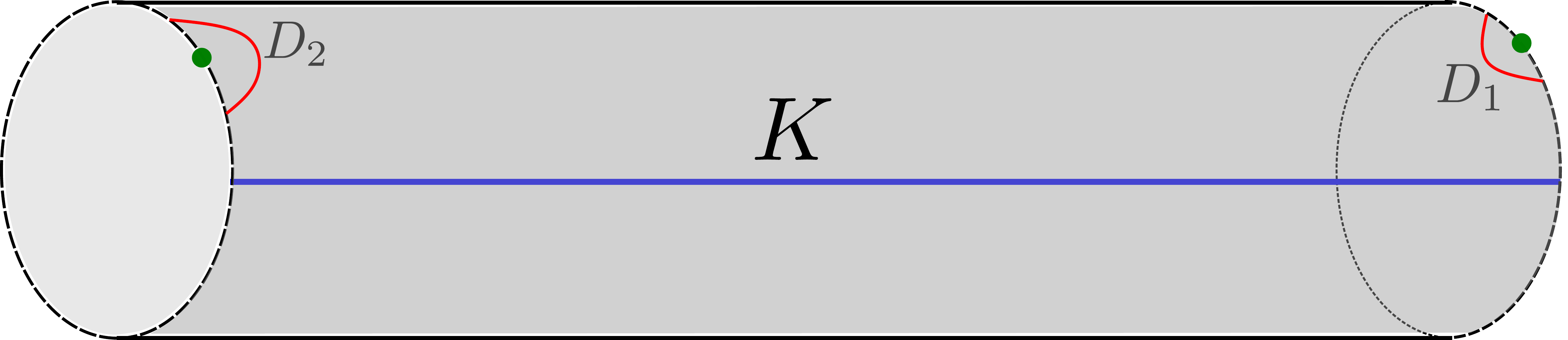}
	\caption{}
	\label{figure:cylinderwithtwostops}
\end{figure}

Similarly, let $M= T^*S^1$ be a cylinder. It turns out that the mirror to compactifying an algebraic variety is to add a ``stop".\footnote{A stop $\fstop$ in a Liouville manifold $M$ is an arbitrary closed subset of its ideal boundary $\d_\infty M$.} Let us choose a stop $\fstop$ which consists of one point in each boundary component. Each component of the stop has a ``Lagrangian linking disk" associated to it, well-defined up to isotopy in $\mc{W}(M)$. We denote these by $D_1, D_2$ and let $\cD= \{D_1, D_2\}$ (see \Cref{figure:cylinderwithtwostops}). Then the kernel of the natural map $\mc{W}(M, \fstop) \to \mc{W}(M)$ is generated by $\cD$. 

Let $K=L$ be a cotangent fiber (disjoint from the stop) and choose the lift $\ov K=\ov L\in \cW(M,\fstop)$ to be the same Lagrangian, considered as an object of $\cW(M,\fstop)$. Recall that an element of $HW(K, K)$ can be viewed as a linear combination of Reeb chords from $\d_\infty K$ to (a small perturbation of) itself. Consider the right boundary component. On this component there is a unique chord corresponding to every natural number $1,2,\dots $ (namely the one that travels the circle once, twice, and so on). Similarly on the other boundary component. However, these chords, considered as morphisms, do not lift to $\cW(M,\fstop)$, and $HW_{(M,\fstop)}(K,K)$ is one dimensional. To represent the shortest chord on the right hand side, one has to ``wrap'' the Lagrangian once past the stop. In $\cW(M,\fstop)$, this gives a surgery exact triangle $K'\to K\to D_1\to K'[1]$, and the corresponding morphism in $\cW(M)$ is represented by a morphism in $\cW(M,\fstop)$ from $K'$ to $K=L$. To obtain the chord that travels around the right component twice, one needs to wrap once more, obtaining a filtration as in \eqref{eq:introreso} of length $2$. To obtain the chords on the left side, one needs to use similar filtration, where cones are in $D_2$. Morally, we filter $HW(K, K)$ by declaring that an element $f \in F^p HW(K, L)$ if it can be represented by a linear combination of chords, each of which crosses $\fstop$ at most $p$ times. 

To connect this to the algebraic geometry picture via mirror symmetry, we note that the category $\cW(M)$ is derived equivalent to $D^bCoh(\bG_m)$, and the categorical compactification above is mirror to $D^bCoh(\bP^1)\to D^bCoh(\bG_m)$. Depending on how one sets up the equivalence, $D_1$ corresponds to $\cO_\infty$ and $D_2$ corresponds to $\cO_0$. Also, $K=L$ correspond to $\cO_{\bG_m}$, and the lifts $\ov K=\ov L$ correspond to $\cO_{\bP^1}$. The derived equivalence gives us $HW(K,K)\cong \bC[x,x^{-1}]$, where the chords on the left correspond to $x^i,i=-1,-2,\dots$ and chords on the right correspond to $x^i,i=1,2,\dots$. Thus, for instance, to obtain the chord corresponding to $x^{2}$, one wraps twice on the right boundary component of $M$, which correspond to taking $\cO_{\bP^1}(-2.\infty)$ on the B-side. One represents $x^2$ as a morphism  from $\cO_{\bP^1}(-2.\infty)$ to $\cO_{\bP^1}$, and uses the filtration $\cO_{\bP^1}(-2.\infty)\to \cO_{\bP^1}(-\infty)\to \cO_{\bP^1}$ to show $x^2$ is in $F^2HW(K,K)$.

Going back to the abstract setting where $\cC$ is a non-proper, smooth $A_\infty$-category over an algebraically closed field $\bK$ of characteristic $0$, and $\phi:\cB\to \cC$ is a smooth categorical compactification, we construct a chain level filtration. More precisely, we consider the Lyubashenko--Ovsienko model of quotient $A_\infty$-categories $\cB/\cD$ (where $\cD$ is a finite set of generators of $ker(\phi)$ as above), and this category is automatically filtered. More precisely, the hom-complexes of this $A_\infty$-category are similar to bar constructions, and we filter them by length (length$+1$ to be precise). Given $K,L\in\cC$, and lifts $\ov{K}$, $\ov{L}$, one has $\cB/\cD(\ov K,\ov L)\simeq \cC(K,L)$, and using this quasi-isomorphism, one obtains a filtration at the cohomology level $H(\cB/\cD)(\ov K,\ov L)\cong H(\cC)(K,L)$. 

\subsection{Growth functions}\label{subsection:growth-functions-intro}
We study filtered vector spaces through a specific invariant called the \emph{growth function}: given a filtered vector space $V$, let $\gr_V(p):=dim(F^pV)$. We write $\gr_{K,L}$ if $V=H(\cB/\cD)(\ov K,\ov L)\cong H(\cC)(K,L)$. We will refer to $\gr_{K,L}$ as the growth function associated to $K, L$.  As $\cB$ is proper, this function is finite; however, the growth function is non-canonical in the following ways:
\begin{enumerate}[topsep=0pt]
	\item\label{introitem:compactification} it depends on the choice of smooth categorical compactification $(\cB,\phi)$;
	\item\label{introitem:choiceoflifts} it depends on the choice of lifts of $K$, $L$;
	\item\label{introitem:choiceofgenerators} it depends on the choice of generators $\cD$ of the kernel of $\phi:\cB\to\cC$.
\end{enumerate}

In fact, the growth functions do depend on the above choices, but only up to a rather mild notion of equivalence.

The following definition will be crucial in this paper. Two weakly increasing functions $\gr_1, \gr_2:\bN\to \bR_{\geq 0}\cup \{\infty\}$ are called \emph{scaling equivalent} (or just equivalent), if there are constants $a\geq 1, b\geq 0$ such that $\gr_1(p)\leq a\gr_2(ap+b)+b$ and vice versa. If the factor $a$ can be taken to be $1$, we call them \emph{translation equivalent}. For instance, $p^{d_1}$ is scaling or translation equivalent to $p^{d_2}$ ($d_1,d_2\geq 0$) if and only if $d_1=d_2$. On the other hand, $e^{d_1 p}$ is always scaling equivalent to $e^{d_2 p}$ as long as $d_1,d_2>0$, but never translation equivalent unless $d_1=d_2$. The functions $p$, $e^{\sqrt{p}}$ and $e^p$ are pairwise inequivalent for both notions.

We now have the following theorem (a geometric instantiation of a purely categorical meta-theorem stated below). 



\begingroup
\def\thethm{\ref*{thm:welldefinedsymplecticgrowth}}
\begin{thm} Given a Weinstein manifold $M$ and a pair of objects $K, L \in \tw^\pi \mc{W}(M)$, the graded vector space $HW(K,L)=H(\cW(M))(K,L)$ admits a filtration such that the associated growth function $\gr_{K,L}:=\gr_{H(\cW(M))(K,L)}$ is well-defined up to scaling equivalence.	
\end{thm}
\addtocounter{thm}{-1}
\endgroup
The filtration in \Cref{thm:welldefinedsymplecticgrowth} is an increasing integral filtration which is constructed according to the categorical procedure described in the previous section. Therefore, contrary to filtrations and growth functions which are defined from the wrapped Floer homology, our construction allows us to relate them to growth in algebraic geometry. 


As explained above, a class of smooth categorical compactifications is given by considering the partially wrapped category $\cW(M,\fc)$ with the stop removal functor $\cW(M,\fc)\to \cW(M)$. Here, we assume $\fc$ is almost Legendrian, and its Legendrian locus has finitely many components. We also need $\cW(M,\fc)$ to be proper. To obtain such a stop $\fc$ (which we call a \emph{full stop}), one appeals to \cite{giroux-pardon} to endow $M$ with the structure of a Lefschetz fibration. Then, the fiber of the fibration (pushed to infinity) is a Weinstein stop, and its core $\fc$  is a full stop. 

To get around \eqref{introitem:compactification}, we have to relate different categorical compactifications constructed in this way. Given $\fc_1,\fc_2$, which (by a small perturbation) can be assumed not to intersect, there is a category which relates to both  $\cW(M,\fc_1)$ and $\cW(M,\fc_2)$: namely $\cW(M,\fc_1\cup \fc_2)$, through localization functors $\cW(M,\fc_1\cup \fc_2)\to\cW(M,\fc_i)$, $i=1,2$. Unfortunately, we cannot prove that $\cW(M,\fc_1\cup \fc_2)$ is itself proper; however, it turns out the properness of each $\cW(M,\fc_i)$ is sufficient. Namely, using this, one obtains an adjoint $\cW(M,\fc_i)\to \cW(M,\fc_1\cup\fc_2)$, and thus a semi-orthogonal decomposition of $\cW(M,\fc_1\cup \fc_2)$ with one component given by $\cW(M,\fc_i)$. This is the main ingredient for proving the independence of the choice of categorical compcatification. We emphasize that the properness of $\cW(M,\fc_i)$ is absolutely necessary for this argument to work.

Our intuition for why the growth functions should be independent of  \eqref{introitem:compactification} comes from algebraic geometry. As mentioned before, if $U$ is a smooth open variety, then any two compactifications of $U$ differ by a sequence of blowups and blowdowns with center disjoint from $U$ (this is called the ``weak factorization theorem"). Given a regular function on $U$, the order of pole at infinity is not affected by blowups/blowdowns.   This trick that we use to relate $\cW(M,\fc_1)$ and $\cW(M,\fc_2)$ can be seen as a symplectic/Fukaya categorical analogue of the weak factorization theorem. 

Also compare this to the following: by a theorem of Orlov, blowups of varieties along smooth centers give rise to semi-orthogonal decompositions. One can think of a semi orthogonal decomposition as a categorical version of a blowup. We cannot relate any two categorical compactifications by a zigzag of compactifications, but we can prove independence from the choice in \eqref{introitem:compactification} when the compactifications we start with happen to be related by such a zigzag.

Also note the following algebro-geometric counterpart of \Cref{thm:welldefinedsymplecticgrowth}:
\begingroup
\def\thethm{\ref*{prop:welldefinedalggeogrowth}}
\begin{prop}
Given a smooth algebraic variety $U$ over $\bK$, and $\scrF,\scrF'\in D^bCoh(U)$, the graded vector space $RHom_U(\scrF,\scrF')$ admits a filtration such that the associated growth function $\gr_{RHom_U(\scrF,\scrF')}$ is well-defined up to scaling equivalence.
\end{prop}
\addtocounter{thm}{-1}
\endgroup
As before, we denote the growth function for $\scrF,\scrF'$ by $\gr_{\scrF,\scrF'}$. The filtration in this case is defined using the categorical compactification defined by $D^bCoh(X)\to D^bCoh(U)$. The independence makes genuine use of Orlov's theorem mentioned above. 

The dependence of the growth function on the choices in \eqref{introitem:choiceoflifts} and \eqref{introitem:choiceofgenerators} can be resolved at a more abstract level. 
In this introductory section, we only explain the independence from \eqref{introitem:choiceofgenerators}. Let $\cD$ and $\cD'$ be two finite collections of generators of $ker(\phi)$, where $\phi:\cB\to \cC$ as before. 
In the $D^bCoh(\bG_m)$ example, where we consider the compactification $D^bCoh(\bP^1)\to D^bCoh(\bG_m)$, simplest $\cD$ is the set of skyscraper sheaves $\cO_0,\cO_\infty$, but for $\cD'$ one can also add the double points at $0$ and $\infty$. The point is that, in the general abstract setting, there is no canonical choice of $\cD$.

The moral idea is the following: since $\cD$ and $\cD'$ each generate the same set, one can express every element of $\cD'$ by taking shifts, direct sums/summands, and cones of objects of $\cD$ finitely many times. Moreover, as $\cD'$ is finite, there is a number $l\geq 0$ that bounds the number of times one must take cones. Roughly, every object of $\cD'$ can be expressed as a ``complex'' of objects of $\cD$ of length at most $l$. In other words, $\cD$ generates $\cD'$ in finite time. if one represents a morphism in $\cB/\cD'$ by a filtration as in \eqref{eq:introreso} of length $p$ and a morphism in $\cB(K',L)$ such that the cones of $K_i\to K_{i-1}$ are in $\cD'$, then one can refine this filtration with successive cones in $\cD$ but of length at most $pl$. In other words, under the identification $H(\cB/\cD)(K,L)\cong H(\cB/\cD')(K,L)$, one has $F^pH(\cB/\cD')(K,L)\subset F^{pl}H(\cB/\cD)(K,L)$. Similarly, $F^pH(\cB/\cD)(K,L)\subset F^{pl'}H(\cB/\cD')(K,L)$ for some $l'$. This implies the corresponding growth functions are scaling equivalent. Observe that for this argument, one does not actually need to assume $\cD$ or $\cD'$ are finite: it is sufficient that they are generated by a finite subset in finite time (i.e. by taking cones a bounded number of times).

As it happens, \Cref{thm:welldefinedsymplecticgrowth} and \Cref{prop:welldefinedalggeogrowth} are both special cases of a purely categorical result which we state as follows. 
\begin{metatheorem}\label{metatheorem:main-abstract-thm}
Let $\phi: \cB \to \cC$ be a smooth categorical compactification. Given a pair of objects $K, L \in \tw^\pi \cC$, the graded vector space $H(\cC)(K,L)$ admits a filtration such that the associated growth function $\gr_{K, L}:= \gr_{H(\cC)(K, L)}$ is well-defined up to scaling equivalence, and independent of $\cB$ up to zigzag of smooth categorical compactifications. 
\end{metatheorem}

This meta-theorem is not stated explicitly in the sequel, but it follows by combining various results proved in Sections \ref{sec:localizationfiltration} and \ref{sec:categoricalpseudocompact} (see in particular \Cref{corollary:growth-function-well-def}). 

\subsection{Computations via spherical functors}\label{subsection:intro-spherical}
It is not a priori obvious that the growth functions introduced in the previous section can be explicitly described or computed except in trivial cases. 

Arguably the technical heart of this paper is to develop a method for computing these growth invariants in terms of colimits of (bi)modules. This is absolutely crucial for all of our applications, in particular \Cref{corollary:hms-main-equiv}. The key steps are carried out in \Cref{subsection:spherical-filtration}, and rely on a healthy amount of filtered homological algebra developed in \Cref{sec:filteredalgebra}, much of which appears to be new.

Let us now summarize the key statements. There is a notion of a spherical functor $f: \cA \to \cB$, where $\cA$, $\cB$ are (for simplicity) pre-triangulated $A_\infty$ categories. Rather than stating the definition, we mention two key examples: the first is the map $j_*: D^bCoh(D) \to D^bCoh(X)$, where $X$ is a variety and $j:D\hookrightarrow X$ is a Cartier divisor. The second is the Orlov functor $\tw \mc{W}(\hat{\fstop}) \to \tw \mc{W}(M, \fstop)$, where $(M, \fstop)$, where $\fstop$ is a Lefschetz fiber (pushed to infinity).

A spherical functor $\cA\to\cB$ induces auto-equivalences of $\cA$ and $\cB$, respectively called the spherical cotwist and the spherical twist. Let $S: \cB \to \cB$ denote the spherical twist. It comes with a natural transformation $1 \to S$.  Now $S^k$ defines an $A_\infty$-bimodule via $\cS^k=\cB(\cdot, S^k(\cdot))$ and $s$ induces bimodule homomorphisms $\cB=\cS^0\to \cS^1\to \cS^2\to \dots$. Let $\cD$ denote the image $f(\cA)\subset \cB$. The key computational theorem of this paper is the following:
\begingroup
\def\thethm{\ref*{thm:spherical}}
\begin{thm}
$\cB/\cD$, considered as a filtered bimodule over $\cB$, is $E_1$-quasi-isomorphic to $\hocolim_k \cS^k$. \footnote{The filtration on the latter and the terminology are explained below.} 
\end{thm}
\addtocounter{thm}{-1}
\endgroup
Here $\hocolim_k$ denotes the homotopy colimit, for which we present a simple model. Instead of explaining every term in the statement, we state a key corollary:
\begingroup
\def\thethm{\ref*{cor:spherical}}
\begin{cor}
The complexes $\hocolim_k \cB(K,S^k(L))$ and $\cB/\cD(K,L)$ are $E_1$-quasi-isomorphic.	
\end{cor}
\addtocounter{thm}{-1}
\endgroup
We have already mentioned that $\cB/\cD(K,L)$ carries a natural filtration. The homotopy colimit is a complex such that its cohomology is the ordinary colimit of cohomologies $H(\cB(K,S^k(L)))$. The filtration on it induces the colimit filtration 
\begin{equation}
	F^p\colim_k H(\cB(K,S^k(L)))=\op{im}(H(\cB(K,S^p(L)))\to \colim_k H(\cB(K,S^k(L))))
\end{equation}
 
A morphism of chain complexes is said to be an \emph{$E_k$-quasi-isomorphism} if it induces a quasi-isomorphism on the $E_k$-page of the associated spectral sequences. Two bimodules are said to be \emph{$E_k$-equivalent}, or \emph{$E_k$-quasi-isomorphic}, if they are joined by a zigzag of morphisms of bimodules, each of which is an $E_k$ quasi-isomorphism on chain level. 

For us, the most important thing about $E_k$-quasi-isomorphisms is that they preserve growth functions. Thus, given objects $K, L \in \cB$, the growth functions $\gr_{K,L}=\gr_{H(\cB/\cD)(K, L)}$ and $\gr_{\colim_k H(\cB)(K, S^k L)}$ are the same. The upshot is that the filtration induced by taking the colimit of spherical twists matches the filtration induced by localization. 

\Cref{thm:spherical} can be used for computations of $\gr_{K,L}$; for instance, for the algebro-geometric example above, it implies
\begingroup
\def\thethm{\ref*{cor:growthsheaves}}
\begin{cor}
$\gr_{\scrF,\scrF'}(p)$ is equivalent to the function given by \eq p\mapsto dim(F^p\colim_n RHom_X(\scrF,\scrF'(nD)))\eeq 
\end{cor}
\addtocounter{thm}{-1}
\endgroup
%
We should also mention that \Cref{thm:spherical} is actually deduced from a slightly more general statement, namely \Cref{theorem:spherical-tensor-equiv}. As will be discussed later, this more general statement is needed to compare $\gr_{K,L}$ with growth functions arising from Hamiltonian dynamics.
%

We would also like to mention that in the case where $U$ is affine and $D=X\setminus U$ is an ample divisor, \Cref{thm:spherical} can be combined with \cite[\S81 Prop.6]{serrefaisceaux} to show:
\begingroup
\def\thethm{\ref*{thm:growthisdimsupp}}
\begin{thm}
Given $\scrF,\scrF'\in D^bCoh(U)$, the growth function $\gr_{\scrF,\scrF'}$ is a polynomial of degree $d=dim(supp(\scrF^\vee\otimes \scrF'))=dim(supp(\scrF)\cap supp(\scrF'))$. In other words, $\gr_{\scrF,\scrF'}(p)$ is scaling equivalent to $p^d$. 
\end{thm}
\addtocounter{thm}{-1}
\endgroup

\subsection{Relations to Hamiltonian filtrations}\label{subsection:intro-ham}
We now return to the ``Reeb length filtration" on wrapped Floer cohomology. The study of this filtration (and its closed string cousin) was pioneered by Seidel \cite{seidel-biased-view} and McLean \cites{mclean-gafa,mclean}. There are actually multiple essentially equivalent ways of setting up the theory, and for technical reasons, one usually avoids working directly with Reeb chords. The easiest setup is probably the following: given cylindrical Lagrangians $K, L$ in a Liouville manifold $(M, \l)$, and a (cylindrical at infinity) Hamiltonian $H$ with positive slope, recall that $\colim_n HF(\phi_{nH}K, L) \cong HW(K, L)$. Define the filtration on $HW(K, L)$ by  $ F^p HW(K, L):= im( HF(\phi_{pH}K, L) \to  HW(K, L))$. In other words, this is the colimit filtration on $HW(K, L)$ mentioned above.

We let $\gr_{K, L}^{ham}$ denote the resulting growth function. This growth function depends on $H$, but as we have already emphasized, it is independent of $H$ up to scaling. Nevertheless, this Hamiltonian action filtration leads to powerful applications: it is used most notably to prove that certain Weinstein manifolds are not affine varieties. It is also quite computable, at least for cotangent bundles: indeed, if $P$ is a (closed, connected) manifold and $F \sub T^*P$ is a cotangent fiber, then McLean showed \cite[Lem.\ 2.10]{mclean} (using work of Abbondandolo--Schwarz--Portaluri \cite{abbondandolo2008homology}) that the Hamiltonian filtration essentially matches the filtration $F^pH(\Omega P)$ induced by a choice of Riemannian metric. 

It is natural to ask how $\gr_{K, L}^{ham}$ is related to the categorical growth functions considered in the present work. In fact, we will prove that these growth functions are the same.

\begingroup
\def\thethm{\ref*{theorem:tensor-equiv-hamiltonian}}
\begin{thm}
With the notation as above, the growth functions $\gr_{K, L}^{ham}$ and $\gr_{K, L}$ are scaling equivalent.
\end{thm}
\addtocounter{thm}{-1}
\endgroup

%

To see why this should be the case, the reader may find it helpful to return to \Cref{figure:cylinderwithtwostops}: if we apply $\phi_{nH}$ to the Lagrangian $K$, then this has the the effect of wrapping $K$ in the positive direction. The number of times that $K$ hits the stop is proportional to $n$ (the proportionality constant depends on the slope of $H$). Thus, up to scaling, one expects to recover the same filtrations.  

To make this intuition rigorous in higher dimension is significantly more delicate. The argument involves multiple steps. First of all, we take advantage of the fact that we are free to choose a convenient stop $\fstop$ to define our growth functions. Thus, we choose as our stop the page of an open book on the ideal boundary of our Liouville manifold $M$. We can assume that $\d_\infty K, \d_\infty L$ are pairwise disjoint, and disjoint from the stop and from the binding. 

One is then tempted to let $H$ to be such that its Hamiltonian flow rotates the pages of the open book (such as a Hamiltonian pulled back from the base of Lefschetz fibration). Unfortunately, such a Hamiltonian is not cylindrical at infinity, and the comparison with cylindrical Hamiltonians is hard. Instead, we construct a cylindrical Hamiltonian whose flow approximately rotates one fixed page of the open book.  This suffices to show that the number of times $\d_\infty K$ passes the stop (i.e.\ the page of the open book) is proportional to the number of times one iterates the Hamiltonian. We obtain a sequence of Lagrangians $K=K_0\rightsquigarrow K_1\rightsquigarrow K_2\rightsquigarrow \dots$ which wrap in the positive direction, and where the wrapping $K^i\rightsquigarrow K^{i+1}$ passes the stop once (note this induces a sequence of maps $\dots \to K^2 \to K^1 \to K^0=K $).  


We now appeal to the key algebraic comparison results introduced in \Cref{subsection:intro-spherical}; in this case, the criterion we use is \Cref{theorem:spherical-tensor-equiv}. To apply this criterion, one needs to verify two conditions: the first is that the cone of $ K^{i+1}\to K^i$ is in the image of the Orlov functor $\tw \cW(\hat{\fstop}) \to \tw \cW(M, \fstop)$; the second is that the cone of $\cB(K^i, D) \to \cB(K^{i+1}, D)$ vanishes in cohomology whenever $D$ is in the image of the Orlov functor.

To verify these conditions, the wrapping exact triangle \cite[Thm.\ 1.10]{gps2} plays a key role. We also rely crucially on the ``stop doubling trick", which we learned from \cite[Sec.\ 7.3]{gps3} and \cite{sylvan-orlov}. This is the focus of \Cref{section:ham-comparison} of this paper.

\subsection*{Acknowledgments}
We would like to thank Sheel Ganatra for introducing us to smooth categorical compactifications, how to construct them in symplectic topology, and for many other helpful correspondences. We would also like to thank John Pardon for suggesting to us how to construct a Hamiltonian as in \Cref{section:ham-comparison}, Zack Sylvan for patiently explaining to us his work \cite{sylvan-orlov}, Roman Bezrukavnikov for pointing out to us the notion of Orlov spectra and generation time, Charles Weibel for helpful homological algebra conversations, Semon Rezchikov for helpful comments about the super-exponential growth in Hamiltonian dynamics, and Sasha Efimov for helpful conversations about relating general smooth categorical compactifications by zigzags. Finally, we wish to thank the anonymous referees for many helpful comments. 

LC was supported by the National Science Foundation under Grant No.\ DMS-1926686 (via the Institute for Advanced Study). Both authors are grateful to the Institute for Advanced Study for providing a  productive working environment.

\section{Background material}\label{sec:ainftybackground}

\subsection*{Conventions}
We adhere to the following conventions unless otherwise indicated.

All manifolds and maps between them are smooth. 

We consider chain complexes of vector spaces over a field $\bK$. Our chain complexes are $\Z$-graded and the differential raises the degree. 
Given a chain complex $C$, we define the \emph{shift} $C[1]^k= C^{k+1}$ with differential $d_{C[1]}= -d_C$.  

Given a morphism $f: C_1 \to C_2$ of chain complexes, the \emph{cone} is the chain complex 
\begin{equation} \op{cone}(f) = \lt(C_1[1] \oplus C_2, \begin{pmatrix} d_{C_1[1]} & 0 \\ f & d_{C_2} \end{pmatrix} \rt). 
\end{equation}

\subsection{$A_\infty$-categories and bimodules}
Throughout the paper $\bK$ will denote a field of characteristic $0$. For definitions of $A_\infty$-categories and modules, we refer the reader to \cite{seidel-book}.

Given an $A_\infty$ functor $F: \mc{B} \to \mc{B}'$, the kernel $\op{ker} F \sub \mc{B}$ is defined to be the largest full subcategory of $\mc{A}$ whose image under $H^*(F)$ is the zero subcategory of $H^*(\mc{B}')$.

\begin{defn}
Let $\cB$ be an $A_\infty$ category over $\bK$. Recall that \emph{a right $A_\infty$-module $\cN$ over $\cB$} is an assignment of a graded vector space $\cN(L)$ to each $L\in ob(\cB)$ and for every $k$ and for every $L_0,\dots, L_k$ a map 
\begin{equation}
\mu_\cN^{1|k}:	\cN(L_k)\otimes \cB(L_{k-1},L_k)\otimes \dots \cB(L_0,L_1)\to \cN (L_0)[1-k]
\end{equation}	
satisfying the standard $A_\infty$-equations. A \emph{left $A_\infty$-module} is defined similarly. Given $A_\infty$ categories $\cB$ and $\cB'$, an \emph{$A_\infty$-bimodule over $\cB'$-$\cB$} is an assignment of a graded vector space $\cM(L,L')$ to each $(L,L')\in ob(\cB)\times ob(\cB')$, and for every $L_0,\dots ,L_k\in ob(\cB)$, $L'_0,\dots ,L'_l\in ob(\cB')$ a map 
\begin{equation}
\mu_\cM^{l|1|k}:	\cB'(L_1',L_0')\otimes \dots \cB'(L_l',L_{l-1}')\otimes \cM(L_k,L_l')\otimes \cB(L_{k-1},L_k)\dots \cB(L_0,L_1)
\to \cM(L_0,L_0')[1-k-l]
\end{equation}
satisfying standard $A_\infty$-equations. Right, resp. left modules, and bimodules form dg categories, where the hom-complexes are given by pre-morphisms. For right modules $\cN_1$ and $\cN_2$, a \emph{pre-morphism} of degree $d$ is a collection of maps 
\begin{equation}
f^k:\cN_1(L_k)\otimes \cB(L_{k-1},L_k)\otimes \dots \cB(L_0,L_1)\to \cN_2 (L_0)[d-k]
\end{equation}
without any relation. See \cite{seidel-book} for the differential and composition of pre-morphisms. Closed pre-morphisms are $A_\infty$-module homomorphisms. The pre-morphisms, their differential and composition are defined similarly for left modules and bimodules. 
\end{defn}
\begin{rk}
One can see a right $A_\infty$-module, resp. a left $A_\infty$-module as a functor from $\cB$, resp. $\cB^{op}$ to the category of chain complexes over $\bK$. Similarly, a $\cB'$-$\cB$-bimodule is an $A_\infty$-bifunctor from $\cB'\times \cB$ to chain complexes.
\end{rk}
\begin{exmp}
Given $L \in \cB$, one can define a right module $h_L=\cB(\cdot, L)$ as the assignment $L_0\mapsto \cB(L_0,L)$, and with $A_\infty$-structure maps given by the $A_\infty$-products on $\cB$. The module $h_L$ is called \emph{the (right) Yoneda module associated to $L$}. Note that one has a natural cohomologically fully faithful functor from $\cB$ to the chains that sends $L$ to $h_L$. This functor is called \emph{the Yoneda embedding}.
\end{exmp}
The compact objects of $\op{Mod}\cB$ are called \emph{perfect modules}. Equivalently, $\cN\in \op{Mod}\cB$ is perfect if it can be expressed as a direct summand of a finite complex of Yoneda modules; see \cite{kellerdg}. The category of perfect modules is denoted by $\op{Perf} \cB$. A module $\cN$ is called \emph{proper} if $H^*(\cN(L))$ is finite dimensional for all $L\in ob(\cB)$. The category of proper modules is denoted by $\operatorname{Prop}(\cB)$. Similar definitions apply to left modules and bimodules. 

An $A_\infty$ category $\cB$ is said to be \emph{smooth} if its diagonal bimodule $\cB$ is perfect. 
Similarly, $\cB$ is called \emph{proper} if the diagonal bimodule $\cB$ is proper. More explicitly, this means that $\cB(K,L)$ has finite dimensional cohomology for every $K,L$.

Note that if $\cB$ is proper, then $\op{Perf} \cB \sub \op{Prop} \cB$; if $\cB$ is smooth then $\op{Prop} \cB \sub \op{Perf} \cB$ (see \cite[Lem.\ A.8]{gps3}).

Given an $A_\infty$-functor $f: \cB \to \cB'$, there is a \emph{pullback (restriction) functor} $f^*: \op{Mod} \cB' \to \op{Mod} \cB$. The pullback admits a left adjoint $f_!: \op{Mod} \cB \to \op{Mod} \cB'$ called \emph{induction} (which corresponds concretely to convolving with the graph bimodule). The induction functor sends a representable over $K \in \cB$ to a representable over $f(K) \in \cB'$, so it induces a functors from triangulated closure of Yoneda image of $\cB$ to that of $\cB'$, as well as from $\op{Perf} \cB$ to $\op{Perf} \cB'$. 
\begin{defn}
 An $A_\infty$ category $\cB$ is said to be \emph{pre-triangulated} if its image under the Yoneda functor $\cB\to Mod(\cB)$ is closed up to quasi-isomorphism under taking cones and translations. It is called \emph{idempotent complete} if it is closed under taking direct summands. Note that $\cB$ is pre-triangulated and idempotent complete if and only if the canonical embedding $\cB \hookrightarrow \op{Perf} \cB$ is a quasi-equivalence.  The homotopy category $H^0(\cB)$ of a pre-triangulated category is a triangulated category in the classical sense. 
\end{defn}

\begin{defn}
Let $\cB$ be an $A_\infty$ category and let $\cD$ be a full subcategory. We say $K\in ob(\cD)$ is \emph{generated} by $\cD$ if it can be represented as an iterated cone of objects of $\cD$. Similarly, we say $K\in ob(\cD)$ is \emph{split-generated} by $\cD$ if it can be represented as a direct summand of an iterated cone of objects of $\cD$. We say that $\cD$ \emph{generates}, resp. \emph{split-generates} $\cB$ if every object of $\cB$ is generated, resp. split-generated by $\cD$.
\end{defn}
Note that the latter is equivalent to natural functor $\op{Perf}\cD\to \op{Perf}\cB$ being a quasi-equivalence. The former is equivalent to the similar functor between pre-triangulated closures of Yoneda images being a quasi-equivalence. Next we define a natural pre-triangulated closure $\operatorname{Tw}\cB$, and the former can also be phrased as $\tw\cD\to \operatorname{Tw}\cB$ being a quasi-equivalence.

Any $A_\infty$-functor $f:\cB\to \cB'$ induces a $\cB$-$\cB'$-bimodule as explained above. Hence, one can extend the set of functors $\cB\to\cB'$ to $\cB$-$\cB'$-bimodules, obtaining the derived Morita category. Any bimodule $X$ induces a functor $(\cdot)\otimes_\cB X:Mod(\cB)\to Mod(\cB')$, and composition in this category is also given by convolution. Categories equivalent in this category are called \emph{Morita equivalent}, and the bimodule $X$ is called a \emph{Morita equivalence}. Morita equivalences identify $\op{Perf}(\cB)$ with $\op{Perf}(\cB')$. When the bimodule induced by $f:\cB\to\cB'$ is a Morita equivalence, we refer $f$ as a Morita equivalence. Concretely, this condition is equivalent to $f$ being fully faithful in cohomology and its essential image split-generating $\cB'$. Hence, we also refer to such $f$ as a \emph{derived equivalence}. See \cite{kellerdg} for more details on Morita category. 
\subsection{Twisted complexes and generation time}
Let $\cB$ be an $A_\infty$ category. We define the \emph{additive enhancement} $\Si \cB$ exactly as in \cite[(3k)]{seidel-book}, except that we only allow multiplicity spaces of dimension one for notational simplicity. Thus, the elements of $\Si \cB$ are formal finite sums $\oplus_{i=1}^m K_i [n_i]$, where $K_i \in \cB$ and $n_i \in \Z$. 

We define the category of twisted complexes $\tw \cB$ as in \cite[(3l)]{seidel-book}, again allowing only one-dimensional multiplicity spaces (cf.\ \cite[Rmk.\ 3.26]{seidel-book}). This category forms a pre-triangulated envelope of $\cB$, and it is canonically quasi-equivalent to the triangulated closure of the Yoneda embedding. In particular, its idempotent completion is quasi-equivalent to $\op{Perf} \cB$.

It will be useful to briefly describe the objects of $\tw \cB$, following \cite[(3l)]{seidel-book}. First of all, a \emph{pre-twisted complex} is a pair
\eq\label{equation:tw-complex} (E=\oplus_{i=1}^m K_i [n_i], \delta_E=(\delta_{ij})_{ij}) \eeq
where $\delta_{ij} \in \cB^1(K_i[n_i], K_j[n_j])$ ($\delta_E$ is called the \emph{differential}). 
%

A differential on a pre-twisted complex is said to be \emph{$l$-lower triangular} for some $l \in \mathbb{N}$ if there exists a filtration $0=E_0 \sub E_1 \sub \dots \sub E_{l-1} \sub E_l= E$ by sub-sums such that $\delta_E$ sends $E_i$ to $E_{i-1}$. A differential is said to be \emph{lower-triangular} if it is $l$-lower triangular for some $l \in \mathbb{N}$. 

Finally, a pre-twisted complex $(E, \delta_E)$ is a \emph{twisted complex} if $\delta_E$ is lower-triangular, and satisfies the Maurer--Cartan equation:
\eq\label{eq:mcequation} \sum \mu_\cB^n(\delta_E, \dots, \delta_E) = 0. \eeq
By the lower triangular condition, \eqref{eq:mcequation} is a finite sum. Twisted complexes form a pre-triangulated $A_\infty$ category, denoted by $\tw\cB$, and will also be referred as \emph{the twisted envelope of $\cB$}. An $A_\infty$-functor between two categories induce one between their twisted envelopes.

Note that the $l$-lower triangular condition basically means the twisted complex has length less than or equal to $l$. Equivalently, these are complexes that can be obtained by taking cones with direct sums of objects of $\cB$, $l-1$ times. Define, for every $l \in \mathbb{N}$, the full subcategory $\tw_{\leq l} \mc{B} \sub \tw \mc{B}$ of \emph{$l$-lower triangular twisted complexes}. The objects of $\tw_{\leq l} \cB$ are twisted complexes $(E, \delta_E)$ such that $\delta_E$ is $l$-lower triangular. Note that $\tw_{\leq l} \cB$ is not pre-triangulated, even though it is closed under translation and direct sums. The cone of a morphism from an object of $\tw_{\leq k} \cB$ to an object of $\tw_{\leq l} \cB$ lies in $\tw_{\leq k+l} \cB$. For instance, an object $L\in\cB$ is an element of $\tw_{\leq 1}\cB$, a twisted complex of the form $\{K\to L\}, K,L\in \cB$ is an element of $\tw_{\leq 2}\cB$, and a complex of the form $\{K\to L\to M\}, K,L,M\in\cB$ is an element of  $\tw_{\leq 3}\cB$.

Twisted complexes admit a natural split-closure, where the objects are triples $(E,\delta_E, \pi_E)$, where $(E,\delta_E)$ is a twisted complex, and $\pi_E$ is an (homotopy) idempotent of $(E,\delta_E)$. This triple represents a direct summand of $(E,\delta_E)$. Denote this split-closure by $\tw^\pi\cB$. Clearly, $\tw^\pi\cB \simeq \op{Perf}\cB$. One can define $\tw^\pi_{\leq l}\cB$ as the subcategory spanned by triples with $(E,\delta_E)\in \tw_{\leq l}\cB$. 
\begin{lem}
If $\mc{B} \to \mc{B}'$ is a functor of $A_\infty$ categories, then the induced functor $\tw \mc{B} \to \tw \mc{B}'$ sends the subcategory $\tw_{\leq l} \mc{B}$ to $\tw_{\leq l} \mc{B}'$.
\qed
\end{lem}
As mentioned, a full subcategory $\cD\subset \cB$ generates $\cB$ if and only if $\tw\cD\to\tw\cB$ is essentially surjective (equivalently a quasi-equivalence). 
\begin{defn}
We say that \emph{$\cD$ generates $\cB$ in time $l \in \mathbb{N}_+$} if every object of $\cB$ can be represented as a twisted complex in objects of $\cD$ of length at most $l$. More precisely, this means every $L\in ob(\cB)$ is quasi-isomorphic to an object of $\tw_{\leq l}(\cD)$. We say it \emph{split-generates in time $l$}, if every object of $\cB$ is a direct summand of an object of $\tw_{\leq l}(\cD)$. 
%
We say that $\cD$ \emph{generates (resp.\ split-generates) $\cB$ in finite time} if it generates (resp.\ split-generates) in time $l$ for some $l \in \mathbb{N}_+$.
\end{defn}
If $\cB$ is generated by a single object in finite time, this object is called \emph{a strong generator}. Existence of a strong generator may initially sound too strong; however, for instance by \cite[Theorem 3.1.4]{bondalvandenbergh}, we know that $D^bCoh(X)$ has a strong generator for any smooth variety $X$. Later we will show this is true for the wrapped Fukaya category of a Weinstein manifold too. 
\begin{lem}\label{lemma:finite-time-ess-surjective}
Let $F: \cB \to \cB'$ be an $A_\infty$ functor, and suppose the induced functor $\tw \mc{B} \to \tw \mc{B}'$ is essentially surjective. If $\mc{E}$ is essentially finite and generates $\tw\mc{B}$ in time $l$, then $F(\mc{E})$ generates $\tw\mc{B}'$ in time $l$. 
\qed
\end{lem}
\begin{lem}\label{lemma:exception-finite-time}
If $\mc{B}$ admits an exceptional collection $\mc{E} =\{E_1,\dots, E_l\}$ (meaning that $H(\cB)(E_i, E_i)=\mathbb{K}$ and $H(\cB)(E_i, E_j)=0$ if $j<i$) which is full (meaning that $\mc{E}$ generates), then $\mc{E}$ generates $\tw \mc{B}$ in time $l$.
\end{lem}
\pf
It is elementary to verify (cf.\ \cite[Lem.\ 3.1]{bondal-rep-algebras}) that any $T \in H^0(\tw \mc{C})$ admits an (essentially unique) decomposition of the form
\eq
0 = T_l \to T_{l-1} \to \dots \to T_1 \to T_0=T,
\eeq
where $\op{cone}(T_i \to T_{i-1})$ is a finite direct sum of shifts of $E_i$. This implies the desired statement. 
\epf

\subsection{Quotients}

Let $\mc{B}$ be a small $A_\infty$ category. Given a full subcategory $\cD \sub \cB$, one can form the \emph{quotient category} $\cB / \cD$, which comes equipped with a canonical functor 
\eq q: \cB \to \cB / \cD. \eeq
There are various models of the quotient category in the literature (see \cite{lyubashenko-manzyuk, lyubashenko-ovsienko} in the $A_\infty$ case and \cite{drinfeld} for the dg case). It will be useful in the sequel to perform certain constructions using the model \cite{lyubashenko-ovsienko}, which is discussed further in \Cref{subsection:tensor-length-filtration}.

The following lemma collects some general properties of quotient categories (valid for any of the above models). 

\begin{lem}\label{lemma:quotient-properties}
Let $\mc{B}$ be a small $A_\infty$ category and let $\cD \sub \cB$ be a full subcategory. 
\begin{enumerate}
\item\label{item:univ-prop-quotient} A functor $F: \mc{B} \to \mc{B}'$ which sends $\cD$ to an acyclic subcategory factors canonically through the quotient. More precisely, precomposition with $q: \cB \to \cB/\cD$ induces a fully faithful embedding of the $A_\infty$ category of functors $\op{Fun}(\cB/\cD, \cB') \hookrightarrow \op{Fun}(\cB, \cB')$ whose image consists precisely of those functors which annihilate $\cD$; 
\item\label{item:split-closure-quotient} if $\mc{D} \sub \mc{D}' \sub \mc{B}$ and $\mc{D}'$ is split-generated by $\mc{D}$, then the canonical map $\mc{B}/\mc{D} \to \mc{B}/\mc{D}'$ is a quasi-equivalence;
\item\label{item:quotient-can-map} Quotients of pre-triangulated categories remain pre-triangulated, but quotients of idempotent complete categories may fail to be idempotent complete. The canonical map $\tw\mc{B}/\mc{D} \to \tw(\mc{B}/\mc{D})$ is a quasi-equivalence; the canonical map $\op{Perf} \mc{B}/ \mc{D} \to \op{Perf} (\mc{B}/\mc{D})$ is cohomologically fully faithful and its image has split-closure $\op{Perf} (\mc{B}/\mc{D})$;\footnote{One instance of this phenomenon is the singularity category $D^bCoh(X)/D^{perf}(X)$ which is not in general idempotent complete (even though $D^bCoh(X)$ is idempotent complete); see \cite[Remark 2.4]{pavic2021k} for further discussion.}
\item\label{item:properties-pres-qcat} if $F: \cB \to \cB'$ is cohomologically fully-faithful (resp.\ a quasi-equivalence, Morita equivalence), then $\mc{B}/\cD \to \cB'/ F(\mc{D})$ is also cohomologically fully-faithful (resp.\ a quasi-equivalence, Morita equivalence);
\item\label{item:split-generation-quotient} if $\mc{B}/\mc{D}$ is the zero category, then $\mc{D}$ split-generates $\mc{B}$. More generally, if $K \in \mc{B}$ is sent to an acyclic object by the quotient map $\mc{B} \to \mc{B}/\mc{D}$, then $K$ is split-generated by $\mc{D}$.
\end{enumerate}
\end{lem}

%

We say that a map of $A_\infty$ categories 
\eq F: \cB \to \cC \eeq
is a \emph{localization} iff the induced map $\cB/ker F \to \cC$ is a Morita equivalence.

\begin{cor}\label{cor:quotientstronggenerator}
	If $\cD\subset \cB$	and $\tw\cB$ admits a strong generator, then so does $\tw(\cB/\cD)$.
\end{cor}
\begin{proof}
	By \Cref{lemma:quotient-properties}\eqref{item:quotient-can-map}, $\tw\cB\to \tw(\cB/\cD)$ is essentially surjective; hence, \Cref{lemma:finite-time-ess-surjective} applies. 
\end{proof}
\begin{exmp}\label{example:finite-time-weinstein}
	Let $(M, \l)$ be a Weinstein manifold. According to work of Giroux and Pardon \cite{giroux-pardon}, $(M, \l)$ admits a Lefschetz fibration with Weinstein fiber. Let $\fstop\subset \partial_\infty M$ denote the corresponding Weinstein stop, and let $\mc{W}(M, \fstop )$ be the associated Fukaya--Seidel category. We can choose a basis of thimbles $E_1,\dots, E_k$ that form a full exceptional collection for $\cW(M,\fstop)$ (cf.\ \cite[Ex.\ 1.4]{gps1}). Also the stop-removal functor $\cW(M,\fstop) \to \cW(M)$ is a localization (cf.\ \cite[Ex.\ 1.23]{gps2}). Hence it follows by combining \Cref{cor:quotientstronggenerator} and \Cref{lemma:exception-finite-time} that $\op{Tw}\mc{W}(M)$ is generated in finite time by some finite collection of objects. 
\end{exmp}

\subsection{Spherical functors}
In this section, we briefly recall the notion of spherical functors. Let $f: \cA\to\cB$ be a functor of pre-triangulated $A_\infty$ categories with left, resp. right adjoints $l$, resp. $r$. Given this data, one has natural transformations $fr\to id_\cB$, $id_\cB\to fl$, $id_\cA\to rf$, and $lf\to id_\cA$, and by taking their cones one obtains endo-functors of $\cA$ and $\cB$. The functor $f$ is called \emph{spherical} if 
\begin{itemize}
	\item $cone(fr\to id_\cB)$ and $cone(id_\cB\to fl)[-1]$ are quasi-equivalences and inverse to each other
	\item $cone(id_\cA\to rf)[-1]$ and $cone(lf\to id_\cA)$ are quasi-equivalences and inverse to each other
\end{itemize}
See \cite[Def.\ 1.1]{annologvinenkospherical} or \cite[Def.\ 2.1]{segalallspherical} for the precise definition. (Strictly speaking, one further requires $r\simeq l\circ cone(fr\to id_\cB)[-1]$ and $r\simeq cone(id_\cA\to rf)\circ l$, but this is unnecessary due to \cite[Thm.\ 1.1]{annologvinenkospherical}). The functor $S=cone(fr\to id_\cB)$ is called \emph{the twist functor}, whereas $(id_\cA\to rf)[-1]$ is called \emph{the cotwist}.
\begin{exmp}\cite{seidelthomas}
Let $\cB$ be as above and $E$ be a spherical object. In other words, $H^*(\cB(E,E))\cong H^*(S^d)$, and the Serre functor acts on $E$ by shift by $d$ (\cite{seidelthomas} gives a slightly more general definition). Then, $f:D^b(\bK)\to\cB$ such that $f(\bK)=E$ is a spherical functor. In this case, a right adjoint is given by $r(L)=\cB(E,L)$ and the twist functor is the Seidel-Thomas spherical twist. In particular, it fits into $\cB(E,L)\otimes E\to L\to S(L)\to \cB(E,L)[1]$.
\end{exmp}

\begin{exmp}[Exmpl.\ 3.5 in \cite{segalallspherical}]\label{exmp:sphericaldivisor}
Let $X$ be a variety and $j:D\hookrightarrow X$ be an effective Cartier divisor, i.e. $\cO(D)$ is a line bundle with a section $\sigma$ that cuts out $D$. Then, $j_*:D^bCoh(D)\to D^bCoh(X)$ is a spherical functor. In this example, the inverse of the twist is easier to describe: $j_*$ has a left adjoint $j^*$, which is equivalent to tensoring with $cone(\sigma)=cone(\cO(-D)\to\cO)\simeq \cO_D$. It is easy to see that the inverse twist is tensoring by $\cO(-D)$ and $S$ is tensoring by $\cO(D)$.\footnote{Since \cite{segalallspherical} does not provide a proof, here is a sketch: in view of \cite[Thm.\ 1.1]{annologvinenkospherical}, it suffices to verify (i) the existence of a left adjoint $j^*$ and right adjoint $j^!$, (ii) that the twist and cotwist are quasi-equivalences. For (i), a sufficient condition is for $j$ to be proper and perfect (see \cite[Sec.\ 2.1, 2.2]{anno2012adjunctions}), which is automatic here. For (ii), 
\cite[Thm 3.1, equations (26)-(27)]{bodzenta_bondal_2022} proves that the (inverse) twist, resp. cotwist are given by tensoring with the ideal sheaf of $D$, resp. its pullback. These are auto-equivalences as $D$ is Cartier.}
\end{exmp}

\begin{exmp}[\cites{sylvan-orlov, abouzaidganatra}]\label{example:symplectic-spherical}
Let $M$	be a Weinstein manifold, with a Lefschetz fibration structure, and let $\fstop \subset \partial_\infty X$ be the stop corresponding to the fibration. Then there is a functor $\cW(\hat{\fstop})\to \cW(M)$ (often called the \textemph{Orlov functor}) that was shown by Sylvan to be spherical. The corresponding twist is the ``\textemph{wrap once negatively functor}'' and the corresponding cotwist is the (inverse) monodromy of the fibration. (In fact, these statements hold under the more general assumption that $\fstop$ is a \textemph{swappable stop}). 

\end{exmp}
\begin{rk}
One does not have to assume $\cA$ and $\cB$ are pre-triangulated, as long as they have enough objects so that the corresponding twist and cotwist functors are defined. Nevertheless, we will have pre-triangulated assumption for simplicity.	
\end{rk}
$S$ fits into an exact triangle
\begin{equation}
    id_\cB\to S\to fr[1]\to id_\cB[1]
\end{equation}
In particular, there is a natural transformation $s:id_\cB\to S$.

Let $\cD$ denote $f(\cA)\subset \cB$.
Observe
\begin{lem}\label{lem:sdvanish}
If $D\in \cD=f(\cA)$, then $[s_D]\in Hom(D,S(D))$ is $0$, i.e. $s_D\in \cB (D,S(D))$ vanishes in cohomology.
\end{lem}
\begin{proof}
Let $D=f(L)$. As $s_{f(L)}$ is induced by $f(L)\xrightarrow{1_{f(L)}}f(L)\subset cone (frf(L)\to f(L))$,  in order to show that the induced map $f(L)\xrightarrow{s_{f(L)}} cone (frf(L)\to f(L))$ vanishes in cohomology, we only have to show it lifts to a map $f(L)\to frf(L)$, i.e. 
\begin{equation}
	\xymatrix{ & f(L)\ar[d]^{1_{f(L)}}\ar@{-->}[ld]_{\exists u} \\ frf(L)\ar[r]& f(L) }
\end{equation}
On the other hand, since $r$ is right adjoint to $f$, we have a natural transformation $s':id_\cA\to rf$. It is easy to check that $u=f(s'_L)$ is the desired lift. Indeed, by definition of adjoint functors (via the counit-unit adjunction), the composition $f\to frf\to f$ is the identity. 
\end{proof}
Observe $s:id_\cB\to S$ induces natural transformations $S^k\to S^{k+1}$, by composing with $S^k$ on the right, i.e. we consider the maps $s_{S^k(L)}: S^k(L)\to S^{k+1}(L)$. 
\begin{rk}
This is not a priori the same thing as what one obtains by composing on the left, i.e. $S^k(s_L)$. 
\end{rk}

Observe that $cone(S^k(L)\to S^{k+1}(L))$ has image in $\cD= f(\cA)$. More precisely, given $L$, $cone(s_{S^k(L)})$ is quasi-isomorphic to $fr(S^k(L))[1]$. 

First, we show 
\begin{lem}\label{lem:discimagevanish}
Given $D\in \cD$, $L\in\cB$, the map $\cB(D,S^k(L))\to \cB(D,S^{k+1}(L))$ vanishes in cohomology.
\end{lem}
\begin{proof}
Consider the diagram
\begin{equation}\label{eq:sqrdiagramcompchange}
	\xymatrixcolsep{7pc}
	\xymatrix{ \cB(D,S^k(L))\ar[r]^{\mu^2(s_{S^k(L)},\cdot)}\ar[d]_{S^{-k}}& \cB(D,S^{k+1}(L))\\ \cB(S^{-k}(D),L)\ar[r]^{\mu^2(\cdot,S^{-k-1}(s_D))}& \cB(S^{-k-1}(D),L)\ar[u]^{S^{k+1}} }
\end{equation}
that commutes in cohomology. Indeed, given a closed $a\in \cB(D,S^k(L))$, if we apply the triple composition along the bottom part, we find an element cohomologous to $\mu^2(S(a), s_D)\in \cB(D,S^{k+1}(L))$. This element is cohomologous to $\mu^2(s_{S^k(L)}, a)$, as follows from the fact that $s$ is a natural transformation; therefore, (\ref{eq:sqrdiagramcompchange}) is commutative in cohomology. 

By Lemma \ref{lem:sdvanish}, $S^{-k-1}(s_D)$ is null-homologous; hence, the bottom horizontal arrow in (\ref{eq:sqrdiagramcompchange}) induces trivial map in cohomology. Therefore, 
\begin{equation}
	\cB(D,S^k(L))\xrightarrow{\mu^2(s_{S^k(L)},\cdot)} \cB(D,S^{k+1}(L))
\end{equation}
also vanishes in cohomology. This is just rephrasing what the lemma asserts.
\end{proof}

\subsection{Background in symplectic topology}\label{sec:symplecticbackground}
Our main source for this \namecref{sec:symplecticbackground} is \cite{gps1, gps2}. When discussing purely geometric notions (such as Liouville manifolds, Liouville sectors, stops, etc.) we follow the conventions of \textit{loc}.\ \textit{cit}.\ unless otherwise indicated (in contrast, as indicated previously, our conventions for $A_\infty$ categories mostly follow \cite{seidel-book}). We limit ourselves to some brief reminders for the purpose of fixing definitions. 

\subsubsection{Basic notions}
A \emph{Liouville manifold} $(M, \l)$ is an exact symplectic manifold which is modeled near infinity on the symplectization of its ideal contact boundary $(\d_\infty M, \xi_\infty)$. We say that $(M, \l)$ is \emph{Weinstein} if, after possibly replacing $\l$ with $\l + df$ for $f: X \to \R$ compactly supported, the Liouville vector field is gradient-like with respect to a proper Morse function. (One also says that $M$ is Weinstein up to deformation, but we will suppress this distinction in the sequel.)

A closed subset $\fk{c} \sub \d_\infty M$ is called a \emph{stop}. A pair $(M, \fk{c})$, where $M$ is a Liouville manifold and $\fk{c} \sub \d_\infty M$ is a stop, is called a \emph{stopped Liouville manifold}. If $\fk{c} \sub \d_\infty M$ is a Liouville domain, this is often called a \emph{Liouville pair}. 

There is also a notion of a Liouville sector, which is an exact symplectic manifold with boundary modeled at infinity on the positive symplectization of a contact manifold with convex boundary. Similarly, one can consider stopped Liouville sectors, etc.\ 

A closed subset $\fk{c}$ of a symplectic manifold of dimension $2n$ is called \emph{mostly Lagrangian} if it admits a decomposition $\fk{c}= \fk{c}^{\op{crit}} \cup \fk{c}^{\op{subcrit}}$ where $\fk{c}^{\op{crit}}$ is Lagrangian and $\fk{c}^{\op{subcrit}}$ is closed and contained in the smooth image of a second countable manifold of dimension $\leq n-1$. There is an analogous notion of a \emph{mostly Legendrian} subset of a contact manifold.

We say that a mostly Lagrangian (resp.\ Legendrian) stop is \emph{tame} if there exists a decomposition $\fk{c}= \fk{c}^{\op{crit}} \cup \fk{c}^{\op{subcrit}}$ with the additional property that $\fk{c}^{\op{crit}}$ has finitely many connected components. Observe that the union of two \textit{disjoint} tame stops is a tame stop.

\begin{exmp}\label{example:tame-construction}
Let $V \sub \d_\infty M$ be a Weinstein hypersurface. After possibly perturbing $V$, we can assume that the cocores of the critical handles are properly embedded. Let $\fk{c}^{\op{subcrit}} \cup \fk{c}^{\op{crit}}$ be the union of the subcritical and critical handles respectively. Then $\fk{c}^{\op{subcrit}}$ is closed and $\fk{c}^{\op{crit}}$ has finitely many components. Hence $\fk{c} \sub \d_\infty M$ is a tame mostly Legendrian stop.
\end{exmp}


\subsubsection{Wrapped Fukaya categories}\label{subsubsection:fukaya-conventions}


The theory of wrapped Fukaya categories for (possibly stopped) Liouville manifolds/sectors was developed by Ganatra--Pardon--Shende \cite{gps1,gps2}, following earlier work of Sylvan \cite{sylvan-wrapped}. We collect some structural properties which will be needed in the sequel. The only deviation from \cite{gps1,gps2} is that we follow Seidel's composition conventions, namely $\mu^2$ is defined as a map
\begin{equation}
{\mc{W}(M, \fk{c})}(L_1,L_2)\otimes {\mc{W}(M, \fk{c})}(L_0,L_1)	\to {\mc{W}(M, \fk{c})}(L_0,L_2)
\end{equation} 
and so on. 
In order for the wrapped Fukaya category to be $\mathbb{Z}$-graded and defined over a field of characteristic zero (which we assume throughout this paper), we must fix \emph{grading and orientation data}. Such data also allows one to define symplectic cohomology in characteristic zero and to endow it with a $\mathbb{Z}$-grading. We refer to \cite[Sec.\ 5.3]{gps3} for a comprehensive summary. Be warned that grading and orientation data need not exist, and need not be unique when it exists. For example, the existence of grading data on $M$ requires that $2c_1(M)=0$. 

\begin{assump}\label{assumption:orientation-grading}
Whenever we consider the wrapped Fukaya category of a Liouville manifold/pair in this paper, we assume that the ambient manifold and the Lagrangians are equipped with grading and orientation data in the sense of \cite[Sec.\ 5.3]{gps2}. The same assumption also applies when considering symplectic cohomology of a Liouville manifold.
\end{assump}


If $\fk{c} \sub \fk{c}'$ is an inclusion of stops, there is an induced functor $\mc{W}(M, \fk{c}') \to \mc{W}(M, \fk{c})$ called \emph{stop removal}. More generally, given a finite diagram of stops with arrows corresponding to inclusions of stops, there is an induced diagram of partially wrapped Fukaya categories. 

\begin{fact}[Thm.\ 1.20 in \cite{gps2}]\label{fact:stop-removal}
If $\fk{c}' \setminus \fk{c} \sub \d_\infty M$ is mostly Legendrian, then the stop removal functor induces a quasi-equivalence 
\begin{equation}
    \mc{W}(M, \fk{c}')/ \mc{D} \to \mc{W}(M, \fk{c}),
\end{equation}
where $\mc{D}$ is the full subcategory of linking disks of $(\fk{c}' \setminus \fk{c})^{\op{crit}}$. 
\end{fact}

The wrapped Fukaya category is invariant under certain deformations of stops. For our purposes, we will only need the following easy case which follows e.g.\ from \cite[Prop.\ 2.10]{sylvan-orlov}.
\begin{lem}\label{lemma:isotopy-fuk-map} 
Let $(M, \l)$ be a Liouville manifold. Suppose that $(\fk{c}_t)_{t \in [0,1]}$ is an isotopy of stops induced by a global contact isotopy. Then there is an induced homotopy-commutative diagram 
\begin{equation}
\begin{tikzcd}
\mc{W}(M, \fk{c}_1)  \ar{rr}{\simeq} \ar{dr} & & \mc{W}(M, \fk{c}_2) \ar{dl} \\
& \mc{W}(M) &
\end{tikzcd}
\end{equation}
\qed
\end{lem}

It is a fundamental fact, due independently to Chantraine--Dimitroglou Rizell--Ghiggini--Golovko \cite{cdgg} and Ganatra--Pardon--Shende \cite[Thm.\ 1.13]{gps2} that wrapped Fukaya categories of Weinstein manifolds are generated by cocores. This can be seen rather easily to imply, due to deep work of Ganatra \cite{sheel-thesis}, that the wrapped Fukaya category of a Weinstein manifold is (homologically) smooth. 

We will need the following more general version these facts, holding for pairs $(M, \fk{c})$ where $M$ is Weinstein and $\fk{c}$ is mostly-Legendrian (up to deformation). 

\begin{fact}[cf.\ Cor.\ 1.19 in \cite{gps2}]\label{fact:smoothness}
Suppose that $(M, \l)$ is Weinstein and $\fk{c} \sub \d_\infty M$ is mostly Legendrian up to deformation, then $\mc{W}(M, \fk{c})$ is smooth.
\end{fact}

(To deduce \Cref{fact:smoothness} from \cite[Cor.\ 1.19]{gps2}, note that the assumption that $\fk{c}$ is mostly-Legendrian ensures that in can be positively displaced from itself, which allows us to apply \cite[Cor.\ 1.19]{gps2} in view of the paragraph directly following the statement of this corollary in loc.\ cit.)


\section{Filtered homological algebra}\label{sec:filteredalgebra}
\subsection{Categories of filtered complexes}

\subsubsection{Basic definitions and conventions}
Our main references for filtered homological algebra are \cite{ciricispectral} and \cite[Sec.\ 5]{weibel}.

A \emph{filtration $F^p(-)$} on a chain complex $C$ is the data of an increasing sequence of chain subcomplexes
\begin{equation}
    \dots \sub F^{p-1} C \sub F^p C \sub F^{p+1} \sub \dots.
\end{equation}
Such a filtration is said to be \emph{integral} if $p \in \Z$.

A filtration is said to be \emph{bounded below} (resp.\ \emph{bounded above}) if $F^pC=0$ for $p \ll 0$ (resp.\ $F^pC= C$ for $p \gg 0$). A filtration is said to be \emph{non-negative} if $F^pC=0$ for $p<0$. A filtration is \emph{exhausting} if $\cup_p F^pC=C$. 

Given a filtration $F^p(-)$ on a chain complex $C$, it naturally induces a filtration on its cohomology, which will also be denoted by $F^p(-)$: 
\begin{equation}\label{equation:filtration-cohomology}
 F^pH(C):= \op{im}(H(F^pC) \to H(C)). 
\end{equation} 

A \emph{filtered chain complex} $C= (C, F^p(-))$ is the data of a chain complex along with a filtration. A \emph{morphism} of filtered chain complexes is a chain morphism $f: C_1 \to C_2$ of which respects the filtration, meaning that $f(F^pC_1) \sub F^pC_2$. 

\textemph{Unless otherwise indicated, all filtrations considered in this paper are increasing, integral and non-negative}. 

There is also a notion of a bifiltration which is defined similarly. The following definition will be used later:
\begin{defn}
Consider a complex $C$ with a bifiltration $F^{k,l}C, k,l\in \bZ_{\geq 0}$, that is increasing in each degree and such that $\bigcup_{k,l}F^{k,l}C=C$. We assume the following compatibility condition: $F^{i,j}C\cap F^{i',j'}C=F^{\min\{i,i'\},\min\{j,j'\}}C$. This allows one to find a non-canonical decomposition $C=\bigoplus_{i,j\geq 0}G_{i,j}$ such that $F^{k,l}C=\bigoplus_{0\leq i \leq k\atop 0\leq j\leq l}G_{i,j}$. One can turn $C$ into a filtered complex by letting $F^pC=\sum_{k+l\leq p} F^{k,l}C$. We call this filtration the \emph{total filtration}.	
\end{defn}
\begin{rk}
Conversely, if one is given two filtrations $F_1$ and $F_2$ on $C$, one can construct a bifiltration $F^{k,l}C=F_1^kC\cap F_2^lC$. 	
\end{rk}
We need to recall the notion of boundary depth for acyclic chain complexes. 
\begin{defn}
Let $C$ be an acyclic chain complex with a filtration. \emph{The boundary depth} of $C$ is the smallest $d$ such that for any $p$ and any $x\in F^pC$ satisfying $dx=0$, there exists $y\in F^{p+d}C$ such that $dy=x$. If there is no such integer, we define it to be infinite.
\end{defn}

Finally, we wish to consider the following notions of equivalence of filtered complexes. 

\begin{defn}\label{definition:s-t-equiv-complex}
	Let $f:C_1\to C_2$ be a morphism of filtered chain complexes. We call it a \emph{scaling equivalence} if it is a quasi-isomorphism and there are constants $a\geq 1, b\geq 0$ such that the pre-image of $F^pH(C_2)$ under the induced isomorphism $H(C_1)\to H(C_2)$ is contained in $F^{ap+b}H(C_1)$. We call it a \emph{translation equivalence} if the factor $a$ can be taken $1$. 
	
	We call two complexes \emph{scaling equivalent} (resp.\ \emph{translation equivalent}) if they can be related by a zigzag of scaling equivalences (resp.\ translation equivalences). Note that we do not demand the existence of a filtered quasi-inverse.
\end{defn}

We record the following (rather obvious) criterion for a filtered morphism to be a scaling equivalence.
\begin{lem}\label{lemma:scaling-criterion}
Let $f:C_1 \to C_2$ be a (filtered) quasi-isomorphism of filtered chain complexes. Suppose that there exists a chain map $g: C_2 \to C_1$ satisfying the following properties:
\begin{itemize}
\item $g \circ f =\op{id}_{C_1}$;
\item $g$ respects the filtration up to scaling (i.e.\ $g(F^pC_2) \sub F^{ap+b}C_1$ for some $a \geq 1, b\geq 0$). 
\end{itemize}
Then $f$ is a scaling equivalence. 
\qed
\end{lem}
In general, it may be hard to characterize equivalent complexes. For this, we require invariants preserved under equivalence.
\begin{defn}\label{definition:s-t-equiv-function}
Consider two weakly increasing functions $\gr_1,\gr_2:\bN\to\bR\cup\{\infty\}$. We call them \emph{translation equivalent}, if there exists $b\geq 0$ such that $\gr_1(p)\leq \gr_2(p+b)+b$ and $\gr_2(p)\leq \gr_1(p+b)+b$ for all $p\in \bN$. We call two functions \emph{scaling equivalent} (or simply \emph{equivalent}), if there exists $a\geq 1$, $b\geq 0$ such that $\gr_1(p)\leq a\gr_2(ap+b)+b$ and $\gr_2(p)\leq a\gr_1(ap+b)+b$ for all $p\in \bN$. 
\end{defn}
\begin{exmp}
The functions $\gr_1(p)=p^{r_1}$ and $\gr_2(p)=p^{r_2}$, $r_1,r_2\geq 0$ are translation or scaling equivalent if and only if $r_1=r_2$. The functions $\gr_1(p)=e^{\sqrt{p}}$ and $\gr_2(p)=e^{p}$ are neither translation, nor scaling equivalent. The functions $\gr_1(p)=e^{r_1p}$ and $\gr_2(p)=e^{r_2p}$, $r_1,r_2\geq 0$ are always scaling equivalent, but they are translation equivalent if and only if $r_1=r_2$. 
\end{exmp}
\begin{defn}\label{definition:growth-function-chain}
	Given a filtered complex $C$, we define the \emph{growth function $\gr_C$} of $C$ by 
	\begin{equation}
	    \gr_C(p)=dim(F^pH(C))=dim(im(H(F^pC)\to H(C))).
	\end{equation}
\end{defn}
\begin{lem}\label{lem:equivalentcxequivalentfunction}
If $C_1$, $C_2$ are related by a zigzag of translation equivalences, then $\gr_{C_1}$ and $\gr_{C_2}$ are translation equivalent. If they are related by a zigzag of scaling equivalences, then their growth functions are equivalent.\qed
\end{lem}
\subsubsection{Spectral sequences}

In this section, we follow the notation of \cite{ciricispectral}, except we use upper indexing for filtrations. Given a filtered chain complex $C$, there is a canonically associated spectral sequence $\{E_r(C), \delta_r\}_{r \geq 0}$. The $r$-th page is a bigraded complex defined as follows:
\begin{equation}
E^{p,q}_r= Z^{p,q}_r/ B^{p,q}_r.
\end{equation}
where the \emph{r-cycles} are
$Z_r^{p, n+p}:= F^pC^n \cap d^{-1}(F^{p-r} C^{n+1})$ 
and the \emph{r-boundaries} are
$B_0^{p, n+p}(C):= Z_0^{p-1, n+p-1}(C)$ for $r=0$ and $
B_r^{p, n+p}(C):= Z^{p-1, n+p-1}_{r-1}(C)+ dZ_{r-1}^{p+r-1, n+p+r-2}(C)$ for $ r \geq 1$.

We now record some standard properties of spectral sequences drawn from \cite[Sec.\ 5]{weibel}, to which we also refer for the relevant definitions. (Although \cite[Sec.\ 5]{weibel} is phrased in terms of homologically graded chain complexes, one can pass to our cohomological conventions via the usual grading swap $C^* \leadsto C_{-*}$.)

\begin{prop}\label{proposition:convergence}
Let $C_1, C_2$ be filtered chain complexes with filtrations bounded below and exhausting. 
\begin{enumerate}
\item The spectral sequence $\{E_r(C_i)\}$ is bounded below and converges to $H^*(C_i)$; 
\item If $f: C_1 \to C_2$ is a morphism of filtered chain complexes such that the induced map $f_r: E_r^{p,q}(C_1) \to E_r^{p,q}(C_2)$ is an isomorphism for some $r$ and all $p,q$, then $f$ is a quasi-isomorphism.
\end{enumerate}
\end{prop}

\pf
The first part is (an abridged version of) \cite[Thm.\ 5.5.1]{weibel}. The second part follows from \cite[Thm.\ 5.2.12]{weibel}.
\epf
As a special case of \Cref{proposition:convergence}, a morphism of filtered chain complexes with bounded below filtration is a quasi-isomorphism if it induces an isomorphism on associated gradeds (this corresponds to the case $r=1$). This fact will be used repeatedly in the sequel. 

The simplest type of map that induces a translation equivalence is a filtered equivalence, i.e.\ a map that induced a quasi-isomorphism $H(\op{Gr}^pC_1)\to H(\op{Gr}^pC_2)$. Unfortunately, this notion is too strong. Hence, following \cite{ciriciguillen}, \cite{ciricispectral}, we define 
\begin{defn}
	A filtered map $C_1\to C_2$ is called an \emph{$E_r$-quasi-isomorphism} if the induced map between $E_r$-page of the corresponding spectral sequences is a quasi-isomorphism (meaning that the induced map between $E_{r+1}$-pages is an isomorphism). If $C$ is a filtered complex and $E_r(C)$ is acyclic, i.e.\ $E_{r+1}(C)=0$, we say that $C$ is \emph{$E_r$-acyclic}. 
\end{defn}
\begin{rk}
If the spectral sequences of $C_1$ and $C_2$ converge (which will always be the case for us), then an $E_r$-quasi-isomorphism is a translation equivalence; thus, $\gr_{C_1}$ and $\gr_{C_2}$ are translation equivalent. Moreover, in this case the growth functions are actually equal, not just equivalent.  
\end{rk}
\begin{lem}\label{lem:bdrdepthimplieseracyclic}
	Let $C$ be an acyclic filtered complex with boundary depth $r\in\bN$. Then $C$ is $E_r$-acyclic.
\end{lem}
\begin{proof}
	Let $x\in F^pC$ be an $(r+1)$-cocycle, i.e.\ $dx\in F^{p-r-1}C$. As $dx$ is closed, there exists $z\in F^{p-r-1+r}C=F^{p-1}C$ such that $dz=dx$. Hence, $x-z$ is closed and there exists $y\in F^{p+r}C$ such that $dy=x-z$, i.e. $x=dy+z$. This shows $x$ is an $(r+1)$-coboundary, i.e.\ it vanishes in $E_{r+1}$-page. Since $x$ was arbitrary, this implies $E_{r+1}(C)=0$. 
\end{proof}
\begin{lem}\label{lemma:E_r-equiv-map}
Let $C_1, C_2$ be filtered chain complexes with bounded below, exhausting filtrations, and let $f: C_1 \to C_2$ be an $E_r$-quasi-isomorphism. Then for any $p$, $f$ induces an isomorphism $F^pH(C_1)\to F^pH(C_2)$. In particular, an $E_r$-quasi-isomorphism induces a translation equivalence.
\end{lem}
\begin{proof}
For any such filtered complex $C$, convergence implies that $\op{Gr}^pH(C)$, corresponding to the filtration $F^pH(C)$, is equal to $\bigoplus_j E_\infty^{p,j}(C)$ (to clarify, one has $E_\infty^{p,n+p}(C)=\op{Gr}^pH^n(C)$). Moreover, $f$ induces an isomorphism of $E_\infty$-pages by \Cref{proposition:convergence}. As the maps of filtered vector spaces that induce isomorphism on associated graded vector spaces are isomorphisms, the result follows. 
\end{proof}

\subsubsection{Cones}
Given a filtered map $f:C_1\to C_2$, one can construct its \emph{$r$-cone} as the ordinary cone, with filtration $F^p cone_r(f)=F^{p-r}C_1[1]\oplus F^pC_2$, and with the usual differential. 
\begin{rk}\label{rk:ciriciquasicone}
One motivation for this definition is that $f$ is an $E_r$-quasi-isomorphism if and only if $cone_r(f)$ is $E_r$-acyclic, i.e. its $E_{r+1}$-page vanishes (see \cite[Remark 3.6]{ciricispectral}). Moreover, in \cite{ciricispectral}, the authors construct model structures on the category of filtered complexes whose weak equivalences are given by $E_r$-quasi-isomorphisms. We do not need this. 
\end{rk}
	$f$ is an $E_r$-quasi-isomorphism if and only if $cone_r(f)$ is $E_r$-acyclic, i.e. its $E_{r+1}$-page vanishes. 
\begin{lem}\label{lem:conewithacyclic}
	Let $f:C_1\to C_2$ be a filtered chain map and $C_1$ be $E_r$-acyclic. Then the natural map $C_2\to cone_r(f)$ is an $E_r$-quasi-isomorphism.
\end{lem}
\begin{proof}
	This presumably follows from \Cref{rk:ciriciquasicone}, but we prefer to give a direct argument to show the induced map between $E_{r+1}$-pages is an isomorphism. 
	
	For injectivity, let $x\in F^pC_2$	be an $(r+1)$-cocycle and assume $(0,x)\in F^pcone_r(f)$ is an $(r+1)$-coboundary, i.e. there exists $(y,z)\in F^{p+r}cone_r(f)=F^pC_1[1] \oplus F^{p+r}C_2$ such that $d(y,z)\in F^{p}cone_r(f)$ and $d(y,z)-(0,x)\in F^{p-1}cone_r(f)$ (which also implies $d(y,z)-(0,x)\in Z_r^{p-1}(cone_r(f))$). In other words, $dy\in F^{p-1-r}C_1$ and $f(y)+dz-x\in F^{p-1}C_2$, i.e. $f(y)+dz=x$ modulo $F^{p-1}C_2$. In particular, $y$ is an $(r+1)$-cocycle as well. As $C_1$ has vanishing $E_{r+1}$-page, there exists $w\in F^{p+r}C_1$ such that $dw\in F^{p}C_1$ and $y=dw$ modulo $F^{p-1}C_1$. Hence, $x=f(y)+dz=d(f(w)+z)$ modulo $F^{p-1}C_2$. This shows $x\in F^pC_2$ is an $(r+1)$-coboundary, completing the proof of injectivity.
	
	For surjectivity, let $(x,y)\in F^pcone_r(f)=F^{p-r}C_1[1] \oplus F^pC_2$ be an $(r+1)$-cocycle, i.e. $d(x,y)=(-dx,f(x)+dy)\in F^{p-2r-1}C_1[1]\oplus F^{p-r-1}C_2$. Therefore, $x\in F^{p-r}C_1$ is an $(r+1)$-cocycle itself and hence is an $(r+1)$-coboundary because $C_1$ is $E_r$-acyclic. As a result, there exists $z\in Z_r^{p}(C_1)$ such that $x=dz$ modulo $F^{p-r-1}C_1$. One can check $(z,0)\in Z_r^{p+r}(cone_r(f))$ and $(x-dz,0)\in Z_r^{p-1}(cone_r(f))$. Therefore, $(x,y)$ has the same class as $(x,y)+d(z,0)-(x-dz,0)$ in the $E_{r+1}$-page and it is of the form $(0,y')$. It is easy to check $y'\in F^p C_2$ is an $(r+1)$-cocycle. This shows the surjectivity of the map induced on the $E_{r+1}$-page as $y'$ maps to $(0,y')$. 
\end{proof}

One can similarly construct iterated $r$-cones. For instance, assume $A_n\to A_{n-1}\to \dots \to A_0$ is a finite sequence of filtered chain complexes, with filtered chain maps such that two adjacent maps compose to $0$. This is an example of a twisted complex of the simplest kind. The iterated $r$-cone has the same underlying complex as the iterated cone, i.e. it is equal to $B=\bigoplus_{i=0}^n A_i[i]$ as a graded vector space and carries the usual lower triangular cone differential. On the other hand, we filter it as 
\begin{equation}\label{eq:bdirectsum}
	F^pB=\bigoplus_{i=0}^n F^{p-ir}A_i[i]
\end{equation}
More generally, consider a sequence of filtered chain maps $A_n\to A_{n-1}\to \dots \to A_0$, that is not required the vanishing of adjacent composition condition, but enhanced into a genuine twisted complex by adding homotopies $A_i\to A_j[j-i+1]$, $j<i-1$. This is the same data as a sequence of chain maps $A_1\to B_0:=A_0$, $A_2\to B_1:=cone (A_1\to B_0)$, $A_3\to B_2:=cone (A_2\to B_1)$, and one can consider the iterated cone as the complex with underlying graded vector space $\bigoplus_{i=0}^n A_i[i]$, and this complex can be filtered as in (\ref{eq:bdirectsum}). One can construct this filtered complex as an iterated $r$-cone: namely as above define $B_0:=A_0, B_1:=cone_r(A_1\to B_0)$. Inductively, one has a map $T_r^{i}A_{i+1}\to B_i$, where $T_r$ is the translation operator ($F^p(T_rC)=F^{p-r}C$). Define $B_{i+1}:=cone_r(T_r^iA_{i+1}\to B_i)$. 

One can generalize this to infinite sequences $\dots \to A_3\to A_2\to A_1\to A_0$ as well. Its underlying graded filtered vector space is 
$B=\bigoplus_{i=0}^\infty A_i[i]$ with filtration $F^pB:=\bigoplus_{i=0}^\infty F^{p-ir}A_i[i]$. Then one has 
\begin{lem}\label{lem:iteratedconequasi}
	Assume $\dots \to A_3\to A_2\to A_1\to A_0$ is as above and let $B$ denote the corresponding iterated $r$-cone. Assume for all $i>0$ that $A_i$ is $E_r$-acyclic and non-negatively filtered. Then the natural map $A_0\to B$ is an $E_r$-quasi-isomorphism.
\end{lem}
\begin{proof}
	Consider the finite iterated cone of $A_n\to \dots \to A_1$, denote it by $B_n$. By an iterated application of Lemma \ref{lem:conewithacyclic}, the natural map $A_0\to B_n$ is an $E_r$-quasi-isomorphism. $B_n$ is a filtered subcomplex of $B$ and $B=\bigcup B_n$. Moreover, for a finite set of $p$, $F^pB_n=F^pB$ for $n\gg 0$ (this uses the fact that the $A_i$ carry non-negative filtrations, $i>0$). This implies that every entry of $E_r(B_n)$ stabilizes for $n\gg 0$; hence, $B_0\to B$ is also an $E_r$-quasi-isomorphism.
\end{proof}

\subsection{Filtered directed systems and homotopy colimits}\label{subsection:filtered-ds} 

A \emph{filtered directed system} indexed by $\mathbb{N}_+$ is the data of vector spaces $\{V_\s\}_{\s \in \mathbb{N}_+}$ along with morphisms $V_\s \to V_{\s'}$ for $\s \leq \s'$ which behave naturally under composition (i.e.\ a functor $\mathbb{N}_+ \to \op{Vect}_k$). A \emph{weak morphism} of filtered direct systems $\{V_\s \} \to \{W_\s\}$ consists in the following data: (i) a natural number $C\geq 1$; (ii) for each $\s \in \mathbb{N}_+$, a map $V_\s \to W_{C\s}$ such that the following diagram commutes:
\begin{equation}
\begin{tikzcd}
V_{\s} \ar{r}  \ar{d} &V_{\s'} \ar{d}\\
W_{C\s} \ar{r} & W_{C\s'}
\end{tikzcd}
\end{equation}
for all $\s \leq \s' $. One defines the composition of weak morphisms in the obvious way, and it is easy to see that filtered directed systems with weak morphisms form a category. 

For any given $C\in\bN, C\geq 1$, one can define a weak morphism $h_C:\{V_\s\}\to \{V_\s\}$ by just using the maps $V_\sigma\to V_{C\sigma}$ of the filtered system. One can check that $h_C\circ h_{C'}=h_{CC'}$, and for a given weak morphism $f:\{V_\sigma\}\to \{W_\sigma\}$, $f\circ h_C=h_C\circ f$. We call two weak morphisms $f_1,f_2:\{V\}\to \{W\}$ \emph{equivalent} if there exist $C_1,C_2\geq 1$ such that $h_{C_1}\circ f_1=h_{C_2}\circ f_2$. It is clear that filtered complexes, with weak morphisms up to equivalence also form a category (one only needs to check the composition is well-defined, which follows from properties above). An isomorphism in this category is called a \emph{weak isomorphism}. Concretely, this is a weak morphism $f:\{V_\sigma\}\to \{W_\sigma\}$ such that there exists $g:\{W_\sigma\}\to \{V_\sigma\}$ satisfying $f\circ g=h_C$, $g\circ f=h_C$ for some $C$ (the factors are the same for the right and the left compositions). Note that the former category is not a linear category, whereas the latter is. 

A filtered directed system induces a filtration on its colimit:
\eq\label{equation:d-s-filtration} F^p (\colim_\tau V_\tau) = \bigcup_{\s \leq p} \op{im}(V_\s \to \colim_\tau V_\tau), \eeq
from which we can extract a growth function as in \Cref{definition:growth-function-chain}. 

\begin{lem}\label{lemma:fds-equivalence}
A weak isomorphism of filtered directed systems induces an isomorphism of vector spaces on the colimits. Moreover, the associated growth functions are scaling equivalent. 
\qed
\end{lem}

\begin{rk}\label{remark:filtrations-ds}
If $C$ is a filtered complex with exhausting filtration, then $H(F^pC)$ is a directed system and we have $\colim_p H(F^pC)= H(\colim_p F^pC)= H(C)$. Then the filtration induced on $C$ via \eqref{equation:d-s-filtration} coincides tautologically with the filtration induced on $C$ via \eqref{equation:filtration-cohomology}. 
\end{rk}

Next, we want to consider a chain-level enhancement of the above discussion. 

\begin{defn}\label{definition:hocolim}
Let $C_0\xrightarrow{f_0} C_1\xrightarrow{f_1} C_2\xrightarrow{f_2} \dots$ be a sequence of chain maps. Define the \emph{homotopy colimit} $\hocolim_i C_i$  to be the cone of the morphism $\bigoplus_{i=0}^\infty C_i\xrightarrow{f-1} \bigoplus_{i=0}^\infty C_i$, where $f$ is the homomorphisms whose components are given by $f_i:C_i\to C_{i+1}$, and $1$ is the identity map. 
\end{defn}
Observe that $H^*(\hocolim_i C_i)\cong \colim_i H^*(C_i)$. 
Also notice that $\hocolim_i C_i$ carries a natural filtration, namely let $F^0 \hocolim_i C_i=cone (0\to C_0)=C_0$ and 
\begin{equation}\label{equation:hocolim-filtration}
F^p \hocolim_i C_i= \op{cone} \Bigg(\bigoplus_{i=0}^{p-1}C_i\xrightarrow{f-1} \bigoplus_{i=0}^{p}C_i\Bigg)\subset \hocolim_i C_i	
\end{equation}
If all $C_i$ are filtered and the $f_i$ are filtered morphisms, then these can be incorporated into homotopy colimit as well.

\begin{lem}
With the notation of \Cref{definition:hocolim}, there is a canonical isomorphism $H(\hocolim_i C_i) =\colim_i H(C_i)$. Moreover, the cohomological filtration (in the sense of \eqref{equation:filtration-cohomology}) induced from \eqref{equation:hocolim-filtration} coincides with the directed system filtration as defined in \eqref{equation:d-s-filtration}.
\end{lem}
\begin{proof}
By definition, one has a long exact sequence 
\begin{equation}
\dots\to H^*\bigg(\bigoplus_{i=0}^\infty C_i\bigg)\xrightarrow{f-1} H^*\bigg(\bigoplus_{i=0}^\infty C_i\bigg)\to H^*(\hocolim_i C_i)\to 	H^{*+1}\bigg(\bigoplus_{i=0}^\infty C_i\bigg)\to\dots 
\end{equation}	
and the map 
\begin{equation}\label{eq:sumtosumcolim}
	H^*\bigg(\bigoplus_{i=0}^\infty C_i\bigg)=\bigoplus_{i=0}^\infty H^*(C_i)\xrightarrow{f-1} H^*\bigg(\bigoplus_{i=0}^\infty C_i\bigg)=\bigoplus_{i=0}^\infty H^*(C_i)\end{equation} 
is clearly injective. Therefore, $H(\hocolim_i C_i)$ identifies with the cokernel of \eqref{eq:sumtosumcolim}. On the other hand, this cokernel, by definition, is spanned by the sequences $(0,\dots, 0, x_i , 0,\dots  )$ modulo the relation $(0,\dots, 0, x_i , 0,\dots  )\sim (0,\dots, 0,0, f(x_i) ,\dots  )$ (in the latter $(i+1)^{th}$ slot is non-zero). More generally, it is the set of sequences $(x_1,x_2,\dots,0,\dots)$ modulo the relation $(x_1,x_2,0,\dots)\sim (0,f(x_1),f(x_2),\dots, 0,\dots )$. This is how one can define colimits in the category of vector spaces; hence, the first claim follows. 

For the second claim, notice that similarly to above, $H(F^p\hocolim_i C_i)$ is the colimit of the finite system $H(C_1)\to\dots \to H(C_p)$; hence, isomorphic to $H(C_p)$. This finishes the proof.  
\end{proof}

The following will also be useful:
\begin{lem}\label{lem:colimitdepth}
Let $d\in\mathbb{N}_+$. If for any $i \geq 0$, the composition $C_i\to \dots\to C_{i+d}$ induces the $0$-map in cohomology, then $\hocolim_i C_i$ has boundary depth at most $d$. In particular, it is $E_d$-acyclic according to \Cref{lem:bdrdepthimplieseracyclic}.
\end{lem}
\begin{proof}
Let $y\in F^p\hocolim_i C_i=\bigoplus_{i=0}^{p-1}C_i[1]\oplus \bigoplus_{i=0}^{p}C_i$ be a closed element. Write $y=(\{x_i\},\{y_i\})$, where $x_i\in C_i[1]$ and $y_i\in C_i$. We need to show $y$ is exact in $F^{p+d}\hocolim_i C_i$. 

As $y$ is closed $dx_i=0$ and $f(x_{i-1})-x_i+dy_i=0$ for all $i$. In particular, $x_0=dy_0$ is exact, and the same holds for all $x_i$ by induction. Choose primitives $z_i$ for all $x_i$ such that $z_i=0$ if $x_i=0$. Then, $(\{z_i\},\{0\})\in F^p\hocolim_i C_i$, and all first components of $y+d(\{z_i\},\{0\})$ vanish. Hence, without loss of generality, we can assume the same for $y$, i.e.\ we can assume $x_i=0$. Moreover, $y$ can be seen as the sum of closed elements $(\{0\},\{0,\dots, 0, y_i,0,\dots \})$; hence, we can assume only one $y_i$ is non-zero, and $i\leq p$ as $y\in F^p\hocolim_i C_i$. 

Notice that if $i<p+d$, then $(\{0\}, \{0,\dots, 0, y_i,0,\dots \})$ is cohomologous to 
\begin{equation}
	(\{0\}, \{0,\dots, 0, 0,f(y_i),\dots \})
\end{equation}
where $f(y_i)$ is put into $(i+1)^{th}$ slot. Indeed, their difference is given by the differential $d(\{0,\dots, 0, y_i,0,\dots \},\{0\})$, where $y_i$ is in the $i^{th}$ slot of the first component. Therefore, the element $(\{0\}, \{0,\dots, 0, y_i,0,\dots \})$ is cohomologous in $F^{p+d}\hocolim_i C_i$ to $(\{0\}, \{0,\dots, 0, f^{p+d-i}(y_i),0,\dots \})$, where $f^{p+d-i}(y_i)$ is in the $(p+d)$-th slot. By assumption, $f^d(y_i)$ is exact in $C_{i+d}$ and, as $p-i\geq 0$, the same is true for $f^{p+d-i}(y_i)\in C_{p+d}$. This implies  $(\{0\}, \{0,\dots, 0, f^{p+d-i}(y_i),0,\dots \})$ is exact in $F^{p+d}\hocolim_i C_i$. Thus, the cohomologous element $(\{0\}, \{0,\dots, 0, y_i,0,\dots \})$ is also exact in $F^{p+d}\hocolim_i C_i$, finishing the proof of the boundary depth assertion.
\end{proof}

\subsection{Filtered $A_\infty$ categories and modules}

\begin{defn}
A \emph{filtered $A_\infty$ category $\cC=(\cC, F^p(-))$} is an $A_\infty$ category $\cC$ such that 
\begin{itemize}
	\item for every $K,L\in ob(\cC)$, $\cC(K,L)$ is a filtered complex
	\item $A_\infty$ operations respect the filtration, i.e. $\mu^k_{\cC}(x_k,\dots,x_1) \in F^{p_1+\dots+p_k}{\cB}(L_0, L_k)$ if $x_i \in F^{p_i}{\cC}(L_{i-1}, L_i)$
\end{itemize}
Unless specified otherwise, the filtration on each $\cC(K,L)$ is assumed to be integral, non-negative and exhausting. 
\end{defn}
A filtered functor $\mc{F}: \cB \to \cB'$ between filtered $A_\infty$ categories is a functor which preserves the filtration, i.e.\  $\mc{F}^k(a_k,\dots,a_1) \in F^{p_1+\dots+p_k}\hom_{\cB'}(\mc{F}(L_0), \mc{F}(L_k))$ if $a_i \in F^{p_i}\hom_{\cB}(L_{i-1}, L_i)$. 

The cohomology category of a filtered $A_\infty$ category is a filtered linear graded category. A filtered functor between filtered $A_\infty$ categories induces a filtered functor on the associated filtered cohomological categories. 

\begin{lem}\label{lemma:trans-equiv-isom}
Let $\cB$ be a filtered $A_\infty$ category. Assume $K, K'$ (resp.\ $L, L'$) are isomorphic objects in $H^*(\cB)$. Then the growth functions of $\hom_{\cB}(K, L)$ and $\hom_{\cB}(K', L')$ are translation equivalent. 
\qed
\end{lem}
Here, isomorphic means isomorphic in the underlying unfiltered category. As the filtration is exhausting, the isomorphisms will all lie in some $F^p$, and the growth functions will be equivalent up to a translation factor of $2p$. 
\begin{defn}
A filtered functor $\mc{F}:\cB \to \cB'$ is called a \emph{scaling equivalence} (resp.\ an \emph{$E_k$-equivalence}) if the induced filtered morphism $\hom_\cB(K,L) \to \hom_{\cB'}(\mc{F}(K), \mc{F}(L))$ is a scaling equivalence (resp.\ an $E_k$ quasi-isomorphism).
\end{defn}
\begin{rk}
We do not require scaling, resp. $E_k$, equivalences to be quasi-equivalences, and perhaps scaling, resp. $E_k$, fully faithful would be a more accurate term. In practice, scaling, resp. $E_k$-equivalences we encounter will be Morita equivalences. 
\end{rk}

\section{The localization filtration}\label{subsection:tensor-length-filtration}\label{sec:localizationfiltration}

\subsection{Construction}\label{subsection:construction-LO-quotient}
\begin{defn}[\cite{lyubashenko-ovsienko},\cite{gps2}]\label{defn:quotientcategory}
Let $\mc{B}$ be an $A_\infty$ category and let $\mc{D}$ be a set of objects. The \emph{(Lyubashenko--Ovsienko) quotient category $\mc{B}/\mc{D}$} is the $A_\infty$ category with $\op{Ob} \mc{B}/\mc{D} = \op{Ob} \mc{B}$ and morphisms given by the bar complex:
\eq
\mc{B}/\mc{D}(K, L):= \bigoplus_{E_1,\ldots,E_k \in \mc{D}} \mc{B}(E_k, L) \otimes \dots \otimes {\mc{B}}(K, E_1)[k]
\eeq
The summand at $k=0$ is $\cB(K,L)$. 
\end{defn}

The differential of $\mc{B}/\mc{D}(K, L)$ is similar to the standard bar differential, namely \eq \mu^1_{\cB/\cD}(y'\otimes x_{k-1}\otimes \dots x_2\otimes y  )\eeq is the signed sum of terms $\mu_\cB(y',\dots x_i)\otimes x_{i-1}\otimes \dots \otimes y$, $y'\otimes \dots \mu_\cB(x_j,\dots,x_i ) \dots\otimes y$ and $y'\otimes x_{k-1}\otimes \dots\otimes \mu_\cB( x_i,\dots,y)$. The higher products are defined similarly. For instance, given $y'''\otimes x_l'\dots x_1'\otimes y''$ and $y'\otimes x_k\otimes \dots y$, to obtain their product, one must apply the $A_\infty$-product to all substrings of \eq y'''\otimes x_l'\dots x_1'\otimes y''\otimes y'\otimes x_k\otimes \dots y\eeq that contain $y''\otimes y'$. See \cite[Sec.\ 3.1.2]{sylvan-orlov} for precise formulas.

Observe that $\cB/\cD$ is filtered by $k$. More precisely, it is a filtered category with filtration given by 
\begin{equation}
F^p\mc{B}/\mc{D}(K, L):= \bigoplus_{0\leq k\leq p\atop E_1,\ldots,E_k \in \mc{D}} \mc{B}(E_k, L) \otimes \dots \otimes {\mc{B}}(K, E_1)[k]	
\end{equation}
The localization as a filtered category is functorial. In other words, if $f:\cB\to\cB'$ is an $A_\infty$-functor such that $f(\cD)\subset\cD'$, then $f$ induces a functor $\cB/\cD\to \cB'/\cD'$ that respect the filtration.

We also have
\begin{lem}\label{lemma:summand-e1} 
Let $\cD\subset\cD'\subset \cB$ be full subcategories (with $\mc{D}, \mc{D}'$ small) and assume every object of $\cD'$ is a shift of a direct summand of an object of $\cD$. Then the natural functor $\cB/\cD\to \cB/\cD'$ is an \emph{$E_1$-equivalence}, i.e. it is a quasi-equivalence inducing an $E_1$-quasi-isomorphism of the underlying complexes.
\end{lem}
\begin{proof}
Let $K,L\in \cB$ and consider the filtered chain map 
\eq\label{eq:e1smdprecoh} (\cB/\cD)(K,L)\to (\cB/\cD')(K,L).\eeq
The induced map of the $E_1$ pages is 
\eq\label{equation:e-1-summand}
H(\cB)/H(\cD)(K,L) \to H(\cB)/H(\cD')(K,L). \eeq
spread over different bidegrees. Here, $H(\cB)$, $H(\cD)$ and $H(\cD')$ are cohomology categories, considered as formal $A_\infty$ categories, and $H(\cB)/H(\cD)$, $H(\cB)/H(\cD')$ are quotients as in \Cref{defn:quotientcategory} (in particular, they are not formal). \eqref{equation:e-1-summand} is induced by the natural functor \eq\label{eq:e-1smdfunctor} H(\mc{B})/ H(\mc{D}) \to H(\mc{B})/ H(\mc{D}')\eeq 
Notice, the cohomology categories inherit the property of $\cD$ and $\cD'$: namely every object of $H(\cD')$ is a shift of a direct summand of an object of $H(\cD)$. Therefore, \eqref{eq:e-1smdfunctor} is a quasi-equivalence, which follows from \Cref{lemma:quotient-properties}\eqref{item:split-closure-quotient} applied to cohomology categories. Note that this quasi-equivalence is not a priori an $E_1$-equivalence; on the other hand, we still have a quasi-isomorphism \eqref{equation:e-1-summand} of chain complexes. As \eqref{equation:e-1-summand} is the map induced by \eqref{eq:e1smdprecoh} on the $E_1$-page, \eqref{eq:e1smdprecoh} is an $E_1$-quasi-isomorphism. 
\end{proof}
\begin{rk}
Note that Lemma \ref{lemma:summand-e1} fails if one merely assumes $\cD$ split-generates $\cD'$. The passage to cohomological categories break exact triangles; therefore, such an assumption does not necessarily imply $H(\cD)$ split-generates $H(\cD')$. Indeed, if this were true, one would have filtered equivalences at the cohomology level, and as we will see next, this only holds after scaling. Note on the other hand, one can also include in $\cD'$ direct sums of (shifts of summands of) objects of $\cD$, and the same proof would work. 	
\end{rk}

The following proposition plays a fundamental role in the sequel.

\begin{prop}\label{prop:scalinginverse}
Let $\cB$ be an $A_\infty$ category and $\cD\subset \cB$. Let $\cD'$ be a subcategory of $\tw_{\leq l} \cD\subset \tw\cB$ such that $\cD\subset \cD'$. Then,  
\begin{equation}
\iota:\tw\cB/\cD\to \tw\cB/\cD'	
\end{equation}
is a scaling equivalence, i.e. for every $K,L\in \tw\cB/\cD$, the map $(\tw\cB/\cD)(K,L)\to (\tw\cB/\cD')(K,L)$ is a scaling equivalence, with a scaling factor of $l$.
\end{prop}
Note that $\tw\cB/\cD$ refers to the quotient of $\tw\cB$ by $\cD$, and not $\tw(\cB/\cD)$, even though these categories are quasi-equivalent by \Cref{lemma:quotient-properties}\eqref{item:quotient-can-map}.
\begin{proof}	
We prove this by writing an explicit quasi-inverse
\begin{equation}
	r_{K,L}: (\tw\cB/\cD')(K,L) \to (\tw\cB/\cD)(K,L)
\end{equation} 
 to $\iota:(\tw\cB/\cD)(K,L)\to (\tw\cB/\cD')(K,L)$ that does not preserve the filtration, but scales it by at most $l$. The complex underlying $(\tw\cB/\cD')(K,L)$ is a direct sum of terms \begin{equation}\label{eq:quotsummandtwisted}
 	(\tw\cB)(E_p,L)\otimes (\tw\cB)(E_{p-1},E_p)\otimes\dots \otimes(\tw\cB)(K,E_1) 
 \end{equation} 
 where $E_i\in\cD'\subset\tw_{\leq l}\cD$ and $r_{K,L}$ sends $x_p \otimes \dots \otimes x_0 \in\eqref{eq:quotsummandtwisted}$ to 
\begin{equation}\label{equation:splitting-map}
 \sum_{j_1,\dots,j_p\in\bN} x_p \otimes \delta_{E_p}^{\otimes j_p} \otimes x_{p-1} \otimes \dots \otimes \delta_{E_1}^{\otimes j_1} \otimes x_0,
\end{equation}
To explain this further, recall that each $E_i$ is a formal direct sum of objects of $\cD$, and hom-sets such as $(\tw\cB)(E_{i-1},E_i)$ are direct sums of hom-sets between the summands. One can find it convenient to think of these hom-sets consisting of matrices with values in hom-sets of $\cB$. One has $\delta_{E_k}\in (\tw\cB)(E_k,E_k)$, and we use $\delta_{E_k}^{\otimes i}$ to denote  $\delta_{E_k} \otimes \dots \otimes \delta_{E_k}$ ($i$ times). More precisely, this refers to tensor product of components of $\delta_{E_k}$. 

For instance, assume $p=1$ and $E_1=\{X\xrightarrow{a} Y\xrightarrow{b} Z \}$ be the given twisted complex (call the $X\to Z[-1]$ component $h$). The components of $\delta_{E_1}$ are given by $a,b,h$, and this twisted complex lies in $\tw_{\leq 3}\cD$. For simplicity, assume $K,L\in\cB$, and for a given $m\in\cB(K,E_1)=\cB(K,X)[2]\oplus \cB(K,Y)[1]\oplus \cB(K,Z)$, denote the respective components by $m_X,m_Y,m_Z$ (and similarly for $n\in\cB(E_1,L)$). Consider $n\otimes m\in \cB(E_1,L)\otimes \cB(K,E_1)$. Then, $\delta_{E_1}^{\otimes 0}$ components of $r_{K,L}(n\otimes m)$ are given by $n_X\otimes m_X\in \cB(X,L)\otimes \cB(K,X)$, $n_Y\otimes m_Y\in \cB(Y,L)\otimes \cB(K,Y)$ and $n_Z\otimes m_Z\in \cB(Z,L)\otimes \cB(K,Z)$. The $\delta_{E_1}^{\otimes 1}$ components of $r_{K,L}(n\otimes m)$ are given by
\begin{align*}
	n_Y\otimes a\otimes m_X\in \cB(Y,L)\otimes\cB(X,Y)\otimes \cB(K,X) \\
	n_Z\otimes b\otimes m_Y \in \cB(Z,L)\otimes\cB(Y,Z)\otimes \cB(K,Y) \\
	n_Z\otimes h\otimes m_X\in \cB(Z,L)\otimes\cB(X,Z)\otimes \cB(K,X)
\end{align*}
and the only $\delta_{E_1}^{\otimes 2}$ component is given by 
\begin{equation*}
		n_Z\otimes b\otimes a\otimes m_X\in \cB(Z,L)\otimes\cB(Y,Z)\otimes \cB(X,Y)\otimes \cB(K,X)
\end{equation*}
More general case is similar. 

Since each $E_i$ is $l$-lower triangular, $r_{X,Y}$ maps $F^p(\tw\cB/\cD')(K,L)$ to $F^{pl}(\tw\cB/\cD)(K,L)$. Note that if instead $E_i\in\cB$ and we were working with bounding cochains $\delta_{E_i}\in \cB^1(E_1,E_1)$, \eqref{equation:splitting-map} would be a more precise notation, whereas in our situation it should be understood as a matrix multiplication. 

The fact that $r_{K,L}$ is a chain map is a rather tedious verification and ultimately follows from the Maurer-Cartan equation. It is also clear that $r_{K,L}\circ \iota_{K, L}$ is the identity, and $\iota_{K, L}$ is a quasi-isomorphism. Therefore, $\iota_{K, L}$ is a scaling equivalence by \Cref{lemma:scaling-criterion}. 
\end{proof}

\subsection{Filtrations via colimits and comparison with the localization filtration}\label{subsection:spherical-filtration}
Let $\cB$ be an $A_\infty$ category and let $f:\cA\to \cB$ be a spherical functor with right, resp.\ left adjoint $r$, resp.\ $l$. For simplicity assume $\cB$ is pre-triangulated. Let $S$ denote the corresponding spherical twist $cone (fr\to id_\cB)$ and let $s:id_\cB\to S$ be the canonical natural transformation. $S^k$ defines an $A_\infty$-bimodule via $\cS^k=\cB(\cdot, S^k(\cdot))$ and $s$ induces bimodule homomorphisms $\cB=\cS^0\to \cS^1\to \cS^2\to \dots$. Let $\cD$ denote the image $f(\cA)$. The goal of this section is to prove the following;
\begin{thm}\label{thm:spherical}
$\cB/\cD$, considered as a filtered bimodule over $\cB$, is $E_1$-quasi-isomorphic to $\hocolim_k \cS^k$.  
\end{thm}
\begin{cor}\label{cor:spherical}
The complexes $\hocolim_k \cB(K,S^k(L))$ and $\cB/\cD(K,L)$ are $E_1$-quasi-isomorphic.	
\end{cor}
\Cref{thm:spherical} will follow from \Cref{lem:sdvanish} and \Cref{prop:colimisloc} below. 
\begin{rk}\label{rk:generalfunctor}
Note that one can generalize \Cref{thm:spherical} to any $S$ equipped with a transformation $s:id_\cB\to S$ satisfying
\begin{enumerate}
	\item $[s_D]\in H(\cB)(D,S(D))$ vanishes
	\item $cone (s_L)\in \cD$
\end{enumerate}
for some subcategory $\cD\subset \cB$. Also notice this result is exact: no scaling or shift is needed. As we have seen in \Cref{prop:scalinginverse}, changing $\cD$ in its twisted envelope normally changes the filtration on cohomology by scaling. The first condition above can be thought as $\cD$ being small enough: including extensions of elements of $\cD$ by themselves will break this condition. On the other hand, the second condition means $\cD$ is sufficiently large.
\end{rk}

\begin{prop}\label{theorem:spherical-tensor-equiv}\label{prop:colimisloc}
Let $\cD \sub \cB$ be a set of objects. Let $L=L^0 \to L^1 \to L^2 \to \dots$ be a sequence of morphisms. Suppose that this data satisfies the following properties:  
\begin{enumerate}
\item\label{colimpropcond:vanishing} for every $D\in \cD$, the induced map $\cB(D, L^k) \to \cB(D, L^{k+1})$ vanishes in cohomology
\item\label{colimpropcond:cone} the cone of $L^k \to L^{k+1}$ is quasi-isomorphic to an object of $\cD$ 
\end{enumerate}
Then the filtered modules $\hocolim_k \cB(\cdot, L^k)=\hocolim_k h_{L_k}$ and $(\cB/\cD)(\cdot, L)$ are $E_1$-quasi-isomorphic. Similarly, if we instead consider a sequence $\dots \to L^2 \to L^1 \to L^0=L$ such that $\cB(L^k, D) \to\cB(L^{k+1}, D)$ vanishes in cohomology and the cone of $L^{k+1} \to L^k$ is contained in $\cD$, then $\hocolim_k \cB(L^k,\cdot)=\hocolim_k h^{L_k}$ and the left quotient $(\cB/\cD) (L,\cdot)$ are $E_1$-quasi-isomorphic.
\end{prop}
\begin{rk}
The conditions (\ref{colimpropcond:vanishing}) and (\ref{colimpropcond:cone}) are analogous to the conditions in \Cref{rk:generalfunctor}, the former is a smallness condition on $\cD$, whereas the latter is a largeness.
\end{rk}
\begin{rk}
The left module version is more naturally stated in terms of the left quotient by $\cD$, namely, the above mentioned colimit is $E_1$-quasi-isomorphic to the left quotient of $h^L$ by $\cD$. On the other hand, left and right quotients of the diagonal bimodule $\cB$ coincide, and it does not matter whether we plug in $L$ first or take the quotient first; therefore, this left quotient is the same as $(\cB/\cD) (L,\cdot)$. 
\end{rk}
We will prove \Cref{prop:colimisloc} in three steps:
\begin{enumerate}
	\item\label{colimitfiltrpropsteps:colimiscolimquot} show that $\hocolim_k \cB(\cdot, L^k)$ is $E_1$-quasi-isomorphic to $\hocolim_k \cB(\cdot, L^k)/\cD$
	\item\label{colimitfiltrpropsteps:cones} the cones of maps $\cB(\cdot, L^k)/\cD\to \cB(\cdot, L^{k+1})/\cD$ are $E_1$-acyclic 
	\item\label{colimitfiltrpropsteps:trivialcolimit} use this to deduce $h_L/\cD=\cB(\cdot, L^0)/\cD$ is $E_1$-quasi-isomorphic to $\hocolim_k(\cB(\cdot, L^k)/\cD)$, where the latter is filtered both by length and $k$
\end{enumerate}
We start with \eqref{colimitfiltrpropsteps:colimiscolimquot}:
\begin{lem}\label{lem:colimitfiltrpropstepscolimiscolimquot}
The natural induced map $\hocolim_k\cB(\cdot,L^k)\to \hocolim_k\cB(\cdot,L^k)/\cD$ is an $E_1$-quasi-isomorphism.
\end{lem}
\begin{proof}
Denote the non-negatively filtered $A_\infty$-module $\hocolim_k\cB(\cdot, L^k)$ by $\cN$. By Lemma \ref{lem:colimitdepth}, the filtered complex
\begin{equation}
\cN(D_l)\otimes 	\cB(D_{l-1},D_l)\otimes \dots \otimes \cB(L,D_1)\cong \hocolim_k (\cB(D_l,L^k)\otimes 	\cB(D_{l-1},D_l)\otimes \dots \otimes \cB(L,D_1))
\end{equation}
is $E_1$ acyclic, for a fixed set of $D_1,\dots ,D_l\in\cD$, if $l\geq 1$. The vanishing condition in \Cref{lem:colimitdepth} follows from the first assumption in \Cref{prop:colimisloc}.

Define $A_l$ to be the filtered complex
\begin{equation}
\bigoplus_{D_i\in \cD} \cN(D_l)\otimes 	\cB(D_{l-1},D_l)\otimes \dots \otimes \cB(L,D_1)	
\end{equation}
Here the sum varies over $D_1,\dots, D_l\in \cD$, where $l$ is fixed (define $A_0:=\cN(L)$). In other words, this is the length $l$ part of $(\cN/\cD)(L)$. Clearly, $A_i$ is also $E_1$-acyclic if $i>0$. 

One has a sequence of filtered chain maps 
\begin{equation}
\dots A_3\to A_2\to A_1 \to A_0	
\end{equation}
and higher homotopies extending this to an ``infinite twisted complex'' (which are just the components of the differential on $\bigoplus A_l[l]=(\cN/\cD)(L)$). This sequence satisfies the conditions of Lemma \ref{lem:iteratedconequasi}; hence, its $1$-iterated cone is $E_1$-quasi-isomorphic to $A_0=\cN(L)$. 

Moreover, the $1$-iterated cone satisfies
\begin{equation}
	F^p(\bigoplus A_l[l])=\bigoplus F^{p-l}A_l[l]
\end{equation}
Hence, the filtration on it coincides with the total filtration on $(\cN/\cD)(L)$ that combines the length and colimit filtrations. 

This show that $(\cN/\cD)(L)$, with its total filtration, is $E_1$-quasi-isomorphic to $\cN(L)$, where the natural quotient map from the latter is the equivalence. This completes the proof.   
\end{proof}
Now we turn to showing property (\ref{colimitfiltrpropsteps:cones}). By (\ref{colimpropcond:cone}) in \Cref{prop:colimisloc}, $D^k:=cone(L^k\to L^{k+1})\in \cD$. Therefore, the cone of the morphism $\cB(\cdot, L^k)\to \cB(\cdot, L^{k+1})$ is quasi-isomorphic to $h_{D^k}$, and 
\begin{equation}
	\cB(\cdot, L^k)/\cD\to \cB(\cdot, L^{k+1})/\cD\simeq h_{D^k}/\cD
\end{equation}
with the induced localization filtrations. The following lemma finishes the proof:
\begin{lem}
For any $D\in \cD$, $h_D/\cD$ is $E_1$-acyclic. 	
\end{lem}
\begin{proof}
For any $L$, $(h_D/\cD)(L)$ is the standard bar resolution of the left module $h^L|_\cD$ evaluated at $D$. The $E_1$-page of the corresponding spectral sequence can be seen as the bar resolution of the graded module $H(h^L)$ over the linear category $H(\cD)$, spread over different bidegrees. Hence, the $E_1$-page is acyclic, i.e. the $E_2$-page vanishes.
\end{proof}
%
Therefore, to conclude the proof of \Cref{prop:colimisloc}, we only have to complete Step (\ref{colimitfiltrpropsteps:trivialcolimit}):
\begin{lem}
$\cB(\cdot, L)/\cD=\cB(\cdot, L^0)/\cD$ is $E_1$-quasi-isomorphic to $\hocolim_k \cB(\cdot, L^k)/\cD$.	
\end{lem}
As before, the latter is considered with its total filtration. It does not matter whether we apply the colimit or the localization first.
\begin{proof}
We have shown $(\cB/\cD)(K,L^k)\to(\cB/\cD)(K,L^{k+1})$ has $E_1$-acyclic cones for all $k$. We apply Lemma \ref{lem:iteratedconequasi} again to conclude. More precisely, define $A_0:=(\cB/\cD)(K,L^0)=(\cB/\mc{D})(K,L)$ and 
\begin{equation}
A_k:=cone ((\cB/\cD)(K,L^k)\to(\cB/\cD)(K,L^{k+1}))[-k]	
\end{equation}	
we have shown $A_k$ is $E_1$-acyclic for $k>0$. One has a sequence of filtered chain maps 
\begin{equation}
	\dots\to A_3\to A_2\to A_1\to A_0
\end{equation}
where any two adjacent maps compose to $0$. The map $A_k\to A_{k-1}$ is extended by the identity of $\cB(K,L^k)$ (up to a sign). 

By Lemma \ref{lem:iteratedconequasi}, this iterated cone is $E_1$-quasi-isomorphic to $A_0=(\cB/\mc{D})(K,L)$. Moreover, if one ignores the filtration, the iterated cone is the same as $\hocolim_k (\cB/\cD)(K,L^k)$. The filtration induced by the $r$-iterated cone construction (where $r=1$), on the other hand, matches the total filtration on $\hocolim_k (\cB/\cD)(K,L^k)$. 

This shows $E_1$-quasi-isomorphism after plugging in $K$. However, the natural map $(\cB/\cD)(K,L)\to \hocolim_k (\cB/\cD)(K,L^k)$ is functorial in $K$. This finishes the proof. 
\end{proof}
The conclusion of \Cref{thm:spherical} from \Cref{prop:colimisloc} is almost immediate: we let $L^k=S^k(L)$ and the map $L^k\to L^{k+1}$ to be $s_{S^k(L)}$. The condition (\ref{colimpropcond:cone}) holds by definition. (\ref{colimpropcond:vanishing}) in this case is Lemma \ref{lem:discimagevanish}. This shows $E_1$-quasi-isomorphism of $(\cB/\cD)(\cdot, L)$ and $\hocolim \cS^k(\cdot,L)$. In this setting, the $E_1$-quasi-isomorphisms constructed for the proof of \Cref{prop:colimisloc} are functorial in $L$, showing the desired equivalence.
%
\section{Categorical pseudo-compactifications}\label{sec:categoricalpseudocompact}

\subsection{Smooth categorical pseudo-compactifications}

\subsubsection{Main definitions}
\defi\label{definition:categorical-semi-comp}
Given a smooth $A_\infty$ category $\mc{C}$, a \emph{smooth (categorical) compactification} of $\mc{C}$ is a pair ($\mc{B}, \phi)$ consisting of a smooth, proper $A_\infty$ category $\mc{B}$ and a functor $\phi: \mc{B} \to \mc{C}$ satisfying the following conditions:
\begin{itemize}
    \item[(i)] there exists a finite collection of objects of $ker \phi \sub \mc{B}$ which split-generates $ker(\phi)$;
    \item[(ii)] the quotient map $\mc{B}/ \op{ker}(\phi) \to \mc{C}$ is a Morita equivalence, i.e. it is cohomologically fully faithful and the image split-generates.
\end{itemize}
A \emph{morphism} $(\mc{B}', \phi') \to (\mc{B}, \phi)$ of smooth compactifications of $\mc{C}$ is a functor $f: \mc{B}' \to \mc{B}$ of $\mc{B}$ such that $\phi \circ f\simeq \phi'$ (i.e.\ $\phi \circ f$ and $\phi'$ are homotopic \cite[(1h)]{seidel-book}), and such that $f$ is also a localization functor (i.e. $\cB'/ker(f)\to \cB$ is a Morita equivalence). Similarly, a pair $(\cB,\phi)$ is called a \emph{smooth (categorical) pseudo-compactification} if it satisfies all properties above except properness of $\cB$. 
Morphisms of pseudo-compactifications are defined similarly. Smooth pseudo-compactifications of $\mc{C}$ form a category.
\edefi
When discussing smooth pseudo-compactifications in the sequel, we will typically drop reference to $\phi$ in our notation. We include two major examples here:
\begin{exmp}[see \Cref{subsection:filtration-dbcoh}]\label{exmp:bsidecompactificationwithzigzag}
Let $U$ be a smooth affine variety, and let $\cC=D^bCoh(U)$ (the $A_\infty$ enhanced bounded derived category of coherent sheaves on $U$). Let $U\subset X$ be a smooth, projective compactification of $U$ by a Cartier divisor $D=X\setminus U$. If one lets $\cB$ be (an enhancement of) $D^bCoh(X)$, then $\cB$ with the restriction functor $D^bCoh(X) \to D^bCoh(U)$ is a smooth categorical compactification of $\cC$. If $X'$ is another such compactification, then $X$ and $X'$ are related by a sequence of blowups and blowdowns at smooth centers in the complement of $U$ (weak factorization theorem). \cite{bondalorlovsod} implies that the pushforward along a blowdown map at a smooth center satisfies the conditions of \Cref{definition:categorical-semi-comp}; hence, it is a morphism of smooth categorical compactifications. In other words, any two smooth categorical compactification of $\cC$ obtained geometrically as above are related by a zigzag of morphisms in the category of categorical compactifications.  
\end{exmp}
\begin{rk}
We are mainly concerned about smooth categorical compactifications related by zigzags as above; however, in the examples coming from symplectic geometry, we have to include zigzags with intermediate steps in smooth categorical pseudo-compactifications. This will not prevent us from defining growth invariants as we will see below (see \Cref{corollary:growth-function-well-def} and \Cref{thm:welldefinedsymplecticgrowth}).
\end{rk}
\begin{exmp}[see \Cref{subsection:filtration-wfuk}]\label{exmp:asidecompactificationwithzigzag}
Let $M$ be a Weinstein manifold. As we will see later (see \Cref{lemma:full-stop}), one can endow it with a Lefschetz fibration and consider the core $\fc$ of a fiber (pushed to infinity) as the stop on $M$. Then $\cW(M,\fc)$ is smooth and proper and the stop removal functor $\cW(M,\fc)\to \cW(M)$ is a smooth categorical compactification. Let $\fc'$ be another such stop. By perturbing it by a contact isotopy, we can assume it is disjoint from $\fc$. Then $\cW(M,\fc\cup \fc')$ is a smooth categorical pseudo-compactification. The stop removal maps $\cW(M,\fc\cup \fc')\to \cW(M,\fc)$ and $\cW(M,\fc\cup \fc')\to \cW(M,\fc')$ are morphisms of smooth categorical pseudo-compactifications.
\end{exmp}


Given a smooth pseudo-compactification $\cB$, with a finite collection $\cD\subset ker (\phi )$ split-generating $ker(\phi)$, one can construct a filtered category $\cB/\cD$ as before. We do not know how to endow $\cC$ with a filtration; however, one can endow the linear category $H(\cC)$ with a filtration as follows: by definition, $\cB/\cD\to \cC$ is a derived equivalence, and by extending $\cB$ in its twisted envelope, without loss of generality we can assume every object of $\cC$ is a direct summand of an object in the essential image of $\cB/\cD\to\cC$. The linear category $H(\cB/\cD)$ is filtered too, and so is its Karoubi completion constructed as pairs of objects and idempotents of $H(\cB/\cD)$. There exists a linear subcategory $H\subset H(\cB/\cD)$ such that $H\to H(\cC)$ is an equivalence of categories. By choosing an inverse equivalence $H(\cC)\to H$, we can transfer the filtration on $H$ to $H(\cC)$. For each finite subcategory of $H(\cC)$, the filtration on it is independent of choices above up to a finite shift ambiguity.

As it happens, we will not need this structure: in what follows, we will only be interested in the growth function associated to pairs of objects of $\cC$. Growth functions will be discussed in detail in \Cref{subsubsection:growth-functions-pc}. To set the stage for this discussion, we need some technical results which are the focus of \Cref{subsubsection:partial-independence}.   

\begin{note}
In light of \Cref{lemma:summand-e1}, one can also allow $\cD$ to be \emph{essentially finite}, i.e. it consists of objects in finitely many quasi-isomorphism classes. If $\cD_0\subset \cD$ is a finite subcollection that contains at least one object in every quasi-isomorphism class, then $\cB/\cD_0$ and $\cB/\cD$ are $E_1$-equivalent.
\end{note}

\subsubsection{A partial independence result}\label{subsubsection:partial-independence}

Suppose that $f: \cB' \to \cB$ is a morphism of pseudo-compactifications of $\cC$. If $\cD' \sub \cB'$ and $\cD \sub \cB$ are collections of objects and $f(\cD') \sub \cD$, then there is an induced morphism $\cB'/ \cD' \to \cB / \cD$. We saw below \Cref{defn:quotientcategory} that this is in fact a filtered morphism.

The purpose of this section is to prove a finer result under certain additional smoothness and properness assumptions. The following is the key proposition:

\begin{prop}\label{prop:sodofcompactifications}
Let $f:\cB'\to \cB$ be a morphism of smooth pseudo-compactifications $(\cB,\phi)$, $(\cB',\phi')$ of $\cC$, and assume $\cB$ is proper. For simplicity, assume $\cB$ and $\cB'$ are split-closed and pre-triangulated. Then, for some essentially finite collections $\cD\subset\cB$ (resp.\ $\cD'\subset \cB'$) split-generating $ker(\phi)$ (resp.\ $ker(\phi')$), there are quasi-equivalences $\cB'/\cD' \to \cB/\cD$ and $\cB/\cD \to \cB'/\cD'$, which are quasi-inverses and which respect the filtration. 
\end{prop}
%
In case $\cB, \cB'$ are not split-closed and pre-triangulated, the same conclusion still holds upon passing to the twisted envelope. We will explain this after proving \Cref{prop:sodofcompactifications}. The key to the proof is the following lemma:
\begin{lem}
The functor $f:\cB'\to \cB$ admits a right adjoint $g:\cB\to \cB'$ such that $f\circ g\simeq 1_\cB$. 
\end{lem}
\begin{proof}
Note that $f$ induces a functor $f_!=(\cdot)\otimes_{\cB'}\cB :Mod(\cB')\to Mod(\cB)$, which is equivalent to $f$ itself on the Yoneda image of $\cB'$ and $\cB$. The restriction functor $f^*:Mod(\cB)\to Mod(\cB')$ is a natural right adjoint to $f_!$. On the other hand, since $\cB$ is proper, so are the perfect modules over it and $f^*$ sends these to proper modules over $\cB'$. As $\cB'$ is smooth, proper modules over it are perfect. In other words, $f^*$ maps $\cB\simeq \op{Perf}(\cB) $ to $\cB'\simeq \op{Perf}(\cB') $, and the restriction of $f^*$ to this subcategory is the desired right adjoint. 

By definition, the induced functor $\cB'/ker(f)\to \cB$ is a derived (Morita) equivalence, which implies $f^*:Mod(\cB)\to Mod(\cB')$ is fully faithful by \Cref{lemma:quotient-properties}\eqref{item:univ-prop-quotient}. It is also easy to verify that the counit $f_! f^* \to 1_{Mod(\cB)}$ is a quasi-isomorphism of functors (which follows either from the fully faithfulness of $f^*$, or \cite[Lemma 3.15]{gps1} which implies that $\cN\otimes_{\cB'} \cB\simeq \cN\otimes_\cB \cB\simeq \cN$ for every $\cN\in Mod(\cB)$ as $f$ is a localization). 
Clearly, this property is inherited by $f\dashv g$, i.e. $fg\to 1_\cB$ is a quasi-isomorphism.
\end{proof}
\begin{rk}
The adjunction $f \dashv g$ gives a semi-orthogonal decomposition $\cB'=\langle im(g),ker(f) \rangle\simeq \langle \cB,ker(f)\rangle$. In other words, every object $L'\in\cB'$ has a decomposition \begin{equation}
E\to L'\to gf(L')\to E[1]
\end{equation}
where $E\in ker(f)$ that is unique up to quasi-isomorphism.
\end{rk}
Observe that the fact that $im(g)\to \cB$ is a quasi-equivalence implies that $\phi' \circ g\simeq \phi$. Indeed, $\phi\circ f\simeq \phi'$ by assumption, and this holds when restricted to $im(g)$ as well. As $g$ is a quasi-inverse to $f|_{im(g)}$, one has $\phi'|_{im(g)}\circ g\simeq \phi$. But $\phi'|_{im(g)}\circ g=\phi'\circ g$. 

In particular, just like $f$ carries $ker(\phi')$ to $ker(\phi)$, $g$ carries $ker(\phi)$ to $ker(\phi')$. Concretely, in terms of the semi-orthogonal decomposition $\cB'\simeq \langle \cB,ker(f)\rangle$, this means that the $\cB$ component of an object $D\in ker(\phi')$ belongs to $ker(\phi)$. Observe that $ker(f)\subset ker(\phi')$. Therefore, the converse is also true: an object with $\cB$-component in $ker(\phi)$ is also in $ker(\phi')$. 
\begin{proof}[Proof of \Cref{prop:sodofcompactifications}]
Choose an essentially finite subset $\cD'\subset \cB'$ that split-generates $ker(\phi')$. By above, for any $D'\in\cD'$, its $\cB$ and $ker(f)$ components (i.e. $gf(D')$ and $cone(D'\to gf(D'))[-1]$) are in $ker(\phi')$; hence, without loss of generality, we may assume that these components are also in $\cD'$. Observe that the objects of $\cD'$ that are in $ker(f)$ split-generate $ker(f)$. Without loss of generality, assume $\cD'$ contains the whole quasi-isomorphism classes of its objects. 

Consider $f(\cD')$ (which is roughly the set of $\cB$-components of objects in $\cD'$), and extend this set by quasi-isomorphic objects. Call the new subset $\cD\subset\cB$. The purpose of these extensions (of $\cD$ and $\cD'$) by quasi-isomorphism classes is to ensure that $f(\cD')\subset \cD$ and $g(\cD)\subset \cD'$ hold simultaneously. As a result, one obtains induced functors $\ov f:\cB'/\cD'\to\cB/\cD$ and $\ov g:\cB/\cD\to\cB'/\cD'$ and they both respect the filtration. Clearly $f\circ g\simeq 1_\cB$ implies $\ov f\circ \ov g\simeq 1_{\cB/\cD}$. On the other hand, $g\circ f$ is the projection to the $\cB$ component, and its restriction to $im(g)\simeq \cB$ is homotopic to the identity. Therefore, if $im(g)/\cD'$ denotes the subcategory of $\cB'/\cD'$ consisting of objects of $im(g)$, then $\ov f|_{im(g)/\cD'}$ and $\ov g$ are quasi-inverses. Also, the inclusion $im(g)/\cD'\to \cB'/\cD'$ is essentially surjective because of the semiorthogonal decomposition $\langle im(g),ker(f) \rangle$, and $\cD'$ containing a subset that split-generates $ker(f)$. This implies $\ov f$ and $\ov g$ are quasi-inverses, which completes the proof
\end{proof}
\begin{rk}
One can presumably prove the quasi-inverses constructed in the proof of \Cref{prop:sodofcompactifications} are $E_r$-equivalences for some $r\gg 0$, under the assumption that $ker(f)$ has a strong generator (which for instance follows from $\cB'$ having a strong generator).  
\end{rk}
The following lets us extend \Cref{prop:sodofcompactifications} to the non pre-triangulated case:
\begin{lem}
If $\cB\to \cC$ is a smooth categorical compactification, resp. pseudo-compactification, then so are $\tw(\cB)\to \tw(\cC)$ and $\tw^\pi(\cB)\simeq \op{Perf}(\cB)\to \op{Perf}(\cC)\simeq \tw^\pi(\cC)$.
\end{lem} 
\begin{proof}
Fix a finite collection $\cD\subset \cB$ split-generating $ker(\phi)$. Then, the functors $\tw(\cB)/\cD\to \tw(\cB/\cD)$ and $\tw^\pi(\cB)/\cD\to \tw^\pi(\cB/\cD)$ are cohomologically fully faithful. Indeed, the former map is a quasi-equivalence by \Cref{lemma:quotient-properties}\eqref{item:quotient-can-map}, and the second claim follows from the first by noting that $\tw^\pi(\cB)/\cD$ is still obtained by adding some idempotents to $\tw(\cB)/\cD$, and $\tw^\pi(\cB/\cD)$ is the full idempotent closure of $\tw(\cB/\cD)$.
As a result, the compositions $\tw(\cB)/\cD\to \tw(\cB/\cD)\to \tw(\cC)$ and $\tw^\pi(\cB)/\cD\to \tw^\pi(\cB/\cD)\xrightarrow{\simeq } \tw^\pi(\cC)$ are fully faithful. Their images clearly split-generate their respective targets. Hence, these compositions are Morita equivalences.

To complete the proof it suffices to check $\cD$ split-generates the kernel in each case. For instance, if we denote the kernel of the former by $\tilde\cD$, then clearly $\cD\subset\tilde\cD$, and by above $\tilde\cD/\cD\to \tw(\cC)$ is fully faithful. By definition of the kernel, the image of this functor is quasi-equivalent to $0$. Hence, $\tilde\cD/\cD$ is the $0$ category itself, which by \Cref{lemma:quotient-properties}\eqref{item:split-generation-quotient} implies that $\tilde \cD$ is in the split-closed triangulated envelope of $\cD$. 
\end{proof}
Therefore, to apply \Cref{prop:sodofcompactifications} in the absence of pre-triangulated assumption, one can simply switch to $\tw(\cB)$, $\tw(\cB')$ and $\tw(\cC)$.

In the light of \Cref{prop:sodofcompactifications}, we consider two smooth compactifications $\cB_1$ and $\cB_2$ of $\cC$ to be equivalent if there is a smooth pseudo-compactification $\cB'$ with morphisms of smooth pseudo-compactifications $\cB_1\leftarrow \cB'\to \cB_2$. More generally, one can consider the equivalence relation generated by this, and as we will see, consider the growth function associated to an equivalence class of compactifications.   

\subsubsection{Growth functions from pseudo-compactifications}\label{subsubsection:growth-functions-pc}

\begin{defn}\label{definition:growth-function-general}
Given a pair of objects $K,L\in\cC$, and a smooth categorical compactification $(\cB,\phi)$ with an essentially finite subcategory $\cD\subset\cB$ split-generating $ker(\phi)$, define the \emph{growth function $\gr^\cB_{K,L}$} (or simply $\gr_{K,L}$) as follows: choose objects $\tilde K,\tilde L\in\cB$ such that $\phi(\tilde K)\simeq K$ and $\phi(\tilde L)\simeq L$. Then the chain complex $\cB/\cD(\tilde K,\tilde L)\simeq \cC(K,L)$ is naturally filtered (as described in \Cref{subsection:construction-LO-quotient}). Now set $\gr^\cB_{K,L}=\gr_{\cB/\cD(\tilde{K},\tilde{L})}$, where 
\eq \gr_{\cB/\cD({\tilde{K},\tilde{L}})} =\op{dim} \op{im} ( H(F^p\cB/\cD(\tilde K,\tilde L)) \to H(\cB/\cD(\tilde K,\tilde L)) ) \eeq is the cohomological growth function (see \Cref{definition:growth-function-chain}). 
\end{defn}
\begin{note}\label{note:idempotentextension}
Observe that \Cref{definition:growth-function-general} makes sense only when the objects $K,L$ actually lift to $\tw^\pi(\cB)$, which may not always be the case (see \Cref{lemma:quotient-properties}\eqref{item:quotient-can-map}). In practice, we will only work with compactifications and objects where the lifting holds; however, one can also extend the definition as follows: given $K,L\in\cC$, choose $K',L'$ that lift (say to $\tilde K,\tilde L$ as above) and such that $K$, resp. $L$ is a direct summand of $K'$, resp. $L'$. Concretely, this means that there are maps $a_K:K\to K'$, $b_K:K'\to K$ satisfying $b_K\circ a_K\simeq 1_K$, and similarly $a_L,b_L$. The construction above endows $H(\cC)(K',L')$ with a natural filtration. 
Define a filtration on $H(\cC)(K,L)$ by $F^pH(\cC)(K,L)=\{x:a_L\circ x\circ b_K\in F^pH(\cC)(K',L') \}$. Alternatively, one can use the identity $H(\cC)(K,L)=b_L\circ H(\cC)(K',L')\circ a_K$ and define $F^pH(\cC)(K,L):=b_L\circ F^pH(\cC)(K',L')\circ a_K$. This gives a translation equivalent filtration.  
Define the growth function by $\gamma^\cB_{K,L}(p):=dim(F^pH(\cC)(K,L))$.
\end{note}
%
%
%
We will mostly omit $\cB$ from the notation. As $\cB$ is assumed to be proper,  $\gr_{\cB/\cD({\tilde{K},\tilde{L}})}$ is finite. 

We will show in \Cref{lem:independencelift} and \Cref{corollary:indep-gen-set} below that $\gr^\B_{K,L}$ is well-defined up to equivalence.  Then, we will show in \Cref{corollary:growth-function-well-def} that  $\gr^\B_{K,L}$ actually only depends on $\cB$ up to zig-zag of smooth categorical compactifications. 
\begin{rk}
	In fact, the expression $\gr^\cB_{K,L}=\gr_{\cB/\cD(\tilde{K},\tilde{L})}$ should be interpreted \textit{up to scaling equivalence}, and the growth function $\gr^\cB_{K,L}$ will only be considered \textit{up to scaling equivalence}. In particular, it does not make sense to evaluate $\gr^\cB_{K,L}$; however, one can still talk about its asymptotic properties such as $\limsup \frac{\gr^\cB_{K,L}(p)}{p}$, and compare two growth functions asymptotically. 
\end{rk}

\begin{lem}\label{lem:independencelift}
With the notation of \Cref{definition:growth-function-general}, if one changes any of the lifts $\tilde K$, $\tilde L$, the growth function $\gr_{\cB/\cD({\tilde{K},\tilde{L}})}$ changes only by translation equivalence (see \Cref{definition:s-t-equiv-function}).
\end{lem}
\begin{proof}
If $\tilde L'$ is another lift of $L$, then in the quotient category $\cB/\cD$, $\tilde L$ and $\tilde L'$ become isomorphic. An isomorphism $a\in \cB/\cD(\tilde L,\tilde L')$ has finite length, say $p_0$; hence, composition by $a$ maps $F^pH(\cB/\cD)(\tilde K,\tilde L)$ to $F^{p+p_0}H(\cB/\cD)(\tilde K,\tilde L')$. As $a$ is an isomorphism, this map is injective and \begin{equation}
	\gr_{\cB/\cD(\tilde K,\tilde L)}(p)\leq  \gr_{\cB/\cD(\tilde K,\tilde L')}(p+p_0).
\end{equation}
The other inequality required for translation equivalence is proven similarly. One can also show (by exactly the same argument) that changing $\tilde K$ does not change the growth function up to translation equivalence. 
\end{proof}
\begin{cor}[of \Cref{prop:scalinginverse}]\label{corollary:indep-gen-set}
With the notation of \Cref{definition:growth-function-general}, if $\cD'$ is another essentially finite set split-generating $ker(\phi)$, then the growth functions $\gr_{\cB/\cD(K,L)}$ and $\gr_{\cB/\cD'(K,L)}$ associated to the quotients $\cB/\cD$ and $\cB/\cD'$ respectively are scaling equivalent.
\end{cor}
\begin{proof}
This follows from \Cref{prop:scalinginverse} (and \Cref{lemma:summand-e1}). Namely, first enlarge $\cD$ inside $\tw^\pi(\cD)$ by adding, for each object of $\cD'$, a quasi-isomorphic complex in $\tw^\pi(\cD)$ (assume we add the same complex for different but quasi-isomorphic objects of $\cD'$). Call this enlargement $\cD\subset \cD_1\subset \tw^\pi(\cD)\subset \tw^\pi(\cB)$, and observe it adds finitely many objects to $\cD$. By \Cref{prop:scalinginverse}, the growth functions  $\gr_{\cB/\cD(K,L)}=\gr_{\tw^\pi(\cB)/\cD(K,L)}$ and $\gr_{\tw^\pi(\cB)/\cD_1(K,L)}$ are the same up to scaling. Similarly, extend $\cD'$ inside $\tw^\pi(\cD')$ by adding representatives of each quasi-isomorphism class in $\cD$, and call this category $\cD'_1$. The same argument shows that the corresponding growth functions are the same up to scaling. 

Now consider $\cD''=\cD_1\cup \cD_1'\subset \tw^\pi(\cB)$. The quasi-isomorphism classes of objects in $\cD''$, $\cD_1$, and $\cD_1'$ are the same, which implies that $\tw^\pi(\cB)/\cD_1\to \tw^\pi(\cB)/\cD''$ and $\tw^\pi(\cB)/\cD'_1\to \tw^\pi(\cB)/\cD''$ are both $E_1$-equivalences by \Cref{lemma:summand-e1}. Hence, the corresponding growth functions are the same, finishing the proof. 
\end{proof}
\begin{note}\label{note:corollaryextendfinitetime}
An alternative proof of \Cref{corollary:indep-gen-set}, which is carried out entirely at the level of triangulated categories, is given in the appendix (cf.\ \Cref{corollary:appendix-conclusion}). Note that the proof of \Cref{corollary:indep-gen-set} (as well as \Cref{corollary:appendix-conclusion}) is still valid if one drops the assumption that $\cD$ and $\cD'$ are essentially finite, as long as there is some $l\geq 0$ such that every object of $\cD'$ is quasi-isomorphic to an object of $\tw_{\leq l}^\pi\cD$, and vice versa. 
\end{note}
\begin{cor}\label{corollary:growth-function-well-def}
If $\cB$ and $\cB'$ are two smooth categorical compactifications, such that there is a smooth categorical pseudo-compactification $\cB''$ and maps of pseudo-compactifications $f:\cB''\to\cB$ and $f':\cB''\to\cB$, then the growth functions $\gr_{K,L}^\cB$ and $\gr_{K,L}^{\cB'}$ are equivalent.	
\end{cor}
\begin{proof}
By replacing $\cC$, $\cB$, $\cB'$ and $\cB''$ by their split-closed pre-triangulated envelopes, assume without loss of generality that they are pre-triangulated. Choose a right adjoint $g$ to $f$ and $\cD\subset \cB$, $\cD''\subset \cB''$ as in the proof of \Cref{prop:sodofcompactifications}. 
Let $K,L\in\cC$, and choose lifts $\tilde K,\tilde L\in\cB''$. Then $f(\tilde K), f(\tilde L)\in \cB$ are lifts of $K,L$ to $\cB$, and can be used to define $\gr_{K,L}^\cB$. By \Cref{prop:sodofcompactifications}, $\cB''/\cD''(\tilde K,\tilde L)$ and $\cB/\cD(f(\tilde K),f(\tilde L))$ are equivalent, more precisely, there are quasi-isomorphisms
\begin{equation}
	\cB''/\cD''(\tilde K,\tilde L)\to \cB/\cD(f(\tilde K),f(\tilde L))\text{ and } \cB/\cD(f(\tilde K),f(\tilde L))\to \cB''/\cD''(gf(\tilde K),gf(\tilde L))
\end{equation} that respect filtration. As a result, 	
\begin{equation}\label{eq:grleqgrleqgr}
	\gr_{\cB''/\cD''(\tilde K,\tilde L)} \leq \gr_{\cB/\cD(f(\tilde K),f(\tilde L))}\text{ and } \gr_{\cB/\cD(f(\tilde K),f(\tilde L))}\leq \gr_{\cB''/\cD''(gf(\tilde K),gf(\tilde L))}
\end{equation}
In particular this proves $\gr_{K,L}^{\cB''}$ is finite (independently of the lifts $\tilde K,\tilde L$).  \Cref{lem:independencelift} implies $\gr_{\cB''/\cD''(gf(\tilde K),gf(\tilde L))}$ is translation equivalent to $\gr_{\cB''/\cD''(\tilde K,\tilde L)}$. This, combined with \eqref{eq:grleqgrleqgr}, implies that $\gr_{\cB/\cD(f(\tilde K),f(\tilde L))}$ and $\gr_{\cB''/\cD''(\tilde K,\tilde L)}$ are translation equivalent. Hence, $\gr_{K,L}^\cB$ and $\gr_{K,L}^{\cB''}$ are equivalent.

Analogously, $\gr_{K,L}^{\cB'}$ and $\gr_{K,L}^{\cB''}$ are equivalent. This finishes the proof that $\gr_{K,L}^\cB$ and $\gr_{K,L}^{\cB'}$ are equivalent.
\end{proof}

\begin{note}
We explained in \Cref{note:idempotentextension} how to extend the definition of $\gamma^\cB_{K,L}$ when $K,L$ does not lift, but are direct summands of objects that lift (which is always the case). This requires one more choice (the objects that lift and contain them as direct summands); however, one can show the independence of the growth function from this choice similarly to \Cref{lem:independencelift}. To see this, first observe that if $L\to L'$ and $L\to L''$ are maps with cohomological left inverses as in \Cref{note:idempotentextension} such that $L',L''$ lift to $\tw^\pi(\cB)$, then one can find another such lifting object $L'''$ and maps $L'\to L''',L''\to L'''$ such that the compositions $L\to L'\to L'''$ and $L\to L''\to L'''$ agree in cohomology. For instance, if $L'\simeq L\oplus L_0'$, $L''\simeq L\oplus L_0''$, then let $L'''=L\oplus L_0'\oplus L_0''\oplus L$. This object lifts as it is quasi-isomorphic to $L'\oplus L''$; however, we define the map $L''\to L'''$ differently. Namely, include $L'$ as $L\oplus L_0'\oplus 0\oplus 0\subset L'''$ and include $L''$ as $L\oplus 0\oplus L_0''\oplus 0$, so that the desired property holds. Now name the maps $a_L:L\to L'$, $a_{L'}:L'\to L'''$ (hence, the map $L\to L'''$ is $a_{L'}\circ a_L$), and call the left inverse of $a_{L'}$ by $b_{L'}:L'''\to L'$. For some $p_0,p_1$, $a_{L'}\in F^{p_0}H(\cC)(L',L''')$ and $b_{L'}\in F^{p_1}H(\cC)(L''',L')$. Thus, for $b_K:K'\to K$ as in \Cref{note:idempotentextension}
\begin{equation}
	\{x: a_L x b_K\in F^pH(\cC)(K',L') \}\subset 	\{x: a_{L'} a_L x b_K\in F^{p+p_0}H(\cC)(K',L''') \}
\end{equation}
\begin{equation}
\{x: a_{L'} a_L x b_K\in F^p H(\cC)(K',L''') \}\subset 	\{x: b_{L'} a_{L'} a_L x b_K\in F^{p+p_1}H(\cC)(K',L') \}\atop =\{x: a_L x b_K\in F^{p+p_1}H(\cC)(K',L') \}
\end{equation}
This shows the filtrations on $H(\cC)(K,L)$ using $L\to L'$ and $L\to L'''$ coincide up to translation. Same argument works to compare $L\to L''$ and $L\to L'''$ as well; thus, growth functions induced by $L\to L'$ and $L\to L''$ also coincide up to translation. 
\end{note}

\subsection{The localization filtration on derived categories of coherent sheaves}\label{subsection:filtration-dbcoh}
This section is an elaboration on \Cref{exmp:bsidecompactificationwithzigzag}. We prove:
\begin{prop}\label{prop:welldefinedalggeogrowth}
Given smooth algebraic variety $U$ over $\bK$, and $\scrF,\scrF'\in D^bCoh(U)$, the graded vector space $RHom_U(\scrF,\scrF')$ admits a filtration such that the associated growth function $\gr_{RHom_U(\scrF,\scrF')}$ is well-defined up to scaling equivalence.
\end{prop}
The term ``algebraic variety'' should be understood to mean a separated scheme of finite type over the field $\bK$ (which is assumed throughout this paper to be of characteristic zero).  Also by convention, we use the notation $D^bCoh(X)$ to denote an $A_\infty$-enhancement of the bounded derived category of coherent sheaves. Functors between such categories are also assumed to be $A_\infty$.

Note that one can actually prove that the filtration on the graded vector space $RHom_U(\scrF,\scrF')$ is well-defined up to scaling equivalence. Because of this, one can consider other invariants, such as the growth of each $RHom_U^i(\scrF,\scrF')$.
\subsubsection{Compactifications of algebraic varieties}
\begin{defn}
Let $U$ be a smooth, irreducible variety. A \emph{compactification} of $U$ is the data of a proper variety $X$ and an open embedding $i: U \hookrightarrow X$ such that $i(U)$ is Zariski dense. The compactification is said to be \emph{smooth} if $X$ is smooth.
\end{defn}

\begin{prop}\label{proposition:existence-ag-sc}
Let $U$ be a smooth variety. Then $U$ admits a smooth compactification.
\end{prop}
\pf
By Nagata's compactification theorem (see \cites{nagataimbedding, nagatageneralimbedding} and \cite[\href{https://stacks.math.columbia.edu/tag/0F3T}{Section 0F3T}]{stacks-project} for instance), $U$ embeds as an open subscheme of a proper variety $X$. Now apply Hironaka's resolution of singularities to $X$.
\epf
\begin{rk}
By blowing up with center $X\setminus U$, one can also assume $X\setminus U$ is a divisor. 	
\end{rk}
Smooth compactifications are unique up to blowups and blowdowns. This is a corollary of the so-called ``weak factorization theorem", which is due to Abramovich, Karu, Matsuki and Wlodarczyk. 
\begin{fact}[see Thm.\ 0.1.1 in \cite{akmw}]\label{fact:weak-factorization}
Let $\phi: X_1 \dashrightarrow X_2$ be a birational map between complete nonsingular algebraic varieties over an algebraically closed field of characteristic zero, and let $U \subset X_1$ be an open set where $\phi$ is an isomorphism. Then $\phi$ can be factored into a sequence of blowups and blowdowns with smooth irreducible centers disjoint from $U$.
\end{fact}

The following proposition relates this discussion to the purely categorical notions considered in the previous section.
\begin{prop}\label{proposition:ag-compactifications}
Let $i_1: U \to X_1$ and $i_2: U \to X_2$ be smooth compactifications of $U$, and let $f: X_1 \to X_2$ be a morphism such that $f\circ i_1=i_2$.
\begin{enumerate}
    \item the induced map $i_k^*: D^bCoh(X_k) \to D^bCoh(U)$ is a categorical compactification of $U$ (see  \Cref{definition:categorical-semi-comp}) with kernel spanned by objects whose set theoretical support is on $X_k\setminus U$; 
    \item suppose that $f: X_1 \to X_2$ is the composition of a sequence of blowdowns with smooth irreducible center disjoint from the image of $U$. Then $f_*: D^bCoh(X_1) \to D^bCoh(X_2)$ is a morphism of categorical compactifications.
\end{enumerate}
\end{prop}
\pf
The first claim is well known. For instance, it follows from \cite[Lemma 2.12]{arinkinbezrukavnikov} (also see the remark afterwards). Moreover, \cite[Lemma 2.13]{arinkinbezrukavnikov} shows that every object in the kernel is scheme theoretically supported on a thickening of the reduced scheme $X_k\setminus U$. As a result, it is in the split-closed triangulated envelope of the image of the pushforward $D^bCoh(X_k\setminus U)\to D^bCoh(X_k)$. The category $D^bCoh(X_k\setminus U)$ is finitely generated by \cite[Theorem 1.4]{ballardfaverokatzarkovorlovspectra} or \cite{neemanstrong}, which implies the same for the kernel of the map $D^bCoh(X_k) \to D^bCoh(U)$. See also \cite[Cor.\ 3.7]{ballard2008sheaves}.

It suffices to show the second claim for a single blowdown. It is clear that $i_2^*\circ f_*\simeq i_1^*$. According to a result of Orlov \cite[Sec.\ 4]{orlov_1993}, if $X_1\to X_2$ is a blowdown map with smooth center $Z\subset X_2$ of codimension $r$, then there exists a semi-orthogonal decomposition
\begin{equation}
D^bCoh(X_1)\simeq \langle D^bCoh(Z)_{-r+1},\dots, D^bCoh(Z)_{-1}, f^* D^bCoh(X_2) \rangle
\end{equation}
Here $D^bCoh(Z)_i$ are subcategories of $D^bCoh(X_1)$ which are equivalent to $D^bCoh(Z)$ \cite[Sec.\ 2]{orlov_1993} and generate the kernel of $f_*$ \cite[Sec.\ 4]{orlov_1993}. The claim follows. 

\epf
\begin{proof}[Proof of \Cref{prop:welldefinedalggeogrowth}]
Let $U$ be a smooth, irreducible variety. Choose a smooth compactification $U \hookrightarrow X$, which exists according to \Cref{proposition:existence-ag-sc}.This induces a smooth categorical compactification according to \Cref{proposition:ag-compactifications}(1). Therefore, by choosing lifts to $\scrF,\scrF'$, one can define a filtration on $RHom_U(\scrF,\scrF')$ from the localization. 

Any two such categorical compactifications are related by zigzags of morphisms of categorical compactifications by \Cref{fact:weak-factorization} and \Cref{proposition:ag-compactifications}(2). Therefore, as long as one uses this class of compactifications to define the growth, it does not depend on the compactification, and is well-defined up to scaling equivalence by \Cref{corollary:growth-function-well-def}. 
\end{proof}
\defi[Growth function for coherent sheaves] 
Given a pair of objects $\scrF, \scrF' \in D^bCoh(U)$, let $\gr_{\scrF,\scrF'}$ be the growth function $\gr_{RHom_U(\scrF,\scrF')}$, which is induced by a categorical compactification $D^bCoh(X)\to D^bCoh(U)$, as in \Cref{definition:growth-function-general}.
\edefi
Notice that this definition is not just a repetition of \Cref{definition:growth-function-general}, the key is the well-definiteness statement \Cref{prop:welldefinedalggeogrowth}.

\subsubsection{Divisor complements} 
In this section, we explicitly compute the growth functions (up to equivalence) when $X\setminus U$ is a divisor. Our main tool is \Cref{thm:spherical} (or \Cref{cor:spherical}). We obtain a particularly explicit description of the growth functions when  $X \setminus U$ is \textit{ample}; see \Cref{thm:growthisdimsupp} below. 

Let $X$ be a smooth projective variety and $D\subset X$ be a Cartier divisor such that $U=X\setminus D$. In particular, $D^bCoh(X)$ is a smooth categorical compactification of $D^bCoh(U)$, with kernel generated by the image of $D^bCoh(D)$ \cite[Cor.\ 3.17]{ballard2008sheaves}. By \Cref{exmp:sphericaldivisor}, the inclusion $D^b(Coh(D))\to D^b(Coh(X))$ is a spherical functor, with corresponding twist given by $S=(\cdot)\otimes\cO(D)$. Moreover, the natural transformation $1_{D^b(Coh(X))}\to S$ is induced by multiplication by a section $\sigma$ of $\cO(D)$ that cuts out $D$. Let $\scrF,\scrF'\in D^bCoh(U)$, and fix lifts $\ov\scrF,\ov\scrF'\in D^bCoh(X)$, which always exists by \cite[Lemma 2.12 (a)]{arinkinbezrukavnikov}. \Cref{cor:spherical} implies
\begin{equation}
RHom_U(\scrF,\scrF')\cong \colim_n RHom_X(\overline{\scrF},S^n(\overline{\scrF}'))\cong \colim_n RHom_X(\overline{\scrF},\overline{\scrF}'(nD))	
\end{equation}
where the localization filtration on $RHom_U(\scrF,\scrF')$ is identified with the filtration on the right hand side induced by the colimit. In this filtration, $F^p\colim_n RHom_X(\overline{\scrF},\overline{\scrF}'(nD))$ is the image of the map \eq RHom_X(\overline{\scrF},\overline{\scrF}'(pD))\to \colim_n RHom_X(\overline{\scrF},\overline{\scrF}'(nD))\eeq As a result,
\begin{cor}\label{cor:growthsheaves}
$\gr_{\scrF,\scrF'}(p)$ is equivalent to the function given by \eq\label{equation:growth-sheaves} p\mapsto dim(F^p\colim_n RHom_X(\overline{\scrF},\overline{\scrF}'(nD)))\eeq
\end{cor}

When $X \setminus U$ is an ample divisor, \eqref{equation:growth-sheaves} admits a rather explicit description, which is the content of the following theorem. 

\begin{thm}\label{thm:growthisdimsupp}
Assume $D=X\setminus U$ is ample and let $\scrF,\scrF'\in D^bCoh(U)$. Then the growth function associated to the pair $(\scrF,\scrF')$ is a polynomial of degree $d=dim(supp(\scrF^\vee\otimes \scrF'))=dim(supp(\scrF)\cap supp(\scrF'))$. In other words, $\gr_{\scrF,\scrF'}(p)$ is scaling equivalent to $p^d$. 
\end{thm}

The statement of \Cref{thm:growthisdimsupp} merits some explanations. For $\scrG \in D^bCoh(X)$, we have by definition
\eq supp(\scrG) = \bigcup_n supp(\mc{H}^n(\scrG)).\eeq
Also by definition, $\scrG^\vee= \mc{R}\sheafhom(\scrF, \mc{O}_X)$ is the derived dual. Finally, the tensor product in the statement of \Cref{thm:growthisdimsupp} is the derived tensor product. We note that is important that $X$ is a smooth scheme for these operations to be defined at the level of $D^bCoh(-)$. We refer to \cite[Sec.\ 3.3]{huybrechts2006fourier} for an introduction to these notions. 


We now begin the proof of \Cref{thm:growthisdimsupp}. First observe that $RHom_X (\ov\scrF,\ov\scrF'(nD))\cong R\Gamma (\ov\scrF^\vee\otimes \ov\scrF'(nD))$ and $RHom_U(\scrF,\scrF')\cong R\Gamma(\scrF^\vee\otimes \scrF')$; therefore, one only needs to compute the filtration growth of $\colim_n RHom_X (\ov\scrF^\vee\otimes \ov\scrF'(nD))$. Let $\scrE=\ov\scrF^\vee\otimes \ov\scrF'\in D^b(Coh(X))$. Therefore, it suffices to prove:
\begin{prop}\label{prop:growthofsheafdimsupp}
	The growth of $\colim_n R\Gamma(\scrE(nD))$ is a polynomial of degree $dim(supp(\scrE|_U))$. 
\end{prop}
For $n\gg 0$, for every $q$, $\mathcal{H}^q(\mathcal{E}(nD))= \mathcal{H}^q(\mathcal{E})(nD)$ has cohomology concentrated in degree zero by Serre vanishing. It follows that $R^q\Gamma( \mathcal{E}(nD)) \cong \Gamma( \mathcal{H}^q( \mathcal{E}(nD)))$ by the hypercohomology spectral sequence. 

Therefore, the growth function of the filtered vector space $\colim_n R\Gamma (\scrE(nD))$ is the sum of the growth functions of $\colim_n \Gamma(\cH^q(\scrE)(nD))$, and it suffices to prove \Cref{prop:growthofsheafdimsupp} for the coherent sheaf $\cH^q(\scrE)$. In other words, without loss of generality, we can assume $\scrE$ is a coherent sheaf. 
To prove \Cref{prop:growthofsheafdimsupp}, we will relate the growth function of $\colim_n \Gamma(\scrE(nD))$ to the Euler-Poincar\'e polynomial $\chi(\scrE(nD))$ and apply the following:
\begin{lem}\cite[\S81 Prop.6]{serrefaisceaux}\label{lem:chipolynomial}
Let $\scrE$ be a coherent sheaf on $X$, and let $d=dim(supp(\scrE))$. Assume $D$ is ample. Then, the Euler-Poincar\'e characteristic $\chi(\scrE(nD))$ is a polynomial of $n$ of degree $d$. 
\end{lem}
\begin{note}
%
\cite[\S81 Prop.6]{serrefaisceaux} proves this statement when $X=\bP^n$ for some $n$. On the other hand, if $D$ is very ample and we consider the embedding into a projective space $\iota:X\to\bP^n$ determined by $D$, then the Hilbert polynomial of $\iota_*\scrE$ is the same as $\chi(\scrE(nD))$. This concludes the proof in the very ample case. Without the very ample assumption, one notes that $\chi(\scrE(nD))$ is still a polynomial of degree at most $d$, by \cite{snapperpolynomial} and \cite{kleimanampleness}. On the other hand, for $\ell\gg 0$ such that $\ell D$ is very ample, $n\mapsto \chi(\scrE(\ell nD))$ is of degree $d$ by above; thus, so is $n\mapsto \chi(\scrE(nD))$.
\end{note}
\begin{proof}[Proof of \Cref{prop:growthofsheafdimsupp}]
First, observe that $R\Gamma(\scrE(nD))=\Gamma(\scrE(nD))$ for $n\gg 0$ as $D$ is ample. 
By replacing $\scrE$ with $\scrE(n_0D)$, $n_0\gg 0$, we can assume this holds for every $n\geq 0$. Hence, we actually consider the growth of $\colim_n \Gamma(\scrE(nD))\cong \Gamma(\scrE|_U)$. This is not the same as the growth of $n \mapsto dim(\Gamma(\scrE(nD)))=\chi (\Gamma(\scrE(nD)))$; however, the kernel of $\Gamma(\scrE(nD))\to \colim_n\Gamma(\scrE(nD))$ is given by $\Gamma(\scrE_D(nD))$, where $\scrE_D$ denote the subsheaf of $\scrE$ of sections set theoretically supported on $D$. Indeed, this is automatic as one defines $\scrE_D$ to be the sheaf of sections annihilated by a power of a function locally cutting out $D$ (or by a power of $\sigma\in \Gamma(\cO(D))$). 

Hence, the dimension of $F^p\colim_n \Gamma(\scrE(nD))$ is given by $dim(\Gamma(\scrE(pD)))-dim(\Gamma(\scrE_D(pD)))$, which is the same as $dim(\Gamma((\scrE/\scrE_D)(pD)))$, for $p\gg 0$, as $R^1\Gamma(\scrE_D(pD))=0$. Similar to before, for $p\gg 0$, $dim(\Gamma((\scrE/\scrE_D)(pD)))=\chi ((\scrE/\scrE_D)(pD))$; therefore, by applying \Cref{lem:chipolynomial} to $\scrE/\scrE_D$, we see this dimension is a polynomial in $p$ of degree $dim(supp(\scrE/\scrE_D))$. To conclude the proof, we need to show that $dim(supp(\scrE/\scrE_D))=dim(supp(\scrE|_U))$. For this purpose, it suffices to show $supp(\scrE/\scrE_D)$ has no irreducible components contained entirely in $D$. If this holds, since $D=X\setminus U$, one has $dim(supp(\scrE/\scrE_D))=dim(supp(\scrE|_U/\scrE_D|_U))$. As $\scrE_D|_U=0$, this is the same as $dim(supp(\scrE|_U))$, which finishes the proof.

To see $supp(\scrE/\scrE_D)$ has no irreducible components contained entirely in $D$ is standard: the ideal sheaf of $supp(\scrE/\scrE_D)$ is given by the annihilator of $\scrE/\scrE_D$. Assume $supp(\scrE/\scrE_D)$ has a component $V$ contained entirely on $D$. On an affine chart $Spec(A)\subset X$, $\scrE/\scrE_D$ is represented by an $A$ module $M$. Consider a function $g\in A$ that vanishes on other irreducible components of $supp(\scrE/\scrE_D)$ but not on $V$. Then $gM\subset M$ is a submodule that is non-zero as $g$ does not vanish on $V$. Any function vanishing on $V\cap Spec(A)$ has a power killing $gM$. In other words, the radical of the annihilator of $gM$ contains the ideal of $V\cap Spec(A)$; hence, $gM$ has set theoretic support on $V$. In particular, a defining function of $D$ has a power that acts as $0$ on $gM$. On the other hand, such a function acts injectively on local sections of $\scrE/\scrE_D$; hence, on $M$ and $gM$. This implies $gM=0$ and this contradiction shows that no such irreducible component $V$ can exist. 
\end{proof}
As remarked, \Cref{thm:growthisdimsupp} follows from \Cref{prop:growthofsheafdimsupp} immediately by letting $\scrE=\ov\scrF^\vee\otimes \ov\scrF$. To conclude, one only has to note that the support of $\scrF^\vee\otimes \scrF$ is topologically the same as $supp\scrF\cap supp\scrF'$. This can be seen as follows: First, we have $supp(\scrF^\vee \otimes \scrF')= supp(\scrF^\vee) \cap supp(\scrF')$ (indeed, let $i: \{*\} \hookrightarrow X$ be the inclusion of a point and note that $i^*(\scrF^\vee \otimes \scr F')= i^*\scrF^\vee \otimes i^*\scrF'$). However,  by \cite[Lem.\ 3.32]{huybrechts2006fourier}, we also have $supp(\scrF)= supp(\scrF^\vee)$. 

We highlight the following corollaries of \Cref{thm:growthisdimsupp}.

\begin{cor}
The growth functions for pairs of objects $(L,L')$ on $D^b(U)$ (whose categorical compactification is given by by a pair $(X,D)$ as above) satisfy triangle inequality (up to scaling equivalence), i.e. if one has an exact triangle $L\to L'\to L''\to L[1]$, then $\gr_{K,L''}\leq\gr_{K,L}+\gr_{K,L'}$.  
\end{cor}
In particular, one can define a global growth invariant for this category, as the growth function of a generator. Up to a constant, it dominates all growth functions $\gr_{K,L}$. 

\begin{cor}
The growth $\gr_{K,L}$ is either constant, or at least linear. 	
\end{cor}

\subsection{The localization filtration on wrapped Fukaya categories}\label{subsection:filtration-wfuk}
This section is an elaboration on \Cref{exmp:asidecompactificationwithzigzag}. We prove
\begin{thm}\label{thm:welldefinedsymplecticgrowth}
Given a Weinstein manifold $M$ and a pair of objects $K, L \in \tw^\pi \mc{W}(M)$, the graded vector space $H(\cW(M))(K,L)$ admits a filtration such that the associated growth function $\gr_{H(\cW(M))(K,L)}$ is well-defined up to scaling equivalence.	
\end{thm}
Similar to previous section, one can prove that the filtration on the graded vector space $H(\cW(M))(K,L)$ is well-defined up to scaling equivalence. Assuming $\cW(M)$ is $\bZ$-graded, one can consider other invariants, such as the growth of each $H^i(\cW(M))(K,L)$.

The following lemma is our main source of smooth pseudo-compactifications on wrapped Fukaya categories of Weinstein manifolds. 

\begin{lem}\label{lemma:tame-induces-compactification}
Suppose that $(M, \l)$ is Weinstein and let $\fk{c} \sub \d_\infty M$ be a tame stop. Then 
\begin{equation}
\tw^\pi \mc{W}(M, \fk{c}) \to \tw^\pi \mc{W}(M)
\end{equation}
is a smooth categorical pseudo-compactification. If $\tw^\pi \mc{W}(M, \fk{c})$ is proper, then it is a categorical compactification.
\end{lem}

\pf
Note first that $\tw^\pi \mc{W}(M)$ and $\tw^\pi \mc{W}(M, \fk{c})$ are smooth by \Cref{fact:smoothness}. We need to check the conditions (i) and (ii) of \Cref{definition:categorical-semi-comp}.

Since $\fk{c}$ is tame, there exists a decomposition $\fk{c}= \fk{c}^{\op{crit}} \cup \fk{c}^{\op{subscrit}}$ where $\fk{c}^{\op{crit}}$ has finitely many components. It follows from \Cref{fact:stop-removal} that $\mc{W}(M, \fk{c})/\mc{D} \to \mc{W}(M)$ is a quasi-equivalence, where $\cD$ is the full subcategory of linking discs, which verifies (ii). To check (i), let $\mc{D}'$ consist of one linking disk for each component of $\fk{c}^{\op{crit}}$. Then $\mc{W}(M, \fk{c})/\mc{D}'=\mc{W}(M, \fk{c})/\mc{D}$, so (i) follows from \Cref{lemma:quotient-properties}\eqref{item:split-generation-quotient}.
\epf

There is also a uniqueness statement. 
\begin{prop}\label{proposition:weinstein-zigzag}
Let $(M, \l)$ be Weinstein and let $\fk{c}_1, \fk{c}_2 \sub \d_\infty M$ be tame stops. Then $\tw^\pi \mc{W}(M, \fk{c}_i) \to \tw^\pi \mc{W}(M)$ are related by a zigzag of morphisms of smooth pseudo-compactifications.
\end{prop}
\pf
An isotopy of stops induces a morphism of smooth pseudo-compactifications according to \Cref{lemma:isotopy-fuk-map}. We may therefore assume that  $\fk{c}_1, \fk{c}_2$ are disjoint. Now $\fk{c}_1 \cup \fk{c}_2$ is a tame stop, and it follows by \Cref{fact:stop-removal} that $\tw^\pi \mc{W}(M, \fk{c}_1 \cup \fk{c}_2) \to \tw^\pi \mc{W}(M, \fk{c}_i)$ is a morphism of smooth pseudo-compactifications.
\epf

\begin{defn}\label{definition:full-stop}
Let $(M, \l)$ be Weinstein. A tame stop $\fk{c} \sub \d_\infty M$ is said to be a \emph{full stop} if $\mc{W}(M, \fk{c})$ is proper (and a fortiori smooth). 
\end{defn}

\begin{lem}\label{lemma:full-stop}
Any Weinstein manifold $(M, \l)$ admits a full stop. 
\end{lem}

\pf
It follows from the main result of Giroux--Pardon \cite{giroux-pardon} that $M$ admits a Lefschetz fibration with Weinstein fibers. In particular, $\d_\infty M$ admits an open book decomposition with Weinstein pages (alternatively, this later fact follows from unpublished work of Giroux or recent work of Honda--Huang \cite[Cor.\ 1.3.1]{honda-huang}). Fix such an open book on $(\d_\infty M, \xi_\infty)$ and let $F \sub \d_\infty M$ be any page. 

As explained in \cite[Ex.\ 2.19]{gps1}, the complement of $F \tms [-\e, \e]$ can be deformed though codimension zero submanifolds with corners to a contactization. Hence $\mc{W}(M, F)$ is proper according to \cite[Lem.\ 3.44]{gps1}. 

After possibly deforming $F$ through Liouville hypersurfaces (which does not affect $\mc{W}(M, F)$), we may assume that the Liouville vector field of $(F, \l|_F)$ is gradient-like for a proper Morse function (here $\l$ denotes a contact form for $\xi_\infty$ defined near $F$). We may also assume that the cocores are properly embedded. There is then a quasi-equivalence $\mc{W}(M,F) \simeq \mc{W}(M, \fk{c})$ where $\fk{c} \sub F$ is the (mostly Legendrian) skeleton of $(F, \l|_F)$ \cite[Cor.\ 3.9]{gps2}. Smoothness of $\mc{W}(M, \fk{c})$ follows from \Cref{fact:smoothness}, and the fact that $\fk{c}$ is tame is \Cref{example:tame-construction}. 
\epf
\begin{proof}[Proof of \Cref{thm:welldefinedsymplecticgrowth}]
Choose a full stop $\fk{c} \sub \d_\infty M$, which exists according to \Cref{lemma:full-stop}.
This induces a smooth categorical compactification of $\tw^\pi \mc{W}(M)$ according to \Cref{lemma:tame-induces-compactification}. Hence, by choosing lifts, we obtain a corresponding localization filtration on $H(\cW(M))(K,L)$. By \Cref{proposition:weinstein-zigzag}, any such smooth categorical compactifications are related by a zigzag through a pseudo-compactification. It follows from \Cref{corollary:growth-function-well-def} that the growth function is well-defined, i.e. it does not depend on the choice of full stop.
\end{proof}
We now come to the main definition of this section.

\defi[Growth function for wrapped Fukaya categories]\label{definition:growth-fun-wrapped-fuk} 
Given Weinstein manifold $M$ and a pair of objects $K, L \in \tw^\pi \mc{W}(M)$, let $\gr_{K, L}$ be the induced growth function, corresponding to a smooth categorical compactification $\cW(M,\fk{c})$ as above (defined as in \Cref{definition:growth-function-general}). 
\edefi

\section{Comparison with Hamiltonian filtrations}\label{section:ham-comparison}

\subsection{Growth functions from iterated Hamiltonians}

Let $(M, \l)$ be a Liouville manifold. Choose objects $K, L \in \mc{W}(M, \l)$ and let $H: M \to \R$ be a (time-independent) cylindrical Hamiltonian which is linear at infinity. We let $\phi_H$ denote the time-$1$ flow of $H$. Building on earlier work of Seidel \cite[Sec.\ 4]{seidel-biased-view} and McLean \cite{mclean-gafa} in the context of symplectic cohomology, McLean proved the following:
\begin{prop}[\cite{mclean}]\label{proposition:iterated-hamiltonian-mclean}
The filtered directed system $\{ HF^\bu(\phi_{nH}K, L) \}_{n \in \mathbb{N}_+}$ is independent of $H$ up to weak isomorphism (see \Cref{subsection:filtered-ds}).
\end{prop}
\pf
This is essentially the content of \cite[Lem.\ 2.5 and 2.6]{mclean}. A more direct argument in the context of symplectic cohomology (which adapts immediately to the setting of wrapped Floer homology) is in \cite[(4a)]{seidel-biased-view}.
\epf

On the other hand, we have by definition 
\eq\label{equation:colimit-hw} \colim_n HF(\phi_{nH} K, L) = HW(K, L)\cong H(\cW(M,\fc))(K,L).\eeq
As explained in \Cref{subsection:filtered-ds}, \eqref{equation:colimit-hw} induces a filtration on $HW(K, L)$. We let $\gr_{K, L}^{ham}$ be the associated growth function (\Cref{definition:growth-function-chain}).
It follows from \Cref{proposition:iterated-hamiltonian-mclean} and \Cref{lemma:fds-equivalence} that $\gr_{K, L}^{ham}$  is well defined up to scaling equivalence.

The same considerations apply for symplectic cohomology: we let $\gr_{SH}^{ham}$ be the growth function obtained by the equality $\colim_n HF(M; nH)= SH(M)$. It follows from \cite[Sec.\ 4]{seidel-biased-view} that this is well defined up to scaling equivalence.

The purpose of this section is to prove the following theorem, which relates the notion of iterated Hamiltonian growth in \cite{mclean} with the growth defined via the localization filtration. 

\begin{thm}\label{theorem:tensor-equiv-hamiltonian}
With the notation as above, the growth functions $\gr_{K, L}^{ham}$ and $\gr_{K, L}$ are scaling equivalent (where $\gr_{K, L}$ is defined in \Cref{definition:growth-fun-wrapped-fuk}). 
\end{thm}
One can attempt to prove \Cref{theorem:tensor-equiv-hamiltonian} as follows: choose a Lefschetz fibration structure on $M$ and let $F\subset\partial_\infty M$ be a fiber (pushed to infinity). Assume without loss of generality that $K,L$ have boundary that does not intersect $F$. Then, by \cite{sylvan-orlov}, one obtains a spherical functor $\cW(F)\to \cW(M,F)\simeq \cW(M,\fc_F)$ (where $\fc_F$ is the core of $F$). The corresponding twist functor $S$ is ``wrap once negatively'' functor, roughly does the same thing as $\phi_{-H}$. Therefore, one wants to use this to prove that the directed systems $\{HF(\phi_{nH}(K),L))\}$ and $\{HF(K,S^n(L)\}$ are weakly isomorphic, and define the same growth function. Combining this with \Cref{thm:spherical} would prove \Cref{theorem:tensor-equiv-hamiltonian}.

Implementation of this idea runs into technical problems. First and foremost is the lack of a Hamiltonian that rotates the pages of the open book decomposition on $N:=\partial_\infty M$ and brings $F$ back to itself. This prevents us from relating $S$ to some $\phi_{H}$ directly. 

We will overcome this difficulty, by definining a Hamiltonian $H$ such that $\phi_H$ does not exactly rotate pages, but carries a small neighborhood of the page into itself. After this, instead of relating $\phi_H$ to an inverse spherical twist given in \cite{sylvan-orlov}, defining the transformation $1\to \phi_{-H}$, and applying \Cref{thm:spherical}, we take a more direct approach. We use the non-functorial version \Cref{prop:colimisloc} of \Cref{thm:spherical} and apply the stop doubling trick inspired by \cite{gps3} and \cite{sylvan-orlov}.

The proof of \Cref{theorem:tensor-equiv-hamiltonian} will occupy most of the remainder of \Cref{section:ham-comparison}. 

\subsection{The contact mapping torus}
In this section, we briefly remind construction and basic properties of the contact mapping torus, which describes complement of a binding on an open book decomposition on $N=\partial_\infty M$. 

Let $(W, \l)$ be a Liouville domain with Liouville vector field $Z_\l$. Let $\phi: W \to W$ be a symplectomorphism having the following properties:
\begin{itemize}
\item $\phi$ is the identity near $\d W$; 
\item $\phi^*\l = \l + df$, for some $f: W \to \R$ which is constant near $\d W$. 
\end{itemize}

We now consider the \emph{contact mapping torus} \cite[Sec.\ 2.11]{van-koert}
\eq \mc{T}:= (\R \tms W, dt + \l)/ \sim \eeq
where $(t, x) \sim (t - f(x), \phi(x))$. We let $\pi_{\mc{T}}: \R \tms W \to \mc{T}$ denote the quotient map. After possibly increasing $f$ by an additive constant, we may assume that $f>0$ and it is then straightforward to verify that $\mc{T}$ inherits the structure of a contact manifold with boundary diffeomorphic to $S^1 \tms \d W$. See \Cref{figure:contacttoruswiththreeregions}. Under $\sim$, the sets $t=0$ and $t=-f(x)$ are identified and correspond to a page of an open book. 

Let us set $A_0= \op{max}_{x \in W} f - \op{min}_{x \in W} f$ and $A:= A_0+2$. 

Let us also set $C \in \R_+$ to be the value of $f$ near $\d W$. It will be convenient to view $C$ as a free parameter which can be chosen to be arbitrarily large and will be fixed later. More precisely, the plan in the next section is to complete $\mc{T}$ to an open book decomposition by gluing on $B \tms D^2$, where $B= \d W$. The procedure for doing this is explained in full detail in \cite[Sec.\ 4.4.2]{geiges2008introduction} so will not be repeated here. We remark that replacing $f$ by $f+c$, $c>0$ (hence, $C$ by $C+c$) does not change the contactomorphism type of the closed manifold $\mc{T} \cup (B \tms D^2)$ by Gray stability. Clearly, $A_0$ and $A$ are also unchanged by this replacement.


\begin{lem}\label{lemma:projection-properties}
After possibly making $C$ larger (depending on $A$), the following properties hold for all $t \in \R$:
\begin{itemize}
\item the quotient map $\pi_{\mc{T}}$ restricts to an embedding on $[-t-A-1, -t+A+1] \tms W \sub \R \tms W$; 
\item we have $\pi_{\mc{T}}([-1,1] \tms W) \sub \pi_{\mc{T}} ([-C-A, -C+A] \tms W)$.
\end{itemize}
\end{lem}
\begin{proof}
The quotient map restricts to an embedding on any interval of length less than $C-A_0\leq \op{min}_{x\in W}	f$, implying the first property. To see the second, we use $(t, x) \sim (t - f(x), \phi(x))$, namely if $(t,x)\in [-1,1] \tms W$, then \eq -C-A\leq  -1-\op{max}_{x\in W} f \leq t-f(x)\leq 1-\op{min}_{x\in W}	f \leq -C+A  \eeq 
\end{proof}
\begin{figure}
	\centering
	\includegraphics[height=9.5 cm]{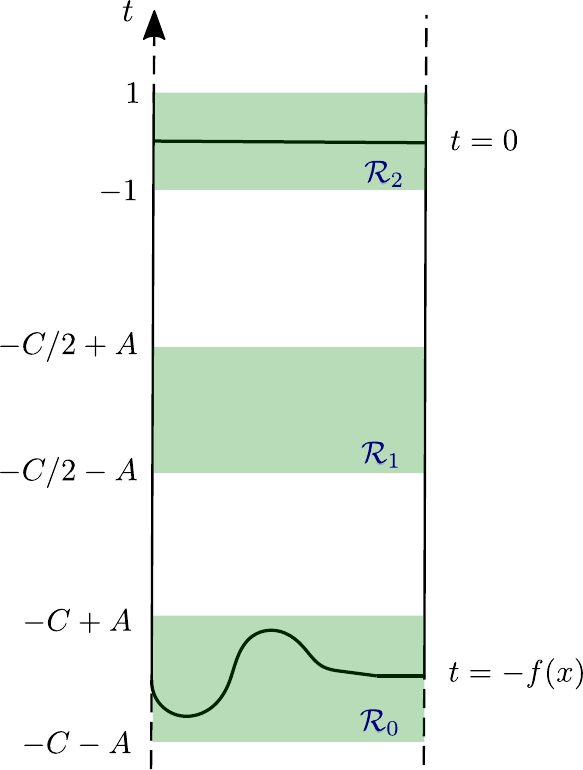}
	\caption{The contact mapping torus.}
	\label{figure:contacttoruswiththreeregions}
\end{figure}
We now begin the construction of a time-dependent Hamiltonian vector field 
which will be used later to isotope Lagrangian submanifolds. We start by constructing it on $\cT$, then we will extend it to $\cT\subset N=\partial_\infty M$, and finally cylindrically to $M$. This vector field is designed to have the property that its time-$C$ flow takes the set $\pi_\mc{T}([-C-A, -C+A] \tms W)$ into itself, except near the boundary. More precisely, let $W'\subset W$ denote a slightly smaller subdomain, and name the regions $\cR_0:=\pi_\mc{T}([-C-A, -C+A] \tms W')$, $\cR_1:=\pi_\mc{T}([-C/2-A, -C/2+A] \tms W')$, and $\cR_2:=\pi_\mc{T}([-1,1] \tms W')$. These regions are highlighted in green in \Cref{figure:contacttoruswiththreeregions}. 
\begin{prop}\label{prop:hamiltgood}
	There exists a time dependent $C$-periodic, positive Hamiltonian $\{G_s\}$ on $\cT$ such that $9/10\leq G_s\leq 11/10$ and satisfying
\vspace{-0.2cm}
\begin{enumerate}[label=(\alph*)]
	\item\label{item:ham1near} $G_s=1$ outside a small neighborhood of the compact subset $\pi_{\mc{T}}(\bR\times W') \subset \cT$ and near $\pi_{\mc{T}}(\bR\times \partial W) \subset \cT$
	\item\label{item:hamilthalfflow} The time $C/2$-flow of $\{G_s\}$ takes the set $\cR_0$ into $\cR_1$ and $\cR_1$ into $\cR_2$
	\item\label{item:hamiltotherhalf} The flow of $\{G_s\}_{s=C/2}^C$ takes $\cR_0$ into $\cR_1$ and $\cR_1$ into $\cR_2$ 
\end{enumerate}
\end{prop}
In particular, the time $C$-flow sends $\cR_0$ into $\cR_2$ and by \Cref{lemma:projection-properties}, $\cR_2\subset \cR_0$. Therefore, the time $C$-flow sends $\cR_0$ into $\cR_0$. Similarly, it sends $\cR_1$ into $\cR_1$.
\begin{proof}
First fix a sufficiently small $\e>0$ so that $\e A<1/100$. For notational simplicity, we start at negative time $s=-C$. For $s\in [-C,0]$ consider the contact Hamiltonian $1+\e(s-t)$ where the corresponding Hamiltonian vector field can be computed to be 
\begin{equation}\label{eq:initialvf}
\partial_t+\e ((s-t)\partial_t -Z_\lambda)	
\end{equation}
Notice that at each time $s\in[-C,0]$, $\e (s-t)\partial_t$ factor compresses the set $[s-a,s+a]\times W$ from bottom and the top, while $\partial_t$ moves it upwards (in \Cref{figure:contacttoruswiththreeregions}). Similarly, $\e Z_\l$ factor pushes the $W$ component inside. In particular, the $s$-dependent vector field $\partial_t+\e (s-t)\partial_t$ carries the interval $[-C-A,-C+A]$ to $[s-Ae^{-\e(s+C)},s+Ae^{-\e(s+C)}]$ at time $s\in[-C,0]$. 

The flow of this vector field from time $s=-C$ to $s=-C/2$ sends $[-C-A,-C+A]$ into $[-C/2-A,-C/2+A]$. Moreover, as mentioned, we can assume $C$ is large, and guarantee the flow of this vector field (from time $s=-C$ to $s=-C/2$) compresses what is initially $[-C/2-A,-C/2+A]$ and sends it into $[-1,1]$. 

Similarly, for a function $\e(x)\in C^\infty(W)$ such that $\e(x)$ is $0$ on a small neighborhood of $\partial W$, constant and positive on a small neighborhood of $W'$, $Z_\lambda(\e)\leq 0$, and $0\leq \e(x)<1/100$, we can consider the contact Hamiltonian $1+\e(x)(s-t)$ whose corresponding Hamiltonian vector field is given by 
\begin{equation}\label{eq:secondvf}
	\partial_t+\e(x) ((s-t)\partial_t -Z_\lambda)+(s-t)X^W_\e-(s-t)Z_\lambda(\epsilon)\partial_t
\end{equation}
where $X^W_\e$ is the Hamiltonian vector field on $W$, i.e. the vector field satisfying $\iota_{X^W_\e}d\lambda=-d\epsilon$. Observe that this vector field matches \eqref{eq:initialvf} on a small neighborhood of $\bR\times W'$ where $\e$ is constant, and it is equal to the Reeb vector field $\partial_t$ on a small neighborhood of $\bR\times \partial W$. Moreover $Z_\lambda(\epsilon)\leq 0$ implies that the $\partial_t$ component $\partial_t+(\e(x)-Z_\lambda(\epsilon)) (s-t)\partial_t $ of \eqref{eq:secondvf} still has the same compressing effect on rectangles $[s-a,s+a]\times W$. 

For simplicity, we first explain how to define an Hamiltonian as in \Cref{prop:hamiltgood} that carries $\cR_0$ into $\cR_1$ at time $C$ (instead of the stronger conditions \ref{item:hamilthalfflow} and \ref{item:hamiltotherhalf}). For $s\in[-C,0]$, we want the Hamiltonian $G_s$ to be equal to $1+\e(x)(s-t)$ near a small neighborhood of $[s-Ae^{-\e(s+C)},s+Ae^{-\e(s+C)}]\times W$, and $1$ outside a slightly larger neighborhood. The corresponding Hamiltonian vector field compresses $[s-Ae^{-\e(s+C)},s+Ae^{-\e(s+C)}]\times W$ at time $s$, and the speed of compression over $[s-Ae^{-\e(s+C)},s+Ae^{-\e(s+C)}]\times W'$ is uniform (as $\e$ is constant on $W'$). Therefore, as before, at time $s$, $[-C-A,-C+A]\times W'$ is carried into $[s-Ae^{-\e(s+C)},s+Ae^{-\e(s+C)}]\times W'$. In particular, it is carried into $\cR_2$ at time $C$. It is easy to see $G_s$ descends onto $\cT$ (or rather, can be extended to $\bR\times W$ so that it descends). One can extend it to $s\in [0,C]$ so that it becomes periodic. Notice, however, this can be discontinuous in $s$. Therefore, we smooth the function in a very small neighborhood of $s=0$ without effecting the condition that $\cR_0$ maps into $\cR_2$ at $s=0$ (we start the flow at $s=-C$). We modify $G_s$ near $s=-C$ and $s=C$ and extend to all $s\in\bR$ periodically. The conditions $9/10\leq G_2\leq 11/10$ and \ref{item:ham1near} can still be assumed to be true.

The general case is similar: we first define it for $s\in[-C,-C/2]$ as \eqref{eq:secondvf} on a small neighborhood of  $[s-Ae^{-\e(s+C)},s+Ae^{-\e(s+C)}]$ and $[C/2+s-Ae^{-\e(s+C)},C/2+s+Ae^{-\e(s+C)}]$. By possibly making $C$ larger, one can guarantee the condition \ref{item:hamilthalfflow} is satisfied, while keeping $9/10\leq G_2\leq 11/10$ and \ref{item:ham1near}. Then one defines a similar Hamiltonian for $s\in[-C/2,0]$, removes discontinuity at $s=-C/2$ and apply the procedure above to extend periodically. 
\end{proof}

\subsection{Growth rates on Liouville manifolds}

Let $(M, \l)$ be a Liouville manifold. Suppose now that the ideal boundary $(N:=\d_\infty M, \xi_\infty)$ is equipped with an open-book decomposition (for instance, a Lefschetz fibration structure on $M$ would equip $N$ with this structure). For an appropriate choice of contact form $\xi_\infty = \op{ker} \a$, we can assume that there is a \emph{strict} contact embedding 
\begin{equation}
\mc{T} \hookrightarrow (N, \a),
\end{equation} 
where $\mc{T}$ is the mapping torus constructed in the previous section (for some Liouville domain $(W, \l)$ and compactly supported symplectomorphism $\phi: W \to W$ depending on $(M, \l)$). We emphasize that we are free to make the parameter $C \in \R_+$ considered in the previous section arbitrarily large.

In \Cref{prop:hamiltgood} we have defined an Hamiltonian function $G_s$. Since $G_s=1$ near the boundary $\partial\cT=\pi_{\mc{T}}(\partial W\times \bR)$, we can extend it to all $N$ by $1$. Let $\{H_s\}$ be a $C$-periodic time dependent Hamiltonian on $M$ which agrees outside a compact set with the cylindrical lift of $H_s$ (i.e. $H_s=rG_s$ on the cylindrical end, where $r$ is the Liouville parameter).

%
%
%
%
The following lemma follows from standard properties of continuation maps (cf.\ \cite[(4a)]{seidel-biased-view} and \cite[Sec.\ 2.3]{mclean}).
\begin{lem}\label{lemma:positivity-continuation}
For $u \in [0,1]$, let $F_u: [0,C]_s \tms N \to \R$ be a family of non-negative, time-dependent contact Hamiltonians on $(N, \a)$. Let $\tilde{F}_u$ be a family of Hamiltonians on $M$ which agree outside a fixed compact set with the cylindrical lift of $F_u$, and for a fixed $u$ let $\phi_{\tilde F_u^s}$ denote the time $s$-flow of $\tilde F_u$. 
Suppose that $\d_uF_u \geq 0$. Then given objects $K, L \in \mc{W}(M,\l)$, we have a weak morphism of directed systems:
\eq \{HF^\bu(\phi^n_{\tilde{F}_0^C} K,L) \}_{n \in \mathbb{N}} \to \{HF^\bu(\phi^n_{\tilde{F}_1^C} K, L)\}_{n \in \mathbb{N}}. \eeq
\qed
\end{lem}

\subsection{The stop doubling trick}

Carrying over the notation from the previous section, let $P_0= \pi_\cT(\{-3C/4\} \tms W') \sub \mc{T} \sub \d_\infty M$ and let $P_1 = \pi_\cT(\{-C/4\} \tms W' )\sub \mc{T} \sub \d_\infty M$. Let $\fk{c}_0, \fk{c}_1$ be the cores of $P_0, P_1$ respectively. Note $P_1$ will have an auxiliary role here. 

We have natural functors 
\begin{equation}\label{equation:orlov-stop-double}
\begin{tikzcd}
\mc{W}(\hat{P_0}) \ar{r} \ar{dr} & \mc{W}(M, P_0 \cup P_1) \ar{d} \ar[r, "\simeq"] &\mc{W}(M, \fk{c}_0 \cup \fk{c}_1) \ar{d} \\
& \mc{W}(M, P_0) \ar[r, "\simeq"] & \mc{W}(M, \fk{c}_0)
\end{tikzcd}
\end{equation}
The functors from $\cW(\hat P_0)$ are Orlov functors, and the vertical arrows are the stop removal functors. We let $\cD$ be the essential image of the induced functor $\tw \mc{W}(\hat{P_0}) \to \tw\mc{W}(M, P_0)$. The following proposition is sometimes known as the ``stop-doubling trick".  
\begin{prop}[Prop.\ 7.9 in \cite{gps3}]\label{lemma:doubling-ff}
The Orlov functor $\cW(\hat P_0)\to \cW(M,P_0\cup P_1)$ is fully faithful. (Hence the same is true upon passing to the twisted envelope.)
\qed
\end{prop}
Recall
\begin{itemize}
\item $\mc{R}_0:= \pi_\mc{T}([-C-A, -C+A] \tms W') \sub N$;
\item $\mc{R}_1:= \pi_\mc{T}([-C/2-A, -C/2+A] \tms W') \sub N$. 
\end{itemize}

\begin{figure}
	\centering
	\includegraphics[height=9.5 cm]{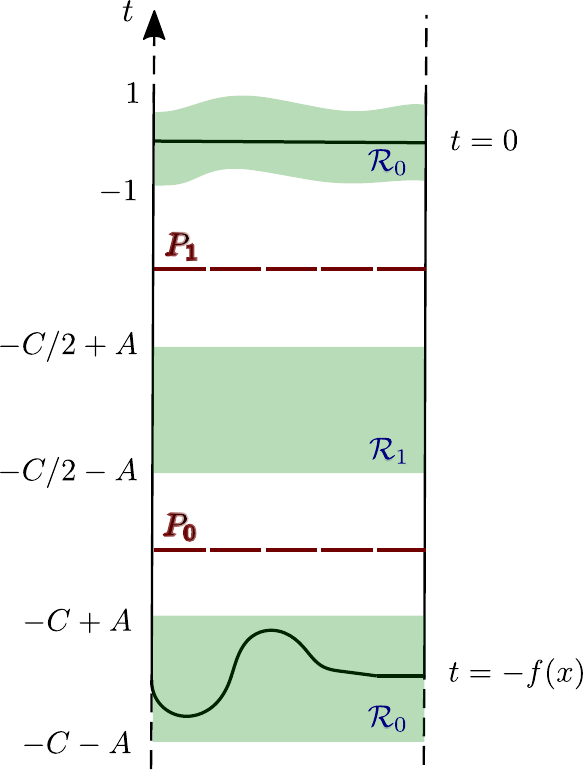}
	\caption{The contact mapping torus embedded in $\d_\infty M$; a Legendrian contained in $\mc{R}_i$ cannot wrap past the stop $P_i$.}
	\label{figure:drawingtry}
\end{figure}
We record the following technical lemma which can safely be skipped on first reading.
\begin{lem}\label{lemma:cofinal-wrappings} 
\begin{itemize}
\item[(i)] Given any Legendrian $\Lambda \sub \mc{R}_0$, it admits a (positive) cofinal wrapping in the wrapping categories $(\Lambda \leadsto-)^+_{(M, P_0)}$ and $(\Lambda \leadsto -)^+_{(M, P_0 \cup P_1)}$ which stays contained in $\pi_\mc{T}( [-C-A, -3C/4] \tms W') \sub N$
\item[(ii)] Given any Legendrian $\Lambda \sub \mc{R}_1$, it admits a (negative) cofinal wrapping in the wrapping categories $(\Lambda \leadsto-)^-_{(M, P_0)}$ and $(\Lambda \leadsto-)^-_{(M, P_0 \cup P_1)}$ which stays contained in $\pi_\mc{T}( [-3C/4, -C/2+A] \tms W') \sub N$
\end{itemize}
\end{lem}
\pf
We will only prove the first statement as the argument for the second one is analogous. Let $H: [-C-A, -3C/4)_t \tms W' \to \R$ be a function, which can also be viewed as a function on $\pi_\mc{T}([-C-A, -3C/4] \tms W')$ since $\pi_\mc{T}$ is injective in this region. If $H$ only depends on the $t$ coordinate, then one can check that the contact vector field induced by $H$ is $X_H= H \d_t + (\d_t H) Z_\l$.

Observe that if $H$ is non-increasing in $t$, then $\mc{R}_0$ remains inside $\pi_\mc{T}( [-C-A, -3C/4) \tms W)$ under the flow of $X_H$. Assume therefore that $H$ is non-increasing, and assume also that $H$ decays sufficiently fast near ${-3C/4} \tms W$ so that the flow is complete (for example, taking $H= (-3C/4 -t)$ near ${-3C/4} \tms W$ works). 

Let $\Lambda^t$ be obtained by flowing $\Lambda$ under $H$ for time $t$. The Hamiltonian vector field $X_H$ is the Reeb vector field for the contact form $1/H \a$. Hence the cofinality of $\Lambda^t$ follows from the criterion in \cite[Lem.\ 3.29]{gps1}. 
\epf

Let us now assume that $K$ is a Lagrangian with boundary in $\mc{R}_0 \sub N$. Let $K^k$ denote the Lagrangian obtained by isotoping $K$ via the flow of $\{H_s\}_{s=0}^{kC}$. As $H_s$ is periodic, $K^{k+1}=\phi_{H^C}(K^k)$, i.e. $K^{k+1}$ is obtained by applying time $C$-flow of $\{H_s\}$ to $K^k$. Observe that $K^k$ also has boundary in $\cR_0$, for $k\in\bN_+$ and it has boundary in $\cR_1$ for $k\in 1/2+\bN$. 

As mentioned we will apply (opposite category version of) \Cref{prop:colimisloc} to prove \Cref{theorem:tensor-equiv-hamiltonian}. The following verifies condition \eqref{colimpropcond:cone} in \Cref{prop:colimisloc}:
\begin{lem}\label{lemma:bar-cond-2}
The cone in $\mc{W}(M, P_0)$ of the continuation map $K^{k+1} \to K^k$ is isomorphic to an object in $\cD=\op{im}(\cW(\hat P_0)\to\cW(M,P_0))$. 
\end{lem}
\pf
It's enough to show that the cone of $K^{k+1/2} \to K^k$ is isomorphic to an object in $\cD$, since we can factor the map $K^{k+1} \to K^{k+1/2} \to K^k$ in cohomology where the first arrow is an isomorphism. 

As $K^{k+1/2}\to K^k$ is similarly an isomorphism in $\cW(M,P_1)$ (and thus in $\cW(M,\fc_1)$), and the kernel of $\cW(M,\fc_0\cup \fc_1)\to \cW(M,\fc_1)$ is generated by the image of $\cW(\hat P_0)\to \cW(M,\fc_0\cup \fc_1)$, the cone of $K^{k+1/2} \to K^k$ is a twisted complex in the image of $\cW(\hat P_0)\to \cW(M,\fc_0\cup \fc_1)$. This functor is fully faithful by \Cref{lemma:doubling-ff}; therefore, this twisted complex lies in the image of $\tw\cW(\hat P_0)\to \tw\cW(M,\fc_0\cup \fc_1)$. Hence, if we further apply $\tw\cW(M,\fc_0\cup \fc_1)\to \tw\cW(M,\fc_0)$ to $cone(K^{k+1/2} \to K^k)$, we find that it is in the image of $\tw\cW(\hat P_0)\to \tw\cW(M,\fc_0)$.
%
\epf

We now want to verify condition \eqref{colimpropcond:vanishing} in \Cref{prop:colimisloc}. Consider the following diagram, where the horizontal arrows are induced by continuation maps and the vertical arrows are induced by stop removal.
\begin{equation}
\begin{tikzcd}
HW_{\mc{W}(M, P_0 \cup P_1)}(K^k, L) \ar{r} \ar{d} & HW_{\mc{W}(M, P_0 \cup P_1)}(K^{k+1/2}, L) \ar{r} \ar{d} & HW_{\mc{W}(M, P_0 \cup P_1)}(K^{k+1}, L) \ar{d} \\
HW_{\mc{W}(M, P_0)}(K^k, L) \ar{r}  & HW_{\mc{W}(M, P_0)}(K^{k+1/2}, L) \ar{r} & HW_{\mc{W}(M, P_0)}(K^{k+1}, L) 
\end{tikzcd}
\end{equation}
\begin{lem}\label{lemma:bar-cond-1}
If $L \in \cD=\op{im}(\cW(\hat P_0)\to \cW(M,P_0))$, the first and last columns are isomorphisms and $HW_{\mc{W}(M, P_0 \cup P_1)}(K^{k+1/2}, L)=0$. Hence, the composition 
\begin{equation}
    HW_{\mc{W}(M, P_0)}(K^k, L) \to HW_{\mc{W}(M, P_0)}(K^{k+1}, L)
\end{equation}
vanishes. 
\end{lem}
\pf
Since $\d_\infty K^k, \d_\infty K^{k+1} \sub \mc{R}_0$, we may apply \Cref{lemma:cofinal-wrappings}(i): the conclusion is that $K^k$ and $ K^{k+1}$ do not wrap past $P_0$; thus, removing $P_1$ does nothing. For the second claim, recall $\d_\infty K^{k+1/2}\subset \cR_1$, and construct a positive cofinal wrapping of $K^{k+1/2}$ (in exactly the same way) which eventually gets stuck at $P_1$ and becomes disjoint from $L$. (See \Cref{figure:drawingtry}.)
\epf
Finally: 
\begin{cor}\label{corollary:partially-wrapped-hf}
Suppose that $\d_\infty L \sub \mc{R}_1$. Then the following diagram commutes and the vertical arrows are isomorphisms: 
\begin{equation}
\begin{tikzcd}
HF(K^k, L) \ar{r} \ar{d} & HF(K^{k+1}, L) \ar{d} \\
HW_{\mc{W}(M, P_0)}(K^k, L) \ar{r} &HW_{\mc{W}(M, P_0)}(K^{k+1}, L) 
\end{tikzcd}
\end{equation}
Hence the directed systems $\{HW_{\mc{W}(M, P_0)}(K^k, L)\} $ and $\{ HF(K^k, L)\}$ are isomorphic. 
\end{cor}
\pf
The horizontal arrows are induced by (right) multiplication by the continuation map $K^{k+1} \to K^k$. The vertical arrows are induced by (left) multiplication by the continuation maps associated to a negative cofinal wrapping of $L$ (\Cref{lemma:cofinal-wrappings}(ii)). The commutativity of the diagram follows. 

Such a negative wrapping of $L$ sends $\d_\infty L$ towards $P_0$ without ever intersecting $\d_\infty K^k, \d_\infty K^{k+1} \sub \mc{R}_0$. Hence the vertical arrows are isomorphism. 
\epf
%
%
%
%
\pf[Proof of \Cref{theorem:tensor-equiv-hamiltonian}]
Let $K,L$ be two exact Lagrangians that are cylindrical at infinity. After possibly isotoping $K, L$ through cylindrical Lagrangians, we can assume that $\d_\infty K \sub \mc{R}_0$ and $\d_\infty L \sub \mc{R}_1$. 
We now apply \Cref{theorem:spherical-tensor-equiv}, the conditions \eqref{colimpropcond:vanishing} and \eqref{colimpropcond:cone} are verified in \Cref{lemma:bar-cond-1} and \Cref{lemma:bar-cond-2} respectively. 

\Cref{theorem:spherical-tensor-equiv} implies that the growth function of $\{HW_{\mc{W}(M, P_0)}(K^k, L)\}$ is equivalent to the localization growth function of $\tw \mc{W}(X, P_0)/\cD (K, L)$. Now, recall that $\cD$ is by definition the essential image of the Orlov functor $\tw \cW(P_0) \to \tw \cW (M, P_0)$. According to \Cref{example:finite-time-weinstein}, $\tw \cW(P_0)$ is generated in finite time (i.e.\ admits a strong generator). Hence $\cD$ also has this property. It follows from \Cref{corollary:indep-gen-set} (and \Cref{note:corollaryextendfinitetime}) that the growth function of $\tw \mc{W}(X, P_0)/\cD (K, L)$ is the same as that of $\tw \mc{W}(X, P_0)/\cD' (K, L)$, where $\cD' \sub \cD$ is some finite collection of objects which generates $\cD$. By definition, the latter growth function is precisely $\gr_{K, L}$.

Thus, $\gr_{K, L}$ agrees with the growth function of the directed system $\{HW_{\mc{W}(M, P_0)}(K^k, L)\}$. Applying \Cref{corollary:partially-wrapped-hf}, this is equivalent to the growth of $\{ HF(K^k, L)\}= \{HF(\phi_{H^C}^kK,L)\}$. Finally, by utilizing \Cref{lemma:positivity-continuation}, we conclude $\{ HF(K^k, L)\}$ and $\{HF(\phi_h^kK,L)\}$ are weakly isomorphic, for any positive Hamiltonian $h$ that is linear at infinity. \Cref{theorem:tensor-equiv-hamiltonian} follows. 
\epf
\subsection{Symplectic cohomology and the diagonal Lagrangian}
The purpose of this section is to explain how to relate the growth of symplectic cohomology to the growth of the wrapped Floer cohomology of the diagonal Lagrangian. More precisely, we have the following proposition.

\begin{prop}\label{proposition:sh-diag}
The growth functions $\gr_{SH}^{ham}$ and $\gr_{\Delta}^{ham}$ are scaling equivalent.
\end{prop}

We will not prove \Cref{proposition:sh-diag} since  closely related arguments already appear in the literature. However, we wish to give an indication to the reader of how to reconstruct the proof. We note that \Cref{proposition:sh-diag} is needed for \Cref{corollary:sh-computation} but is otherwise entirely independent of the rest of the paper.

To begin with, we need to discuss wrapped Floer cohomology with product-type data. Let $\tilde{M}= (\ov{M} \tms M, - \lambda \oplus \lambda)$ and let $\Delta \sub \tilde{M}$ be the diagonal. We consider pairs $(\tilde{H}, \tilde{J})$, such that $\tilde{H}(x_1,x_2)= H(x_1)+H(x_2)$ and $\tilde{J}= -J \oplus J$. Here $H: M \to \R$ is a positive, cylindrical at infinity Hamiltonian and $J$ is a cylindrical at infinity, compatible almost-complex structure on $(M, \l)$. We may now define a group $HW_{prod}(\Delta)$ by taking direct limits over pairs according to the ordering $(\tilde{H}_1, \tilde{J}_1) \leq (\tilde{H}_s, \tilde{J}_2)$ iff $\tilde{H}_1 \leq \tilde{H}_2$ pointwise. Note that one needs a compactness argument to control Floer trajectories. See \cite[Sec.\ 8.2]{sheel-thesis} for a discussion of Floer homology with product-type data (using however a somewhat different technical setup). 

It can then be shown that $SH(M) = HW_{prod}(\Delta)$. More precisely, let $(H, J)$ be a pair as above and let $(\tilde{H}, \tilde{J})$ be the corresponding pair on $\ov{M} \tms M$. Then there is a canonical bijection of chain complexes $CF(M; nH, J) \equiv CF(\Delta; n\tilde{H}, \tilde{J})$ identifying orbits, differentials, and continuation maps. Passing to cohomology and taking direct limits, the left hand side is symplectic cohomology of $M$ while the right hand side is a version of wrapped Floer cohomology defined using product-type Hamiltonians and almost-complex structures. 

In particular, we have an identification of filtered directed systems $\{HF(M; nH)\}= \{HF(\Delta; n\tilde{H})\}$. To prove \Cref{proposition:sh-diag}, one only needs to show that the filtered directed system $\{HF(\Delta; n\tilde{H})\}$ is scaling equivalent to $\{HF(\Delta; nG)\}$, where $G$ is any positive Hamiltonian on $\ov{M} \tms M$ which is cylindrical at infinity. Said differently, we need to show that the growth rate of the diagonal defined with respect to product-type Hamiltonians agrees with the growth rate defined with respect to cylindrical Hamiltonians. 

Here, we refer the reader to the work of McLean; more precisely, to the proof of Thm.\ 4.1 in \cite{mclean2011computability}, which is the closed-string analog of the statement we need. The basic structure of the argument is as follows: first, McLean proved in \cite[Sec.\ 4]{mclean-gafa} that the positive-action part of the growth-rate of $SH(M)$ can be defined with respect to a large class of pairs $(H, J)$ which he called ``growth rate admissible". Cylindrical-type data is growth-rate admissible, and he verifies in \cite[Sec.\ 4]{mclean2011computability} that product-type data is also admissible. Finally, one observes that restricting to the positive-action part does not change the filtered directed systems up to weak equivalence, since the only difference is the finite-dimensional ``low-energy" part. 

\section{Applications}\label{sec:applications}

In this section, we discuss some applications of the theory developed in this paper. The most important ones are probably \Cref{corollary:hms-main-equiv} and \Cref{corollary:growth-most-exponential}.

\subsection{Computations}
\subsubsection{Homological Mirror Symmetry}
Let $(M, \fstop)$ be a Weinstein pair. Let $(X, D)$ be a proper algebraic variety equipped with a Cartier divisor. We say that this data satisfies homological mirror symmetry for pairs if there is a diagram
\eq
\begin{tikzcd}
\op{Perf} \mc{W}(M, \fstop) \ar[r] \ar[d, equal] & \op{Perf}\mc{W}(M) \ar[d, equal] \\
D^bCoh(X) \ar{r} & D^bCoh(X - D)
\end{tikzcd}
\eeq
(Most formulations of homological mirror symmetry for pairs also require an equivalence between the compact Fukaya category of $\fstop$ and $\op{Perf}(D)$, but this is unnecessary for our purposes.)

Here are some known instances of homological mirror symmetry (HMS) for pairs.
\begin{itemize}
\item The work of Hacking--Keating \cite{hacking-keating} (generalizing earlier work of Keating \cite{keating}) establishing HMS for Log Calabi--Yau surfaces;  
\item The work of Gammage--Shende \cite{gammage-shende} establishing HMS for toric stacks.
\end{itemize}

\begin{cor}\label{corollary:hms-main-equiv}
Suppose that the pairs $(M, \fstop)$ and $(X, D)$ are homologically mirror. Let $K, L \in \tw^\pi \mc{W}(M)$ be objects and let $\scrF_K, \scrF_L$ be their image in $D^bCoh(X -D)$. Then $ \gr_{K, L}$ and  $\gr_{\scrF_K, \scrF_L}$ coincide up to scaling equivalence.
\end{cor}

Combining \Cref{corollary:hms-main-equiv} with \Cref{cor:growthsheaves} and \Cref{theorem:tensor-equiv-hamiltonian}, we find:

\begin{cor}\label{corollary:hms-main-equiv2}
Under the assumptions of \Cref{corollary:hms-main-equiv}, the Hamiltonian growth function $\gr_{K, L}^{ham}$ is scaling equivalent to the function \eq p\mapsto dim(F^p\colim_n RHom_X(\scrF,\scrF'(nD)))\eeq 
\end{cor}

\begin{exmp}\label{example:anti-can-2d}
Let $X= \mathbb{CP}^2$ and let $D$ be the union of a line and a conic. Homological mirror symmetry for pairs holds in this setting (Pascaleff \cite{pascaleff2014floer}, Hacking--Keating \cite{hacking-keating}). The Weinstein mirror is $X= \{(u, v) \mid uv \neq 1\}$. Since $D$ is ample, \Cref{thm:growthisdimsupp} applies. We find that $\gr_{K, L}$ (and hence also $\gr_{K, L}^{ham}$) is scaling equivalent to a polynomial of degree $dim (supp(\scrF_K) \cap supp(\scrF_L))$. 
\end{exmp}


\begin{lem}\label{lemma:kunneth-hms}
Suppose that $(M, \fstop)$ is a Weinstein pair which is homologically mirror to $(X, D)$. Then $(\ov{M} \tms M, \fstop \tms \fk{c}_M \cup \fk{c}_M \tms \fstop \cup \fk{c}_M \times \fk{c}_M \times \mathbb{R})$ is a Weinstein pair which is mirror to $(X \tms X, D \tms X \cup X \tms D)$. Moreover, the mirror functor takes the diagonal Lagrangian to the diagonal coherent sheaf.
\end{lem}
\pf
This is a formal consequence of the fact that both sides of the mirror equivalence satisfy a K\"{u}nneth formula. On the symplectic side, see \cite[Thm.\ 1.5]{gps2}. For coherent sheaves, see \cite[Prop.\ 4.6.2]{gaitsgory2011ind} (in \textit{loc}.\ \textit{cit}.\ note that $D^bCoh(X)$ can be recovered from $\op{IndCoh}(X)$ by passing to compact objects). 
\epf

\begin{cor}\label{corollary:sh-computation}
Suppose that $(M, \fstop)$ and $(X, D)$ are homologically mirror. Then $\gr_{SH}^{ham}$ is scaling equivalent to the function $p \mapsto \sum_{k=1}^n dim(H( \O^k_X(2pD))[k])$. In particular, if $D$ is ample, this is equivalent to $p^{dim(X)}$.
\end{cor}
\pf Let $\Delta \sub \ov{X} \tms X$ be the diagonal. We know by combining \Cref{proposition:sh-diag} and \Cref{theorem:tensor-equiv-hamiltonian} that $\gr_{SH}^{ham}$ is scaling equivalent to $\gr_{\Delta, \Delta}$. 

On the other hand, let $j: X \to X \tms X$ be the diagonal embedding. Let $\mc{O}_\Delta:= i_* \mc{O}_X$. It can be shown that $i^* \mc{O}_\Delta\simeq \oplus_{k=0}^n \O^k_X[k]$ (see \cite[(3.2)]{huybrechts2009derived}). Then we have 
\begin{align}\label{equation:diag-target} 
    RHom_{X\times X}(\mc{O}_\Delta, \mc{O}_\Delta \otimes \pi_1^* \mc{O}_X(pD) \otimes \pi_2^*\mc{O}_X(pD)) &= RHom_X(\mc{O}_X, (i^*  \mc{O}_\Delta) \otimes \mc{O}_X(2pD))= \\
     H(i^* \mc{O}_\Delta \otimes \mc{O}_X(2pD)) &=     H(\oplus_{k=0}^n \O^k_X[k] \otimes \mc{O}_X(2pD)) \notag
\end{align}
According to \Cref{cor:growthsheaves}, the growth function of \eqref{equation:diag-target} agrees up to scaling equivalence with  $\gr_{\cO_\Delta, \cO_\Delta}$. Finally, combine \Cref{lemma:kunneth-hms} and \Cref{corollary:hms-main-equiv}.
\epf

\begin{exmp}
Let $(X, D)$ be as in \Cref{example:anti-can-2d}. Then $\gr_{SH}^{ham}(M)=2$. Indeed, $D$ is ample so this follows from \Cref{corollary:sh-computation} and \Cref{lem:chipolynomial}.
\end{exmp}

\subsubsection{Cotangent bundles}

Let $N$ be an oriented closed manifold and consider its cotangent bundle $(T^*N, \l_{can})$. We will consider the wrapped Fukaya category of $T^*N$ with respect to the canonical grading/orientation data induced by the cotangent polarization \cite[Sec.\ 5.3]{gps3}. If $F \sub T^*N$ denotes a cotangent fiber, then the growth function $\gr_{F,F}$ turns out to be a purely topological quantity. 

To explain this, we begin by fixing a Riemannian metric $g$ on $N$. Consider the function 
\eq f_{N, g}(n) = \op{dim} im (H(\O_{\leq n} N) \to H(\O N)), \eeq
where $\O_{\leq n} N$ is the space of based loops of length at most $n$ with respect to $g$. It's easy to see that $f_{N, g}$ is independent of $g$ up to scaling. 

\begin{lem}[Lem.\ 2.10 in \cite{mclean}]
Let $F \sub T^*N$ be a cotangent fiber. Then we have
\eq \limsup_{n \to \infty} \frac{\log f_{N, g}}{\log n} = \limsup_{n \to \infty} \frac{\log \gr_{F, F}}{\log n}. \eeq
\end{lem}
The proof of \cite[Lem.\ 2.10]{mclean} presumably implies that the scaling equivalence class of $f_{N, g}$ is precisely $\g_{F,F}$, but we have not checked this.

\begin{exmp}\label{example:mclean-hyp-computation}
If $N$ is hyperbolic, then $f_{N, g}$ grows at least exponentially (this follows e.g.\ from \cite[Thm.\ 2]{milnor1968note}). If $M$ is simply-connected, a theorem of Gromov (\cite[Thm.\ 1.4]{gromov1978homotopical} or \cite[Thm.\ 7.3]{gromov2007metric}) implies that $ \sum_{i \leq m} dimH_i(\O N) \leq dim(H(\O_{\leq C m}N \to H(\O N))$ for some $C>0$. Hence $f_{N,g}$ grows exponentially if $N$ is rationally hyperbolic \cite[Prop.\ 33.9]{felixhalperinthomas}. (A simply-connected topological space $X$ is said to be \emph{rationally hyperbolic} if $\pi_*(X) \otimes \mathbb{Q}$ is an infinite dimensional $\mathbb{Q}$-vector space.) 
\end{exmp}

\subsection{Categorical entropy}\label{subsection:cat-entropy}

Dimitrov--Haiden--Kontsevich--Katzarkov \cite{dhkk} introduced a notion of \emph{(categorical) entropy} of an endofunctor acting on a triangulated category. If $\cT=H^0(\cB)$ is the homotopy category of a smooth and proper pre-triangulated $A_\infty$ category $\cB$ and $F: \cB \to \cB$ is an endofunctor, then the entropy of $F$ is
\eq\label{equation:cat-entropy} h(F):= \op{lim}_{n \to \infty} \frac{1}{n} \log dim H(\cB)(G, F^n G) \in [-\infty, \infty) \eeq
where $G$ is any (split-)generator. It can be shown that the entropy is well-defined, meaning that the limit exits and do not depend on $G$. (We remark that there is a more general definition of entropy which is equivalent under smoothness or properness assumptions.)

More generally, one can replace the factor $1/n$ in \eqref{equation:cat-entropy} by $1/g(n)$, where $g$ is any function which grows to infinity with $n$. For example, taking $g(n)=\log n$ gives a notion of \emph{polynomial (categorical) entropy} which we call $h_{pol}(-)$. (This notion coincides with the one considered in \cite{ffo} when $h(F)=0$.)

Let us now relate these notions of entropy to the growth functions which are the main objects of this paper. The main point can be summarized in the following:. 

\begin{thm}\label{theorem:meta-thm-entropy}
Suppose that $f:\cA \to \cB$ is a spherical functor and $\cD=f(\cA)$. Given objects $K, L \in \cB$, the growth function $\gr_{K,L}=\gr_{H(\cB/\cD)(K, L)}$ gives a lower bound for the entropy/polynomial entropy of the corresponding spherical twist. More precisely, the entropy is bounded below by $\limsup\limits_{p\to \infty}\frac{\log\gr_{H(\cB/\cD)(K, L)}(p)}{p}$ and the polynomial entropy is bounded below by $\limsup\limits_{p\to \infty}\frac{\log\gr_{H(\cB/\cD)(K, L)}(p)}{\log p}$.
\end{thm}
\begin{rk}
It is important to note that for the claim about entropy, one considers $\gr_{H(\cB/\cD)(K, L)}$ exactly, not up to scaling equivalence. The reason for this  is that replacing $\gr_{H(\cB/\cD)(K, L)}$ by a scaling equivalent function (such as $\gr_{H(\cB/\cD)(K, L)}(2p)$) changes $\limsup\limits_{p\to \infty}\frac{\log\gr_{H(\cB/\cD)(K, L)}(p)}{p}$, even though one still knows whether it is positive or $0$. On the other hand, $\limsup\limits_{p\to \infty}\frac{\log\gr_{H(\cB/\cD)(K, L)}(p)}{\log p}$ remains unchanged under scaling equivalences. In other words, from the scaling equivalence class of the growth function, we can determine whether the entropy is positive, whereas we have an exact lower bound $\limsup\limits_{p\to \infty}\frac{\log\gr_{H(\cB/\cD)(K, L)}(p)}{\log p}$ on the polynomial entropy.
\end{rk}
\begin{proof}[Proof of \Cref{theorem:meta-thm-entropy}]
We show this for polynomial entropy only, as the actual entropy is analogous. Clearly, \begin{equation}
	dim(F^p\colim_k H(\cB)(K,S^k(L)))\leq dim(H(\cB)(K,S^p(L)))
\end{equation}
On the other hand, the right hand side is equal to $\gr_{H(\cB/\cD)(K, L)}(p)=dim(F^pH(\cB/\cD)(K, L))$ by \Cref{thm:spherical}. Therefore, by  taking the logartithm, dividing by $\log p$ and taking $\limsup$, we obtain the inequality. More precisely, $\limsup\limits_{p\to \infty}\frac{\log dim(H(\cB)(K,S^p(L)))}{\log p}$ is bounded above by  $\limsup\limits_{p\to \infty}\frac{\log dim(H(\cB)(G,S^p(G)))}{\log p}$ for any split generator $G$. 
\end{proof}
As advertised in the introduction, this implies: 
\begin{cor}\label{corollary:growth-most-exponential}
Let $M$ be a Liouville manifold (equipped with orientation/grading data as in \Cref{assumption:orientation-grading}). The growth of $\gr_{K,L}=\gr_{H(\cB/\cD)(K, L)}$ is at most exponential. Hence, by \Cref{theorem:tensor-equiv-hamiltonian}, the growth of wrapped Floer cohomology in the sense of McLean \cite{mclean} is at most exponential.
\end{cor}
Similarly, the growth of symplectic cohomology (in the sense of \cites{seidel-biased-view, mclean-gafa}) is at most exponential (see \Cref{proposition:sh-diag}). 

\begin{proof}
Since the entropy is finite \cite[Lem.\ 2.5]{dhkk}, we see that $\limsup\limits_{p\to \infty}\frac{\log\gr_{(\cB/\cD)(K, L)}(p)}{p}$ is also finite, say strictly less than $C\in\bR_+$. This implies $\gr_{(\cB/\cD)(K, L)}(p)\leq e^{Cp}$ for all but finitely many $p\in\bN$. 
\end{proof}

Let us now discuss further applications of \Cref{theorem:meta-thm-entropy}. We begin by elaborating on \Cref{exmp:sphericaldivisor}. 

Suppose that $X$ is a smooth algebraic variety and that $D \sub X$ be a Cartier divisor with $U:= X\setminus D$. We saw that $j:D^bCoh(D) \to D^bCoh(X)$ is spherical and that the functor $S:=(-) \otimes \mc{O}(D)$ is the spherical twist. It can be shown that the entropy of $S$ is always zero. However, the polynomial entropy of $S$ is potentially interesting. 

Let $\scrF, \scrF'$ be objects of $D^bCoh(U)$ and let $\ov{\scrF}, \ov{\scrF'}$ be lifts to $D^bCoh(X)$. 
By \Cref{theorem:meta-thm-entropy}, we have $h_{pol}(S) \geq  \limsup\limits_{n \to \infty} \frac{ \log \gr_{\scrF, \scrF'}(n)}{ \log n} $. 
We emphasize that while $h_{pol}(S)$ depends on the pair $(X, D)$, $\limsup\limits_{n \to \infty} \frac{ \log \gr_{\scrF, \scrF'}(n)}{ \log n}$ only depends on $U \sub X$. 
\begin{exmp}
Suppose that $D \sub X$ is canonical. Then $S:= (-) \otimes \mc{O}(D)[n]$ is the \textemph{Serre functor} on $D^bCoh(X)$. If $D$ is moreover ample, then \Cref{thm:growthisdimsupp} implies that $h_{pol}(S) \geq dim X$ and asymptotic Riemann--Roch implies that equality holds (a fact that was already known, see \cite[Rmk.\ 6.11]{ffo}). 
\end{exmp}

Let us now consider categorical entropy in the symplectic setting. We elaborate on \Cref{example:symplectic-spherical}. Let $M$ be a Weinstein manifold, and $\fstop \sub \d_\infty M$ be a Weinstein hypersurface. For such stop there is a notion of \textemph{swappability} defined in  \cite[Def.\ 1.1]{sylvan-orlov}. This means that there exists an isotopy of hypersurfaces from a positive pushoff $\fstop^+$ to a negative pushoff $\fstop^-$ which avoids $\fstop$ \cite[Def.\ 1.1]{sylvan-orlov}.  The canonical example of a swappable stop is a page of an open book decomposition. (This is possibly also the only known example.)

When $\fstop$ is swappable, Sylvan \cite{sylvan-orlov} describes two natural functors $\mathbf{W}: \mc{W}(M, \fstop) \to \mc{W}(M, \fstop)$ and $\mathbf{M}: \mc{W}(\hat{\fstop}) \to \mc{W}(\hat{\fstop})$. The first functor is the ``wrap-once positively" auto-equivalence, and it acts on a Lagrangian by wrapping it one past the stop in the positive Reeb direction. The second functor is the ``monodromy" autoequivalence, and it acts on a Lagrangian by the monodromy induced by the positive isotopy $\fstop_t$. (For instance, if $\fstop$ is a page of an open book, then $\mathbf{M}$ acts by the monodromy of the open book). 

Consider now the Orlov functor $\mc{W}(\ov{\fstop}) \to \mc{W}(M)$. Sylvan \cite[Thm.\ 1.3]{sylvan-orlov} proved that this functor is spherical when $\fstop$ is swappable. Moreover, he proved that the spherical twist is $\mathbf{W}^{-1}$ and the spherical cotwist is $\mathbf{M}^{-1}[2]$. 

To connect this discussion to our growth functions, let $K, L$ be objects of $\mc{W}(X)$ and let $\ov{K}, \ov{L}$ be lifts to $\mc{W}(X, \fstop)$. 
According to \Cref{theorem:meta-thm-entropy}, we have the following inequality: 
\begin{equation}\label{equation:entropy-symplectic} h(\mathbf{W}^{-1}) \geq  \limsup\limits_{n \to \infty}\frac{\log\gr_{K, L}(n)}{n}. \end{equation}
Similarly, $h_{pol}(\mathbf{W}^{-1}) \geq \lim_{n \to \infty} \frac{\log \gr_{K, L}(n)}{\log n}$.  The right hand side of \eqref{equation:entropy-symplectic} only makes sense up to scaling equivalence. As a result, it only makes sense to ask whether the right hand side is zero, strictly positive or infinite. (In contrast, the limit on the right hand side of the second inequality is well-defined because the log in the denominator absorbs the scaling factor). 

More generally, a variant of \eqref{equation:entropy-symplectic} also holds if $K, L$ are objects of $\tw^\pi \mc{W}(X)$ (see \Cref{note:idempotentextension}).

If $F$ is a endofunctor acting on a triangulated category $\mc{T}$ which admits a Serre functor, then it is straightforward to see that $h(F)=h(F^{-1})$ and $h_{pol}(F)= h_{pol}(F^{-1})$. In the case at hand, it is expected that $\mathbf{W}$ is the Serre functor on $\tw^\pi \mc{W}(X, \fstop)$. In fact, if $\fstop$ is the fiber at infinity of a Lefschetz fibration, this statement was proved by Seidel \cite{seidel2017fukaya} (modulo some folkloric compatibility issues about identifying the classical Fukaya--Seidel category with $\mc{W}(X, \fstop)$). The upshot is that \eqref{equation:entropy-symplectic} implies the same lower bound for the entropy/polynomial entropy of $\mathbf{W}$.

\begin{exmp}\label{example:entropy-hyp-qhyp}
Let $N$ be a hyperbolic or rationally hyperbolic closed manifold, and fix a Lefschetz fibration structure with Weinstein fibers on $M=T^*N$. Set $\fstop = f^{-1}(\{\infty\})$ and consider the Fukaya--Seidel category $\mc{W}(M, \fstop)$. Then it follows from \eqref{equation:entropy-symplectic}, \Cref{example:mclean-hyp-computation}, and the previous paragraph that $h(\mathbf{W}) = h(\mathbf{W}^{-1})>0$. 
\end{exmp}

One can also consider the compact Fukaya category $F(\hat{\fstop})$ of $\hat{\fstop}$, i.e.\ the subcategory of $\tw^\pi \mc{W}(\hat{\fstop})$ split-generated by compact Lagrangians. The (inverse) monodromy functor $\mathbf{M}^{-1}$ induces an autoequivalence of this category, and we can ask about the entropy of this autoequivalence. 

Let $\cap$ denote the left-adjoint of the Orlov functor. Using the fact that $\mathbf{W}^{-1}$ and $\mathbf{M}^{-1}[-2]$ are respectively the spherical twist and cotwist of the Orlov functor, it is a general fact \cite[Lem.\ 2.9]{kim-entropy} that we have $\mathbf{M}^{-1}[2] \cap =\cap \mathbf{W}^{-1}$. One can then verify that 
\eq\label{equation:monodromy-wrap-bound} \sum_{i=1}^n dim H(\mc{W}(\hat{\fstop}))( \cap K, (\mathbf{M}^{-1})^i\cap L) \geq\atop dim H(\mc{W}(M, \fstop))(K, (\mathbf{W}^{-1})^{n+1} L)- dim H(\mc{W}(M, \fstop))(K, (\mathbf{W}^{-1}) L). \eeq
(Indeed, $\mc{W}(\hat{\fstop})( \cap K, (\mathbf{M}^{-1})^i\cap L) = \mc{W}(\hat{\fstop})( \cap K, \cap (\mathbf{W}^{-1})^i L) = \mc{W}(\hat{\fstop})(  K, \cup \cap (\mathbf{W}^{-1})^i L)$. Now use the exact triangle $1 \to \cup \cap \to \mathbf{W}^{-1}$.)

Let us suppose that the image of $\cap$ split-generates the compact Fukaya category of $\hat{\fstop}$. For instance, in the Lefschetz fibration case, this corresponds to the assumption that the compact Fukaya category of $\hat{\fstop}$ is generated by vanishing cycles). Then it follows from \eqref{equation:monodromy-wrap-bound} that $h(\mathbf{W}^{-1}) \leq h({\mathbf{M}^{-1}}|_{F(\hat{\fstop})})$ and that $h_{pol}(\mathbf{W}^{-1}) \leq h_{pol}({\mathbf{M}^{-1}}|_{F(\hat{\fstop})}) +1$. As before, if the compact Fukaya category admits a Serre functor, we also get $h(\mathbf{M}|_{F(\hat{\fstop})})= h({\mathbf{M}^{-1}}|_{F(\hat{\fstop})})$ and $h_{pol}(\mathbf{M}|_{F(\hat{\fstop})})= h_{pol}({\mathbf{M}^{-1}}|_{F(\hat{\fstop})})$. The upshot is that, in this setting, the lower bounds \eqref{equation:entropy-symplectic} on the entropy/polynomial entropy of $\mathbf{W}^{-1}$  also imply lower bounds on $\mathbf{M}|_{F(\hat{\fstop})}, {\mathbf{M}^{-1}}|_{F(\hat{\fstop})}$.

\appendix

\section{Filtrations on the Verdier quotient}\label{sec:verdierappendix}
In this \namecref{sec:verdierappendix}, we show that localization filtration is essentially independent of the specific model for the $A_\infty$-quotients (Lyubashenko--Ovsienko model in our case), and explain how to recover it at the level of triangulated categories. The results in this appendix are not used elsewhere in the paper, but they may provide a useful perspective on some of the previous constructions.

For motivation, recall how Verdier quotients are defined: given a triangulated category $\fT$, and $\fT_0\subset \fT$ be a triangulated subcategory that is closed under taking direct sums. Then, $\fT/\fT_0$ is a triangulated category with the same set of objects as $\fT$. A morphism $K\to L$ in $\fT/\fT_0$ is given by a roof diagram 
\begin{equation}\label{eq:roofkkl}
	\xymatrix{ & K'\ar[dr]\ar[dl]& \\ K&&L }
\end{equation}
such that the cone of $K'\to K$ is in $\fT_0$. If $\fT_0$ is generated by $\fD\subset \fT$, then one can find a resolution
\begin{equation}\label{eq:lengthpresofroof}
\xymatrix{ K'=K_p \ar[r]& K_{p-1}\ar[r]& K_{p-2}\ar[r]&\dots \ar[r]&  K_0=K } 	
\end{equation}
such that each $cone(K_i\to K_{i-1})\in \fD$. One can filter $\fT/\fT_0(K,L)$ by $p$: a given morphism $K\to L$ in $\fT/\fT_0$ belongs to $F^p\fT/\fT_0(K,L)$ if and only if it is represented by a roof diagram \eqref{eq:roofkkl} such that a resolution of length $p$ as in \eqref{eq:lengthpresofroof} exists. Note that even if we assume $\fD$ only split-generates $\fT_0$, this does not make any difference: one can always replace $K'$ by a similar object containing it as a direct summand. Also notice the superficial similarity of this filtration to the original definition of categorical entropy in \cite{dhkk}. 

Let $\cB$ be an $A_\infty$-category as before, and let $\cD\subset\cB$ be a subcategory. Then, we know $\cB/\cD$ is a filtered $A_\infty$-category. A resolution like \eqref{eq:lengthpresofroof} in $\tw\cB$ such that each $cone(K_i\to K_{i-1})$ is quasi-isomorphic to an object in $\cD$ implies $K'$ and $K$ are quasi-isomorphic in $\cB/\cD$. Our goal is to prove the following:
\begin{thm}\label{thm:verdiervstensor}
A closed morphism $f\in \cB/\cD(K,L)$ is in $F^p\cB/\cD(K,L)$ if and only if $K$ admits a resolution in $\tw \cB$ of length at most $p$ (i.e.\ as in \eqref{eq:lengthpresofroof}), such that $cone(K_i\to K_{i-1})$ is quasi-isomorphic to an object in $\cD$ and such that the composition $K'\to K\xrightarrow{f} L$ in $\cB/\cD$ lifts to a morphism $K'\to L$ in $\tw\cB$ (up to homotopy).
\end{thm}
The theorem immediately follows from:
\begin{prop}\label{prop:liftstocone}
A closed morphism $f\in \cB/\cD(K,L)$ is in $F^p\cB/\cD(K,L)$ if and only if there exists $(X,\delta)\in \tw_{\leq p}\cD$ and a morphism $K\to (X,\delta)$ in $\tw\cB$ such that the composition $K':=cone(K\to (X,\delta))[-1]\to K\xrightarrow{f}L$ lifts to a morphism $K'\to L$ in $\tw\cB$. 
\end{prop}
We split the proof of \Cref{prop:liftstocone} into two: the only if part is an inverse to \Cref{prop:scalinginverse}:
\begin{lem}
If $f\in F^p\cB/\cD(K,L)$ is a closed morphism, then there exists $(X,\delta)\in \tw_{\leq p}\cD$ and a morphism $K\to (X,\delta)$ such that the composition $K'=cone(K\to (X,\delta))[-1]\to K\xrightarrow{f}L$ is cohomologous to a morphism in $\tw\cB(K',L)\subset \tw\cB/\cD(K',L)$. The element bounding the difference of maps $K'\to L$ can also be taken from $F^p\tw\cB/\cD(K',L)$.
\end{lem}
\begin{proof}
We prove this by induction on $p$. There is nothing to prove for $p=0$. Assume the statement is true for filtered part less than $p$. Consider 
	\eq
f\in 	F^p(\cB/\cD)(K,L)=\bigoplus_{D_1,\dots, D_k\in\cD\atop k\leq p} \cB(D_k,L)\otimes \dots\otimes \cB(K,D_1)[k]
	\eeq
By homological perturbation theory, one can without loss of generality assume $\cB$ is minimal. Assume for simplicity that the $k=p$ component of $f$ lies in a single $\cB(D_p,L)\otimes \dots\otimes \cB(K,D_1)[p]$. This component is a finite sum of the form $\sum_i f_p^{(i)}\otimes \dots\otimes f_0^{(i)}$. Then, the collection $f_0^{(i)}:K\to D_1$ defines a morphism $f_0:K\to D_1^{\oplus N}$. Here $N$ is the number of summand in the expression for $k=p$ component of $f$. By minimality of $\cB$, $f_0$ is a closed morphism; therefore, one can form $cone(K\xrightarrow{f_0}D_1^{\oplus N})$. Let $K_1=cone(K\xrightarrow{f_0}D_1^{\oplus N})[-1]$. By pushing $f$ along the map 
	\eq
	F^p(\cB/\cD)(K,L)\to F^p(\tw\cB/\cD)(K_1,L)
	\eeq
	one obtains a cocycle in the latter such that $k=p$ component has the same form. Let $\pi^{(i)}\in \tw\cB(D_1^{\oplus N},D_1)$ denote the projection to $i^{th}$-component. Now consider the element in 
	\eq
	\cB(D_p,L)\otimes \dots\otimes \tw\cB(D_1^{\oplus N},D_1)
	\eeq 
	given by 
	\eq\sum_i f_p^{(i)}\otimes \dots\otimes f_1^{(i)}\otimes \pi^{(i)}\eeq
	which in turn gives a cochain in
	\eq\label{eq:twistedbounder}
	\cB(D_p,L)\otimes \dots\otimes \tw\cB(cone(K\xrightarrow{f_0}D_1^{\oplus N})[-1],D_1)
	\eeq 
Call this cochain $\beta$. Then $d(\beta)$, considered as an element of $F^p\tw\cB/\cD(K_1,L)$, has the same $k=p$ component as $f$, i.e. it is given by $\sum_i f_p^{(i)}\otimes \dots\otimes f_0^{(i)}$. Note that this also uses minimality of $\cB$, and also the fact that the differential of $\pi^{(i)}$ considered as an element of $\tw\cB/\cD(cone(K\xrightarrow{f_0}D_1^{\oplus N})[-1],D_1)$ is equal to $f_0^{(i)}$. Therefore, one can correct image of $f$ in $F^p(\tw\cB/\cD)(K_1,L)$ by $d(\beta)$ to lower the length by $1$. One can proceed inductively to produce a twisted complex $(X,\delta)$ and a map $K\to (X,\delta)$ such that the image of $f$ inside $(\tw\cB/\cD)((K\to (X,\delta))[-1],L)$ is cohomologous to a class in $F^0(\tw\cB/\cD)((K\to (X,\delta))[-1],L)$. It is clear from the proof that the element bounding the difference can be taken from $F^p(\tw\cB/\cD)((K\to (X,\delta))[-1],L)$.
\end{proof}
\begin{rk}
In the proof, we assumed minimality of $\cB$. During the induction, one extends the category $\cB$ by a cone and needs to pass to a minimal model at each step. However, this does not create a difficulty, as this passages induce $E_0$-equivalences; hence, the quasi-isomorphism type of $F^p(\cB/\cD)$ also does not change. 
\end{rk}
Conversely, 
\begin{lem} Consider a closed element $f \in (\cB/\cD)(K,L)$ which can be represented as an element of $\tw\cB(cone(K\to (X,\delta))[-1],L)$, where $(X,\delta)\in \tw_{\leq p}\cD$ (i.e. it has length at most $p$) and $K\to (X,\delta)$ is a closed morphism in $\tw\cB$. Then $f\in F^p(\cB/\cD)(K,L)$.
\end{lem}
\begin{proof}
Let $K'=cone(K\to (X,\delta))[-1]$, and consider the morphism $\alpha: K'\to K$ in $\tw\cB$. Multiplication by $\alpha$ induces a chain map \begin{equation}
	\tw\cB/\cD(K,L)\to \tw\cB/\cD(K',L)
\end{equation}
and by assumption, this map sends $f\in \tw\cB/\cD(K,L)$ into $F^0\tw\cB/\cD(K',L)$. The image of $\alpha$ in $\tw\cB/\cD(K',K)$ is invertible, and to prove the claim, it suffices to show that it has a quasi-inverse in $F^p\tw\cB/\cD(K',K)$. 

The morphism $\alpha$ factors through a sequence $K'=K_p\to K_{p-1}\to \dots\to K_0=K$, where each $cone(K_i\to K_{i-1})$ is quasi-isomorphic to a direct sum of objects of $\cD$. Therefore, it suffices that each $K_i\to K_{i-1}$ has a quasi-inverse in $F^1\tw\cB/\cD(K_{i-1},K_i)$. In other words, without loss of generality we can assume $p=1$ and $(X,\delta)$ is a direct sum of shifts of objects of $\cD$ (thus $\delta=0$), and show that $\alpha$ has a quasi-inverse in $F^1\tw\cB/\cD(K,K')$. Say $(X,\delta)=\bigoplus_i D_i$, where $D_i\in\cD$ (we are ignoring the shifts for notational simplicity).

Let $\eta_i:K\to D_i$ denote $i^{th}$-component of $K\to \bigoplus D_i=(X,\delta)$, and let $\iota_i:D_i\to cone(K\to\bigoplus_i D_i)=K'$ denote the inclusion into $\bigoplus_i D_i$. Let $\theta:K\to cone(K\to\bigoplus_i D_i)=K'$ denote the inclusion of $K$ into the first component. Note that $\eta_i,\iota_i$ are closed morphisms, but $\theta$ is not. However, one can form a closed cocycle $\beta\in F^1\tw\cB/\cD(K,K')$ with $k=0$ component given by $\theta\in\tw\cB(K,K')$ and $k=1$ component (up to sign) given by 
\begin{equation}
	\sum_i \iota_i\otimes\eta_i\in \bigoplus_i \tw\cB(D_i,K')\otimes \cB(K,D_i)[-1]\subset F^1\tw\cB(K,K')
\end{equation}
It is easy to check that $\beta$ is indeed a cocycle, i.e. it is closed. Using the fact that $\mu^2(\alpha, \iota_i)=0$ and $\mu^3(\alpha,\iota_i,\eta_i)=0$, we see that $\alpha\circ \beta=1$. As $\alpha$ in an invertible class, $\beta$ is a quasi-inverse (one can check it is also a right inverse directly, but this is not needed). This completes the proof.
\end{proof}
One can use \Cref{thm:verdiervstensor}, or \Cref{prop:liftstocone} to reprove \Cref{corollary:indep-gen-set} (or its extension given in \Cref{note:corollaryextendfinitetime}). This follows from
\begin{cor}\label{corollary:appendix-conclusion}
Let $\cD\subset\cD'$, and assume that every object of $\cD'$ is quasi-isomorphic to an object of $\cD$ obtained by taking cones at most $l$ times. Then if a closed element $f\in \cB/\cD(K,L)$ has image cohomologous to an element of $F^p\cB/\cD'(K,L)$, then $f$ cohomologous to an element in $F^{pl}\cB/\cD(K,L)$.
\end{cor}
\begin{proof}
The assumption implies there is a resolution as in \eqref{eq:lengthpresofroof} with $cone(K_i\to K_{i-1})$ in $\cD'$ such that $f$ lifts to a morphism $K'\to L$. As elements of $\cD'$ are iterated cones of length at most $l$ in $\cD$, one can further factorize $K_i\to K_{i-1}$ to a sequence of length at most $l$ and with cones in $\cD$. Therefore, the map $K'\to K$ factorizes to $pl$ maps with cones in $\cD$, and by using \Cref{thm:verdiervstensor} again, we conclude the result.
\end{proof}
Therefore, $\gr_{\cB/\cD(K,L)}(p)\leq \gr_{\cB/\cD'(K,L)}(p)\leq \gr_{\cB/\cD(K,L)}(pl)$, reproving \Cref{corollary:indep-gen-set}.

\bibliographystyle{alpha}
\bibliography{bibliofiltration}	

\end{document}